\documentclass[11pt,a4paper]{article}
\usepackage[utf8]{inputenc}
\usepackage[english]{babel}
\usepackage[margin=0.8in]{geometry}
\usepackage{newunicodechar} 
\usepackage{mathrsfs} 
\usepackage{booktabs}
\usepackage{graphicx}
\usepackage[linesnumbered,ruled,vlined]{algorithm2e}
\usepackage{amsmath}
\usepackage{amssymb}
\usepackage{setspace}
\usepackage{wasysym}
\usepackage{color}
\usepackage{blindtext}
\usepackage{amssymb}
\usepackage{breqn}
\usepackage[english]{babel}
\usepackage{natbib}
\usepackage{authblk}

\usepackage{xcolor}
\usepackage{hyperref}
\hypersetup{
	colorlinks=true,
	linkcolor=blue,
	filecolor=red,      
	urlcolor=magenta,
	citecolor=blue
}
\usepackage{amsthm}
\usepackage{amsmath}
\newtheorem{proposition}{Proposition}[section]
\newtheorem{theorem}{Theorem}[section]

\newtheorem{lemma}[theorem]{Lemma}
\theoremstyle{remark}
\newtheorem{remark}{Remark}
\newtheorem{definition}{Definition}[section]
\newtheorem{assumption}{Assumption}
\theoremstyle{definition}
\usepackage[english]{babel}
\DeclareMathOperator*{\argmin}{argmin}
\DeclareMathOperator{\tv}{TV}
\DeclareMathOperator{\kl}{D}

\DeclareMathOperator{\tr}{Tr}
\newcommand{\sumi}{\sum_{i=1}^n}

\newcommand{\Sc}{\mathcal{S}}

\newcommand{\Pb}{\mathbb{P}}

\newcommand{\lb}{\left(}
\newcommand{\rb}{\right)}
\newcommand{\Eb}{\mathbb{E}}

\newcommand{\vt}[1]{V_{#1}V^T_{#1}}

\newcommand{\hvvt}[2]{\hat{V}_{#1}^{#2}(\hat{V}_{#1}^{#2})^T}
\newcommand{\vvT}[1]{V_{#1}V_{#1}^T}

\newcommand{\Rb}[1]{\mathbb{R}^{#1}}
\newcommand{\bb}{\mathbb{B}}

\begin{document}

\title{Robust Multi-Task Learning for Principal Component Analysis}
\author{Dali Liu, Haolei Weng \thanks{Corresponding Author. Email: wenghl@sustech.edu.cn }}

 \date{}
\large	
\maketitle 
 
\begin{abstract}
Principal component analysis (PCA) is a fundamental tool for learning low-dimensional structure from high-dimensional data.
When data are collected from multiple sources, the underlying task distributions may exhibit unknown degrees of similarity, with some tasks potentially arising from arbitrary distributions. We propose new multi-task PCA procedures that  exploit similarity structure across tasks to improve eigenspace estimation while remaining robust to outlier tasks.
We establish non-asymptotic convergence rates and show that the proposed procedures attain minimax optimal rates in a range of regimes.
One of the procedures builds on the matrix-depth notion of \cite{chen2018robust} and can achieve the optimal dependence of the estimation error on the proportion of outlier tasks, addressing a key challenge in robust multi-task learning. Extensive simulations and real-data analyses demonstrate the effectiveness of the proposed methods.
\end{abstract}

\section{Introduction}\label{sec:intro}

Principal component analysis (PCA) \cite{pearson1901on, hotelling1933analysis} is one of the most
fundamental tools for dimension reduction. Given observations drawn from a distribution with covariance matrix $\Sigma$, PCA aims to estimate the leading eigenspace of $\Sigma$, which captures the principal directions of variation. Its methodology and theory have been extensively studied; see \cite{jolliffe2016principal,zou2018selective,fan2021robust,greenacre2022principal} for comprehensive overviews.

In many modern applications, data are collected from multiple related sources or populations. Examples include single-cell measurements across tissue locations \cite{rao2021exploring,shang2022spatially}, patient samples collected from different hospitals \cite{sheller2020federated,li2024multisource}, and financial data across markets \cite{juneja2012common,nobi2016state}. We regard each source or population as a task and consider the resulting multi-task PCA problem. For each $i=1,2,\ldots, m$, task $i$ consists of observations drawn from a distribution with covariance matrix $\Sigma_i$. Let $V_i\in\mathbb{R}^{d\times K}$ contain, as orthonormal columns, the leading $K$ eigenvectors of $\Sigma_i$. The goal is to estimate the task-specific leading eigenspaces, i.e., column spaces of $V_i$, by borrowing information across related tasks. 

A central challenge in multi-task learning (MTL) is task heterogeneity: the degree of similarity across tasks may be unknown and vary substantially, with some tasks having distributions that can differ markedly from the majority or even be adversarially contaminated. Effective methods should exploit shared information across related tasks while remaining robust to such heterogeneity. Recent work has developed MTL procedures with provable guarantees for addressing this challenge under mixture models \cite{tian2022robust,tian2024towards}, parametric regressions \cite{tian2025learning}, and empirical risk minimization \cite{duan2023adaptive}. The current paper contributes to this line of work in the context of PCA. We focus on the following key issues:

\begin{itemize}
\item \emph{Adaptivity}: An MTL procedure should exploit strong similarity among eigenspaces when present, while avoiding excessive information sharing when the eigenspaces differ substantially across tasks.
\item \emph{Robustness}: To safeguard against data contamination, the procedure shall be robust against a fraction of outlier tasks from arbitrary distributions. 
\end{itemize}



To tackle the above two issues, we propose two multi-task PCA procedures within a two-stage framework. In the first stage, we construct a robust \emph{center} estimator that captures the shared eigenspace structure across related tasks. In the second stage, we adaptively combine the center estimator with task-specific pilot estimators through a soft-thresholding rule to obtain the final eigenspace estimator for each task. Our main contributions are summarized as follows:
\begin{itemize}
    \item[1.] We formulate the multi-task PCA problem under a Huber-type contamination model, which allows outlier tasks to follow arbitrary distributions. Two adaptive and robust MTL procedures are developed.

    \item[2.] We derive non-asymptotic bounds that quantify the effects of eigenspace similarity and task-level contamination on the estimation error, to reveal the adaptivity and robustness of the proposed procedures. In particular, the second procedure, based on the notion of matrix depth \cite{chen2018robust}, removes an unfavorable dimension factor from the contamination error term, thereby addressing an issue that also arises in other robust multi-task learning problems \cite{duan2023adaptive, tian2022robust}.

    \item[3.] We establish minimax lower bounds for the maximum error over related tasks and the average error across tasks. By comparing the upper and lower bounds, we characterize the regimes in which the proposed procedures attain the corresponding optimal rates.

    \item[4.] We conduct extensive simulation studies and real-data analyses to evaluate the empirical performance of the proposed methods. We also develop efficient tuning strategies to facilitate their practical implementation.

\end{itemize}

\subsection{Related Works}


\paragraph{Shared-subspace estimation.}
\cite{fan2019distributed,chen2022distributed,he2025distributed,li2026few} proposed communication-efficient distributed PCA algorithms that achieve estimation accuracy comparable to centralized PCA with only a small number of communication rounds. \cite{huang2021communication,shen2026dimension,zhou2026fedfask} incorporated fast sketching techniques into distributed and federated PCA, achieving a balance between statistical accuracy and computational efficiency for large-scale federated data. \cite{grammenos2020federated,froelicher2023scalable} developed privacy-preserving federated PCA methods based on differential privacy and multiparty homomorphic encryption, respectively. \cite{singh2024byzantine} proposed the Subspace-Median algorithm, which provides provable robustness to Byzantine attacks while enabling efficient estimation in federated PCA. These representative works, along with most existing distributed and federated PCA methods, primarily focus on estimating a common eigenspace under homogeneous data distributions. In contrast, the multi-task setting considered in this paper allows for heterogeneity among task-specific eigenspaces. PCA has also been studied in the context of heterogeneous multi-source data, typically under a general formulation that decomposes the underlying structure into a globally shared eigenspace and source-specific components \cite{shi2024personalized,yang2025estimating,wang2025stablepca,puchhammer2026sparse,li2026robust}. Yet, this line of research often does not account for task-level data contamination, whereas our work develops robust procedures that allow a fraction of tasks to arise from arbitrary distributions.

\paragraph{Multi-task and transfer learning for PCA.}
Several studies formulate dimension reduction directly from a multi-task perspective. Multi-task discriminant analysis accommodates heterogeneous feature spaces by decomposing each task-specific transformation into shared and task-specific components \cite{zhang2011multi}. Multi-task PCA jointly estimates task-specific principal subspaces using a Grassmannian regularizer that encourages similarity across tasks \cite{yamane2016multitask}. Another related direction is transfer learning for PCA, which leverages auxiliary datasets to improve eigenspace estimation for a designated target. Representative work considers factor recovery in high-dimensional models with weak factors \cite{he2025transpca}, knowledge transfer across heterogeneous PCA studies based on Grassmannian
barycenter \cite{li2024knowledge}, estimation for a target panel using auxiliary panels \cite{duan2024target}, and transferable latent-factor identification using cross-environment invariance and auxiliary-label information \cite{gu2026unveiling}. The above studies have made important progress in addressing eigenspace heterogeneity, but assume that all source datasets are uncontaminated. Complementing this line of work, we accommodate data contamination and develop robust learning procedures with statistical optimality guarantees.

\paragraph{Multi-task learning in other settings.}
Beyond PCA, there is a rapidly growing literature in the statistics community studying multi-task learning under various settings, including \cite{duan2023adaptive,knight2023multi,sui2026multisource,sui2025encodings,xu2025multitask,huang2025optimal,kim2026multi,javanmard2024multi,tian2025learning,guo2025statistical,tian2022robust,jing2026tuning,tian2026contaminated}, among others. These works typically address task heterogeneity based on either distance-based or representation-based similarity among task-specific model parameters, whereas our work characterizes task relatedness through spectral similarity among their principal eigenspaces.

\subsection{Notations and Organization}

We collect the notations used throughout the paper. For two sequences $\{a_n\}$ and $\{b_n\}$, $a_n=O(b_n)$ means $\sup_n |a_n/b_n|<\infty$, and $a_n=\Omega(b_n)$ if and only if $b_n=O(a_n)$. For simplicity, we use $a_n\lesssim b_n$ to denote that there exists a constant $C$ satisfying $a_n\leq Cb_n$, where $C>0$ may only depend on the constant $M$ from Definition \ref{Def:subgaussian}. For two real numbers $a$ and $b$, $a\vee b$ and $a\wedge b$ represent $\max(a,b)$ and $\min(a,b)$, respectively; $a_+=\max(0, a)$. For a vector $x\in \mathbb{R}^d$, we use $\|x\|_2$ to denote its Euclidean norm and let $\mathbb{B}({x},r)=\{y\in \mathbb{R}^d: \|y-x\|_2\leq r\}$. For a matrix $A\in \mathbb{R}^{d_1\times d_2}$, we use $\|A\|_F$ and $\|A\|$ to denote its Frobenius norm and spectral norm, respectively. Col($A$) denotes the linear space spanned by the columns of $A$. For a symmetric matrix $A$, let $\sigma_j(A)$ represent the $j$-th largest eigenvalue, and $\sigma_{\min}(A), \sigma_{\max}(A)$ denote the smallest and largest eigenvalue respectively. For a positive integer $k$, $[k]$ stands for the set $\{1,2,\ldots, k\}$, and $I_k$ is the identity matrix in $\Rb{k\times k}$. For any set $S$, $|S|$ denotes its cardinality and $S^c$ denotes its complement. We reserve $g_i$ for independent standard Gaussian random variables. For two random vectors $Y$ and $Z$, $Y\overset{d}{=}Z$ means they have the identical distribution. Let $e_j$ denote the vector whose components are all zero except that the $j$-th one equals 1. 
Let $\mathcal{O}^{d_1\times d_2}$ be the set of matrices with  orthonormal columns.
For two matrices $V_1, V_2\in \mathcal{O}^{d_1\times d_2}$, we adopt $\|V_1V_1^T-V_2V_2^T\|_F$ to measure the distance between $\operatorname{Col}(V_1)$ and $\operatorname{Col}(V_2)$. Note that $\|V_1V_1^T-V_2V_2^{T}\|_F$ is a well-defined distance between linear subspaces, and it is equivalent to the widely used sin$\Theta$ distance \cite{stewart1990matrix,vu2013minimax,cai2013sparse,yu2015useful,fan2019distributed}.

The remainder of the paper is organized as follows. Section \ref{section:setup} introduces the problem setup, and Section \ref{sec:method} presents the proposed methodology and its interpretation. Section \ref{sec:theory} establishes nonasymptotic upper bounds and the corresponding minimax lower bounds. Section \ref{sec:depth} develops the alternative method based on matrix depth. Sections \ref{sec:numerical studies} and \ref{sec:real data} present the simulation studies and real-data analyses, respectively. Section \ref{sec:discussion} concludes with a discussion of the main findings and directions for future research. The proofs are collected in Section \ref{sec:proofs}.

\section{Problem setup}\label{section:setup}

To formalize the multi-task PCA problem, consider that  there are $m$ tasks, where we observe a data matrix $X_i \in \mathbb{R}^{n_i\times d}$ for the $i$-th task. Suppose there exists an unknown subset $\Sc\subseteq [m]$ such that the rows of $X_i$ are i.i.d random vectors with zero mean and covariance matrix $\Sigma_i$ for $i\in \Sc$, while data from tasks in $\Sc^c$ can be arbitrarily distributed. For each $i\in \Sc$, let $V_i\in \mathbb{R}^{d\times K}$ denote the top $K$ eigenvectors of $\Sigma_i$\footnote{Throughout the paper, we suppress the dependency of various notations on $K$ for simplicity.}. The goal of the $i$-th task is to estimate the top $K$ eigenspace Col($V_i$), i.e., the column space of $V_i$. In the context of multi-task learning, we shall view tasks in $\Sc$ as related tasks, and we aim to leverage the potential similarity among those tasks to obtain improved learning performance. This motivates us to consider the following assumptions.
\begin{assumption}\label{assumption:data simi}
There exists a matrix $V\in \mathbb{R}^{d\times K}$ with orthonormal columns such that
\begin{equation}\label{cond:assump simi}
\max_{i\in \Sc}\|V_iV_i^T-VV^{T}\|_F\leq \delta.
\end{equation}
Moreover, 
\begin{equation}\label{cond:corrup-portion}
   \frac{|\Sc^c|}{m} \leq \epsilon.
\end{equation}
\end{assumption}

\begin{definition}\label{Def:subgaussian}
Following \cite{koltchinskii2017concentration,fan2019distributed,reiss2020nonasymptotic}, we say a random vector $Z\in \mathbb{R}^d$ is subgaussian if there exists a constant $M>0$ such that
\begin{equation*}
    \left\|u^T Z\right\|_{\psi_2} \leq M \sqrt{\mathbb{E}\left(u^T Z\right)^2}, \quad \forall u \in \mathbb{R}^d,
\end{equation*}
where the subgaussian norm of a random variable $Y\in \mathbb{R}$ is defined as
\begin{equation*}
 \|Y\|_{\psi_2}=\inf\Big\{\beta>0: \Eb\exp\lb Y^2/\beta^2\rb\leq 2\Big\}.   
\end{equation*}
\end{definition}

\begin{assumption}\label{assumption:subgaussian}

For each $i\in \Sc$, the rows of $X_i$ are i.i.d. subgaussian random vectors.    
\end{assumption}

To model potentially contaminated data, the above assumptions do not impose any distributional constraints for tasks in $\Sc^c$. These tasks can be considered as outlier tasks. Such a Huber's $\epsilon$-contamination \cite{huber1964robust} type modeling strategy has been recently adopted in other multi-task learning works \cite{konstantinov2020sample, duan2023adaptive, tian2022robust} as well. Condition \eqref{cond:assump simi} of Assumption \ref{assumption:data simi} formalizes the similarity structure by assuming the top $K$ eigenspaces in $\Sc$ are all within $\delta$-distance from a ``center" eigenspace. The parameter $\delta$ quantifies the similarity level among related tasks. The $\epsilon$ parameter in Condition \eqref{cond:corrup-portion} of Assumption \ref{assumption:data simi} represents the proportion of contaminated data at the task level. Assumption \ref{assumption:subgaussian} requires that the samples from tasks in $\Sc$ follow distributions whose tails decay as fast as a Gaussian vector - a standard distributional assumption in the PCA literature \cite{vu2013minimax, koltchinskii2017concentration, fan2019distributed, reiss2020nonasymptotic}.



We should emphasize that the set of related tasks $\Sc$ and the similarity parameter $\delta$ are both unknown. In the next two sections, we will develop a multi-task learning procedure that not only can effectively leverage the unknown similarity among related tasks (i.e. adaptive to the unknown task similarity), but also is robust against a fraction of outlier tasks (i.e. robust to data contamination). 

Finally, we mention a distributional property, introduced in \cite{fan2019distributed}, which leads to stronger theoretical guarantees for some of the results in Section \ref{sec:theory}.

\begin{definition}\label{Def: symmetric innovation}
 For a random vector $Z\in \mathbb{R}^d$ with covariance $\Sigma$, let $\Sigma=V\Lambda V^T$ be the spectral decomposition, and $Y={\Lambda}^{-1/2} V^T Z$. We say $Z$ has symmetric innovation if $Y \stackrel{d}{=}\left(I_d-2 e_j e_j^T\right) Y, \quad \forall j \in[d]$.
\end{definition}

A random vector has symmetric innovation means that flipping the sign of one component of the standardized vector does not change the distribution. All elliptical distributions with zero mean have symmetric innovations \cite{fang1990symmetric, fan2019distributed}.


\section{Methodology}\label{sec:method}

Recall that we have $m$ data matrices $\{X_i\}_{i=1}^m$. The goal is to estimate the top $K$ eigenspace Col($V_i$) of the covariance $\Sigma_i$. To fix the idea, for now consider that the parameter $\epsilon$ in Assumption \ref{assumption:data simi} equals zero, i.e., there is no outlier task. When the similarity parameter $\delta$ in Assumption \ref{assumption:data simi} is infinity, the $m$ tasks become unrelated. Then we may estimate each eigenspace Col($V_i$) separately by the empirical top $K$ eigenspace Col($\hat{V}_i$), where $\hat{V}_i\in \mathbb{R}^{d\times K}$ denotes the top $K$ eigenvectors of the empirical covariance matrix $\hat{\Sigma}_i=\frac{1}{n_i}X_i^TX_i$. On the other hand, when $\delta=0$ so that the $m$ tasks share the same top $K$ eigenspace, it is natural to estimate the $m$ eigenspaces by a ``pooled" estimator. In particular, the communication-efficient PCA algorithm proposed in \cite{fan2019distributed} computes $\bar{\Sigma}=\frac{1}{m}\sum_{i=1}^m \hat{V}_i\hat{V}_i^T$ and then uses the top $K$ eigenspace of $\bar{\Sigma}$ as the estimator. Now for the general case where $\delta\in (0,\infty)$ is unknown, it is desirable to have a procedure that can adapt to the unknown similarity. Moreover, the procedure shall also be robust against a fraction of outlier tasks when $\epsilon\neq 0$. To this end, we propose a multi-task learning procedure that interpolates between the two preceding methods, achieving both adaptivity and robustness. 

The key part of our method lies in aggregating those empirical eigenvectors $\{\hat{V}_i\}_{i=1}^m$ via penalization. Our procedure, termed MTL-PCA, is described in Algorithm \ref{alg:ardPCA}.
\begin{algorithm}[h]  \label{alg:ardPCA}
\caption{MTL-PCA}
Compute the top $K$ eigenvectors $\hat{V}_i$ of the empirical covariance matrix $\frac{1}{n_i}X_i^TX_i$, $\forall i\in [m]$. \\
Calculate $\hat{A}_i= \hat{V}_i\hat{V}_i^T$, and solve
\begin{equation}
\label{key:interp}
(\tilde{A}_1,\ldots,\tilde{A}_m, \tilde{A})=\argmin_{A_1,\ldots, A_m, A} \sum_{i=1}^m\left(\frac{1}{2}\|A_i-\hat{A}_i\|_F^2+\lambda_i\|A_i-A\|_F\right).
\end{equation}\\
Compute the top $K$ eigenvectors of $ \tilde{A}_i$, denoted by $\tilde{V}_i$, for $i\in [m]$. \\
Output: $\{\tilde{V}_i\}_{i=1}^m$
\end{algorithm}
As is clear from Algorithm \ref{alg:ardPCA}, setting $\lambda_i=0$ and $\lambda_i=\infty$ recovers the two aforementioned methods. It turns out that choosing proper positive values for $\lambda_i$ in the aggregation step \eqref{key:interp} will induce adaptivity and robustness. The following proposition sheds some lights on this point. Define the Huber loss function
\begin{equation*}
   \rho_\lambda(x) = 
\begin{cases}
\frac{x^2}{2}, & \text{if } |x| \leq \lambda, \\
\lambda \left( |x| - \frac{\lambda}{2} \right), & \text{if } |x| > \lambda.
\end{cases} \label{huber function}
\end{equation*}

\begin{proposition}\label{theorem:structure}
The solution of \eqref{key:interp} can be characterized as follows:
\begin{equation}\label{tildeAi}
\begin{aligned}
   \tilde{A}_i=\begin{cases}
     \tilde{A}, & \text{if }\|\tilde A-\hat{A}_i\|_F\leq \lambda_i,\\
     (1-\alpha)\hat{A}_i+\alpha\tilde A, & \text{if } \|\tilde A-\hat{A}_i\|_F> \lambda_i,
 \end{cases}  
\end{aligned}   
 \end{equation}
 where $\alpha=\frac{\lambda_i}{\|\hat{A}_i-\tilde A\|_F}$.  And $\tilde{A}$ is given by
 \begin{align}\label{At:huber}
 \tilde{A}=\argmin_A \sum_{i=1}^m \rho_{\lambda_i}(\|A-\hat{A}_i\|_F).
 \end{align}
\end{proposition}
\begin{figure}
    \centering    \includegraphics[width=0.6\linewidth]{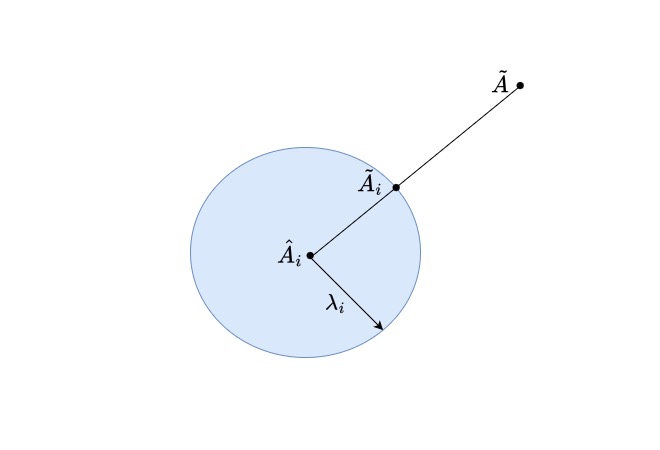}
    \caption{$\tilde{A}_i$ when $\hat{A}_i$ is not close to $\tilde{A}$.}
    \label{fig:Ahat}
\end{figure}

Recall that $\hat{A}_i=\hat{V}_i\hat{V}_i^T$ is an estimator for $V_iV_i^T$. Combining Condition \eqref{cond:assump simi} of Assumption \ref{assumption:data simi} with \eqref{At:huber} in Proposition \ref{theorem:structure}, we see that $\tilde{A}$ acts as a robust estimator of $VV^T$, thanks to the robustness property of the Huber loss. Therefore, $\tilde{A}$ offers a good summary of the ``center" information even in the presence of a small fraction of outlier tasks. Moreover, \eqref{tildeAi} in Proposition \ref{theorem:structure} shows that each $\tilde{A}_i$ is obtained by shrinking $\hat{A}_i$ towards the ``center" estimator $\tilde{A}$ in a soft-thresholding fashion: when $\hat{A}_i$ is close to $\tilde{A}$ (so that $V_iV_i^T$ is likely close to $VV^T$), $\tilde{A}_i$ is exactly equal to $\tilde{A}$; when $\hat{A}_i$ is not close enough to $\tilde{A}$, $\tilde{A}_i$ lies on the line connecting $\hat{A}_i$ and $\tilde{A}$ with the shrinkage amount $\lambda_i$ as shown in Figure \ref{fig:Ahat}. In summary, introducing a center parameter $A$ into the penalization optimization \eqref{key:interp} provides an adaptive and robust way to aggregate relevant information across tasks. \cite{duan2023adaptive} gives similar discussions for the one-dimensional Gaussian location estimation problem and further presents a general treatment of this penalization framework. We refer to \cite{duan2023adaptive} for more details.

The adaptivity and robustness of our method will be more precisely characterized in Section \ref{sec:theory}, via a non-asymptotic analysis of the convergence rate. Before that, it is important to point out another implication of Proposition \ref{theorem:structure}: The center estimator $\tilde{A}$ in \eqref{At:huber} may be replaced by other types of robust estimators in order to achieve better robustness properties. We discuss this point in detail in Section \ref{sec:depth}. 

\section{Theory}\label{sec:theory}

In this section, we develop statistical theories for our proposed procedure MTL-PCA. In Section \ref{sec:ub}, we first derive its non-asymptotic convergence rates. In Section \ref{section:minimax}, we then obtain minimax lower bounds to establish its minimax rate optimality in a wide range of regimes, and further identify the area of improvement.

\subsection{Convergence rate of MTL-PCA}\label{sec:ub}

For a symmetric matrix $\Sigma$, define its condition number and effective rank as follows:
\begin{align}
\label{kappa:r:def}
      \kappa(\Sigma)=\dfrac{\|\Sigma\|}{\sigma_K(\Sigma)-\sigma_{K+1}(\Sigma)}, \quad r(\Sigma)=\frac{\tr(\Sigma)}{\|\Sigma\|}. 
  \end{align} 
We use $\kappa_i, r_i$ to denote $\kappa(\Sigma_i),r(\Sigma_i)$ for notational simplicity. The following theorem characterizes the error rates of our procedure MTL-PCA (see Algorithm \ref{alg:ardPCA}). 
  
\begin{theorem}\label{theorem:upperbound}
Let Assumptions \ref{assumption:data simi} and \ref{assumption:subgaussian} hold. Assume $\max_{i \in \Sc}\sqrt{K}\kappa_i^2r_i/n_i\leq c_1, \min_{i\in \Sc}n_i\geq 2\log m, \epsilon\leq 1/4$, and choose
\begin{equation*}
    \lambda_i=c_2 \max_{i\in \Sc}\sqrt{\dfrac{K\kappa_i^2(r_i+\log m+ t)}{n_i}}, \quad i\in [m],~t\in \Big[(1-\log m)_+, \min_{i\in \Sc}n_i/2\Big],
\end{equation*}
where $c_1>0$ is a small constant and the constant $c_2>0$ is sufficiently large.

\begin{itemize}

\item[(i)] 

 
The following holds with probability at least $1-2\exp(-\tilde{t})-\exp(-t)$,
\begin{align}
\max_{i\in \Sc}\|\tilde{V}_i\tilde{V}_i^T-\vt{i}\|_F&\lesssim \tilde{t}\sqrt{\frac{1}{|\Sc|^2}\sum_{i\in \Sc}\frac{K\kappa_i^2 r_i}{n_i}}+\delta\wedge \max_{i\in \Sc}\sqrt{\dfrac{K\kappa_i^2(r_i+\log m+ t)}{n_i}}\nonumber \\
&+\epsilon \max_{i\in \Sc}\sqrt{\dfrac{K\kappa_i^2(r_i+\log m+ t)}{n_i}}+\frac{1}{|\Sc|}\sum_{i\in \Sc}\frac{K^{1/2}\kappa_i^2r_i}{n_i},  \label{bound:general} 
\end{align}
where $\tilde{t}\in [0, c_2c_3\sqrt{|\Sc|}]$ for some small constant $c_3>0$. If the rows of $X_i$ have symmetric innovation (see Definition \ref{Def: symmetric innovation}) for $i\in \Sc$, the last term in the upper bound \eqref{bound:general} can be dropped. 

\item[(ii)] Additionally assume that for $i\in\Sc^c$, the rows of $X_i$ follow subgaussian distributions with mean zero and covariance matrices $\Sigma_i$ satisfying $r_i\leq n_i$. We obtain that with probability at least $1-\exp(-t)$,
\begin{align}
\label{bound:outlier:task}
\max_{i\in \Sc^c}\|\tilde{V}_i\tilde{V}_i^T-\vt{i}\|_F\lesssim \max_{i\in [m]}\sqrt{\dfrac{K\kappa_i^2(r_i+\log m+ t)}{n_i}},
\end{align}
for $t\in [(1-\log m)_+, \min_{i\in [m]}n_i/2]$. As a result of \eqref{bound:general} and \eqref{bound:outlier:task}, with probability at least $1-2\exp(-\tilde{t})-2\exp(-t)$, the average error satisfies
\begin{align}
\notag \frac{1}{m}\sum_{i\in [m]}\|\tilde{V}_i\tilde{V}_i^T- \vt{i}\|^2_F&\lesssim \frac{\tilde{t}^2}{|\Sc|^2}\sum_{i\in \Sc}\frac{K\kappa_i^2 r_i}{n_i}+\delta^2 \wedge \max_{i\in \Sc}\dfrac{K\kappa_i^2(r_i+\log m+ t)}{n_i}\nonumber \\
&+\epsilon \max_{i\in [m]}\dfrac{K\kappa_i^2(r_i+\log m+ t)}{n_i}+\bigg(\frac{1}{|\Sc|}\sum_{i\in \Sc}\frac{K^{1/2}\kappa_i^2r_i}{n_i}\bigg)^2. \label{bound:sum}
\end{align}
If the rows of $X_i$ have symmetric innovation for $i\in \Sc$, the last term in the above bound can be dropped. 

\end{itemize}
\end{theorem}

The proof of Theorem \ref{theorem:upperbound} can be found in Section \ref{subsec:thm1}. We make several remarks about Theorem \ref{theorem:upperbound}. 
\begin{itemize}
\item Part (i) provides the upper bound for maximum error over related tasks. Since no distributional assumptions are made for the outlier tasks, we do not expect to have error control on $\Sc^c$. The upper bound consists of four parts. Consider the case $\kappa_i=\kappa,r_i=r,n_i=n, i\in \Sc$ for a concise interpretation. The rate $\sqrt{K\kappa^2 r/(nm)}$ in the first part is the oracle rate achievable when all tasks are the same \cite{fan2019distributed}. The second part characterizes the error induced by the task heterogeneity. It first increases with the similarity parameter $\delta$ (i.e. degree of dissimilarity among related tasks) and then flattens out, reaching the minimax rate of single-task learning $\sqrt{K\kappa^2r/n}$ \cite{vu2013minimax,cai2013sparse} (up to $\log m$). The $\log m$ term is unavoidable due to the simultaneous control of error for related tasks. The third part quantifies the impact of outlier tasks from which data can be arbitrarily contaminated. Its error increases with the contamination level $\epsilon$, but never exceeds the single-task learning rate. This shows that our method is robust against a fraction of outlier tasks from arbitrary distributions. The last part plays as a bias term arising from aggregating the empirical eigenvectors across tasks. The bias is negligible when the sample size is large enough, and it becomes zero for distributions with symmetric innovation. See \cite{fan2019distributed} for more details on this point.

\item When $\delta=0,\epsilon=0$ (i.e. all tasks share the same top $K$ eigenspaces), our convergence rate in \eqref{bound:general} is reduced to that of distributed PCA in \cite{fan2019distributed} (see Theorems 4 and 6 therein). In general, the convergence rate varies smoothly with $(\delta, \epsilon)$, and is never worse than the single-task learning rate (up to $\log m$). Moreover, the choice of the tuning parameters $\{\lambda_i\}_{i \in [m]}$ does not depend on $(\delta, \epsilon)$. These results demonstrate the adaptation of our procedure to the unknown task relatedness. 

\item When the data on $\Sc^c$ follow subgaussian distributions, Part (ii) shows that our procedure has controlled error on $\Sc^c$ as well. The convergence rate is comparable to the single-task learning rate. This is essentially the best achievable rate, because tasks on $\Sc^c$ do not share any similarity with others. The average error is a direct consequence of maximum errors over $\Sc$ and $\Sc^c$. Its four parts have similar interpretations as the ones in \eqref{bound:general}.

\item In the next section, we complement these upper bounds by the corresponding minimax lower bounds. The upper and lower bounds together reveal that our procedure MTL-PCA attains optimal rate for the average error, but can be suboptimal for the maximum error over related tasks. We then address the suboptimality issue in Section \ref{sec:depth}.

\end{itemize}

\subsection{Minimax lower bounds}\label{section:minimax}

To derive the minimax lower bounds for multi-task learning in PCA, we first specify the statistical model under consideration. We observe the data $\mathcal{X}=\{X_i\}_{i=1}^m$ where $X_i\in  \mathbb{R}^{n\times d}$ is the data matrix from the $i$-th task. For a given triplet $(\delta, \epsilon, \sigma)$ with $\delta \geq 0, 0\leq \epsilon\leq 1, \sigma>0$, a model $\mathcal{M}(\delta,\epsilon,\sigma)$ includes all distributions that satisfy:
\begin{enumerate}
\item There exists a subset $\Sc \subseteq [m]$ such that for all $i\in \Sc$, the $X_i$'s are independent and the rows of $X_i$ are i.i.d drawn from $\mathcal{N}(0,\Sigma_i)$ with $\Sigma_i=\sigma V_iV_i^T+I_d$, where $V_i\in \mathbb{R}^{d\times K}$ has orthonormal columns. Moreover, 
\[
\max_{i,j\in \Sc}\|V_iV_i^T-V_jV_j^T\|_F\leq \delta.
\]
\item The data $\{X_i\}_{i\in \Sc^c}$ follow an arbitrary distribution, independent from $\{X_i\}_{i\in \Sc}$. And the cardinality of $\Sc^c$ satisfies
\[
|\Sc^c|\leq \epsilon m.
\]
\end{enumerate}

The model $\mathcal{M}(\delta,\epsilon,\sigma)$ imposes Gaussian distributions with spiked covariance for related tasks. 
The spiked covariance model was introduced in \cite{johnstone2001distribution}. 
The minimax rates of eigenspace estimation of the spiked matrices in a single task have been studied extensively in the literature, such as \cite{birnbaum2013minimax,ma2013sparse,cai2013sparse,cai2015optimal}.
It is direct to confirm that the model considered in Part (i) of Theorem \ref{theorem:upperbound} contains $\mathcal{M}(\delta,\epsilon,\sigma)$ as a submodel. We thus develop the minimax lower bound for the maximum error under $\mathcal{M}(\delta,\epsilon,\sigma)$.

\begin{theorem}[Minimax lower bound for the maximum error]\label{theorem:minimax1}
Assume  $\epsilon\leq \frac{1}{2}$, $1\leq K\leq \frac{d}{C_1}$ and $\frac{\sigma+1}{\sigma^2}\frac{d}{n}\leq C_2$. It holds that
\begin{align*}
\inf_{\{\hat{V}_i\}_{i\in \Sc}} \sup_{\Pb \in \mathcal{M}(\delta, \epsilon,\sigma)}
\Pb \Bigg( & \max_{i\in \Sc}\|\hat{V}_i\hat{V}_i^T-\vt{i}\|_F \\
&\geq C_3\bigg(\sqrt{\frac{(\sigma+1)Kd}{\sigma^2nm}}+\delta\wedge \sqrt{\frac{(\sigma+1)Kd}{\sigma^2n}}   +\epsilon\sqrt{\frac{\sigma+1}{\sigma^2n}} \bigg)\Bigg)\geq C_4,
\end{align*}
where $\{C_i\}_{i=1}^4$ are positive universal constants.
\end{theorem}

Define the model $\mathcal{A}(\delta,\epsilon,\sigma)$ as a subset of $\mathcal{M}(\delta,\epsilon,\sigma)$, which further requires that for all $i\in \Sc^c$, the $X_i$'s are independent and the rows of $X_i$ are i.i.d drawn from $\mathcal{N}(0,\Sigma_i)$ with $\Sigma_i=\sigma V_iV_i^T+I_d$, where $V_i\in \mathbb{R}^{d\times K}$ has orthonormal columns. Part (ii) of Theorem \ref{theorem:upperbound} is developed under a model containing $\mathcal{A}(\delta,\epsilon,\sigma)$ as a subset. We then derive the minimax lower bound for the average error under $\mathcal{A}(\delta,\epsilon,\sigma)$.

\begin{theorem}[Minimax lower bound for the average error]\label{theorem:minimax2}
Under the same conditions of Theorem \ref{theorem:minimax1}, the following bound holds 
\begin{align*}
    \inf_{\{\hat{V}_i\}_{i=1}^m} \sup_{\Pb \in \mathcal{A}(\delta, \epsilon,\sigma)} \Pb \Bigg(&\frac{1}{m}\sum_{i=1}^m \|\hat{V}_i\hat{V}_i^T-\vt{i}\|_F^2 \\
    &\geq C_5  \lb\frac{(\sigma+1)Kd}{\sigma^2nm}+\delta^2\wedge\frac{(\sigma+1)Kd}{\sigma^2n}+\frac{(\sigma+1)\epsilon Kd}{\sigma^2n}\rb\Bigg)\geq C_6,
\end{align*}
where $C_5, C_6>0$ are universal constants.
\end{theorem}
The proof of Theorems \ref{theorem:minimax1} and \ref{theorem:minimax2} can be found in Section \ref{proof:them:2and3}.

\begin{remark}
\label{rem:ave}
Compare the lower bound in Theorem \ref{theorem:minimax2} with the upper bound \eqref{bound:sum}. Under spiked covariance model, the condition number and effective rank take values $\kappa_i=\frac{\sigma+1}{\sigma}, r_i=\frac{d+K\sigma}{\sigma+1}$. Hence, the average error bound \eqref{bound:sum} reads (ignore the $\log m$ term for now)\footnote{We use the symbol $\lesssim_p$ to represent high-probability upper bounds.}
\begin{align}
\notag \frac{1}{m}\sum_{i\in [m]}\|\tilde{V}_i\tilde{V}_i^T- \vt{i}\|^2_F&\lesssim_p   \frac{(\sigma+1)K(d+K\sigma)}{\sigma^2 nm}+\delta^2\wedge   \frac{(\sigma+1)K(d+K\sigma)}{\sigma^2 n} \nonumber \\
& + \frac{(\sigma+1)\epsilon K(d+K\sigma)}{\sigma^2 n}+\frac{(\sigma+1)^2K(d+K\sigma)^2}{\sigma^4 n^2}. \label{new:form:upper}
\end{align}
When $K\sigma \lesssim d$ and the last term above is negligible, the upper bound matches with the lower bound in Theorem \ref{theorem:minimax2}. Note that the last term in \eqref{new:form:upper} can be dropped if the data distributions have symmetric innovation (e.g. Gaussian distributions) or it can be absorbed into the first term as long as $n=\Omega((\sigma+1)\sigma^{-2}md)$. We thus conclude that our procedure MTL-PCA achieves minimax optimal rate (up to logarithmic term) for the average error under mild conditions. 
\end{remark}

\begin{remark}
Compare the lower bound in Theorem \ref{theorem:minimax1} with the upper bound \eqref{bound:general}. In light of the spiked covariance model with $K\sigma \lesssim d$, we simplify \eqref{bound:general} as (ignore the $\log m$ term)
\begin{align}
\notag \max_{i\in \Sc}\|\tilde{V}_i\tilde{V}_i^T- \vt{i}\|_F&\lesssim_p  \sqrt{\frac{(\sigma+1)Kd}{\sigma^2 nm}}+\delta \wedge   \sqrt{\frac{(\sigma+1)Kd}{\sigma^2 n}} \nonumber \\
& + \epsilon\sqrt{\frac{(\sigma+1)Kd}{\sigma^2 n}}+\frac{(\sigma+1)K^{1/2}d}{\sigma^2 n}. \label{new:form:maximu:err}
\end{align}
Following similar arguments as in Remark \ref{rem:ave}, the last term in \eqref{new:form:maximu:err} is negligible under mild conditions. Even so, we see that the upper bound differs from the lower bound in Theorem \ref{theorem:minimax1} by a dimensional factor $\sqrt{Kd}$ in the term involving $\epsilon$. This is an error term related to data contamination. Therefore, for maximum error over related tasks, MTL-PCA might not be minimax optimal, and this happens when the $\epsilon$-related term becomes the dominating term in the convergence rate (i.e. there is sufficient amount of data contamination). In the next section, we will remove the dimensional factor by developing an alternative multi-task learning procedure. 

\end{remark}

\section{Matrix depth based approach}\label{sec:depth}

As discussed in Section \ref{section:minimax}, MTL-PCA attains minimax optimal convergence rate for the average error. However, for the maximum error over related tasks, its convergence rate has a suboptimal dependence on the contamination level $\epsilon$ by a dimensional factor. In this section, we explore a depth-based approach to remove this dimensional factor, thereby achieving possible optimal robustness against a fraction of outlier tasks from arbitrary distributions. The underlying idea is straightforward. Recall that Proposition \ref{theorem:structure} decomposes the aggregation step \eqref{key:interp} of MTL-PCA into two parts: (i) deriving the robust center estimator $\tilde{A}$ by minimizing the Huber loss; (ii) shrinking towards $\tilde{A}$ via soft-thresholding to obtain the individual estimators. Our theoretical analysis in Section \ref{subsec:thm1} reveals that the suboptimality of MTL-PCA arises from $\tilde{A}$. In other words, the Huber loss function does not induce optimal robustness (in terms of convergence rate) against contamination in the context of multi-task learning. Thusly motivated, we seek to deliver a more robust center estimator by optimizing a notion of matrix depth function instead.


\subsection{Matrix depth}
\label{matrix:depth:sec}

Depth function is a classical concept in robust statistics \cite{tukey1975mathematics,liu1990notion,donoho1992breakdown,liu1999multivariate,zuo2000general,vardi2000multivariate,zuo2003projection,zuo2021general}. Motivated by Tukey’s depth function for location parameters \cite{tukey1975mathematics}, \cite{chen2018robust} introduced a new concept called matrix depth that achieves minimax optimal rates for robust covariance matrix estimation under Huber's $\epsilon$-contamination model. We will extend their matrix depth function to the multi-task learning setting. To start with, the matrix depth of a positive semidefinite matrix $\Gamma \in \mathbb{R}^{d\times d}$ with respect to a distribution $\Pb$ is defined in \cite{chen2018robust} as
\begin{equation*}
    \mathcal{D}(\Gamma, \mathbb{P}) = \inf_{\|u\|_2 = 1} \min \Big\{ \mathbb{P} \left( |u^T Y|^2 \leq u^T \Gamma u \right), \mathbb{P} \left( |u^T Y|^2 \geq u^T \Gamma u \right) \Big\},
\end{equation*}
where $Y\in \mathbb{R}^d \sim \Pb$ and $\|\cdot\|_2$ is the Euclidean norm. Let $\beta$ be the median of the chi-squared distribution $\chi^2(1)$. \cite{chen2018robust} showed that $\beta\Sigma$ is the maximizer of $\mathcal{D}(\Gamma, \mathbb{P})$ when the distribution $\Pb$ is $\mathcal{N}(0,\Sigma)$. That is, the population covariance matrix is the deepest point (up to a scaling) with respect to a zero-mean Gaussian distribution. Given i.i.d samples $Y_1,\ldots, Y_n$ from the contamination model $(1-\epsilon)\mathcal{N}(0,\Sigma)+\epsilon Q$ with an arbitrary distribution $Q$, the robust estimator of $\Sigma$ is defined as $\hat{\Sigma}=\hat{\Gamma}/\beta$, where $\hat{\Gamma}$ is obtained by maximizing the empirical matrix depth, 
\begin{equation*}
\hat{\Gamma}\in  \arg\max_{\Gamma \succeq 0} \mathcal{D}(\Gamma, \mathbb{P}_n),
\end{equation*}
with $\Pb_n$ denoting the empirical distribution. \cite{chen2018robust} proved that $\hat{\Sigma}$ achieves the minimax optimal convergence rate under the $\epsilon$-contamination model. Similar ideas and theories have been generalized for estimating structured covariance matrix (or scatter matrix) of elliptical distributions in the same work.

Now for our multi-task learning setting with $m$ datasets $\{X_i\in \mathbb{R}^{n_i\times d}\}_{i=1}^m$, there exist two noticeable differences: (a) the contamination occurs at the task level instead of individual data point level; (b) the uncontaminated data $\{X_i\}_{i\in \Sc}$ may not follow the same distribution. To address such differences, we propose a modified depth function. For each $i\in [m]$, we split the data $X_i$ into $B_i$ roughly equal-sized groups, denoted by
\begin{align*}
&X_i^{(1)}\in \mathbb{R}^{n_{i1}\times d},~X_i^{(2)}\in \mathbb{R}^{n_{i2}\times d},\ldots, ~X_i^{(B_i)}\in \mathbb{R}^{n_{iB_i}\times d}, \\
&n_{i1}=\Big(n_i-(B_i-1)\Big\lfloor\frac{n_i}{B_i} \Big\rfloor\Big), n_{i2}=\Big\lfloor\frac{n_i}{B_i} \Big\rfloor,\ldots, n_{iB_i}=\Big\lfloor\frac{n_i}{B_i} \Big\rfloor.
\end{align*}
The new depth function is defined as 
\begin{align}
   \mathcal{D}\left(\Gamma, \{X_i\}_{i=1}^m \right) =
\inf_{\|u\|_2=1} \min &\Bigg\{
\frac{1}{m} \sum_{i=1}^{m} \frac{1}{B_i}\sum_{j=1}^{B_i}\mathbb{I} \Big\{ \frac{u^T(X_i^{(j)})^TX_i^{(j)}u}{n_{ij}}\leq u^{T} \Gamma u \Big\}, \; \nonumber \\
&
\frac{1}{m} \sum_{i=1}^{m} \frac{1}{B_i}\sum_{j=1}^{B_i}\mathbb{I} \Big\{ \frac{u^T(X_i^{(j)})^TX_i^{(j)}u}{n_{ij}}> u^{T} \Gamma u \Big\}
\Bigg\}. \label{new:depth:def}
\end{align}
When $B_i=n_i=n, i\in [m]$, our depth function is reduced to the one in \cite{chen2018robust} as described in the last paragraph. In general, instead of utilizing each data point separately, our depth function takes the sample covariance matrix $\frac{1}{n_i}(X_i^{(j)})^TX^{(j)}_i$ from each group $X_i^{(j)}$ as input. As illustrated in the next section, without grouping (i.e. $B_i=n_i$), the contamination-related error term  would be inflated by an order of $\sqrt{n_i}$, while proper grouping helps achieve the correct order $n_{i}^{-\frac{1}{2}}$ in terms of the sample size. Based on the new depth function \eqref{new:depth:def}, we define the depth-based estimator as 
\begin{align}
\label{cov:center}
\hat{\Gamma}\in  \arg\max_{\Gamma \succeq 0}  \mathcal{D}\left(\Gamma, \{X_i\}_{i=1}^m \right).
\end{align}
The estimator $\hat{\Gamma}$ in \eqref{cov:center} (up to a scaling) summarizes information across $\{\Sigma_i\}_{i\in\Sc}$. Since we are only interested in eigenspaces, calibration like the one in \cite{chen2018robust} is not needed. Let $\hat{V}^D\in \mathbb{R}^{d\times K}$ denote the top $K$ eigenvectors of $\hat{\Gamma}$. Our new center estimator is 
\begin{align}
\label{new:center:robust}
\tilde{A}=\hat{V}^D(\hat{V}^D)^T.
\end{align}

\subsection{Depth-based multi-task learning}
\label{dbml:sec:main}
We introduced a depth-based center estimator $\tilde{A}$ in Section \ref{matrix:depth:sec}. The idea is to replace the Huber-type center estimator used in MTL-PCA by the new one and keep the remaining steps the same. The new estimator $\tilde{A}$ in \eqref{new:center:robust} thus leads to a depth-based multi-task learning procedure, termed DBMTL-PCA, as described in Algorithm \ref{alg:depthPCA}.

\begin{algorithm}[h]\label{alg:depthPCA}
\caption{DBMTL-PCA}
Compute $\tilde{A}=\hat{V}^D(\hat{V}^D)^T$, where $\hat{V}^D$ is the top $K$ eigenvectors of $\hat{\Gamma}$ from \eqref{cov:center}. \\
Calculate $\hat{A}_i= \hat{V}_i\hat{V}_i^T, i\in[m]$, where $\hat{V}_i$ is the top $K$ eigenvectors of the empirical covariance matrix $\frac{1}{n_i}X_i^TX_i$. \\
Compute $ \tilde{A}_i, i\in [m]$, according to \eqref{tildeAi}.\\
Calculate the top $K$ eigenvectors of $ \tilde{A}_i$, denoted by $\tilde{V}_i$, for $i\in [m]$. \\
Output: $\{\tilde{V}_i\}_{i=1}^m$
\end{algorithm}






The following theorem provides the error bound for the depth-based procedure DBMTL-PCA. To streamline the presentation, we consider the case $C_1n\leq \min_{i\in S}n_i\leq \max_{i\in S}n_i\leq C_1^{-1}n, C_2B\leq \min_{i\in S}B_i\leq \max_{i\in S}B_i\leq C_2^{-1}B$, and $\max_{i\in S}\kappa(\Sigma_i) \leq \kappa, \max_{i\in S}r(\Sigma_i) \leq r$. The notation $\kappa(\Sigma)$ and $r(\Sigma)$ was introduced in \eqref{kappa:r:def}.

\begin{theorem}\label{theorem:depth space}
Suppose Assumption \ref{assumption:data simi} holds, and the rows of $X_i$ are i.i.d. subgaussian and elliptical for $i\in S$. Choose
\begin{equation*}
    \lambda_i=C_3 \sqrt{\dfrac{K\kappa^2(r+\log m+ t)}{n}}, \quad i\in [m],~t\in \Big[(1-\log m)_+, C_1n/2\Big].
\end{equation*}
Additionally, we assume $\epsilon\leq \frac{1}{4}, C_1n\geq \max(r,2\log m), \frac{n}{B}\geq C_4$, and there exists $\Sigma \succ 0$ with top $K$ eigenvectors $V$\footnote{Recall that $V\in \mathbb{R}^{d\times K}$ is from Assumption \ref{assumption:data simi}.} such that
\begin{align}
\label{key:cov:heter}
\sqrt{\frac{n}{B}}\max_{i\in \Sc}\frac{\|\Sigma-\Sigma_i\|}{\sigma_{\min}(\Sigma_i)\wedge \sigma_{\min}(\Sigma)}+\sqrt{\frac{d+\tilde{t}}{mB}}\leq C_5. 
\end{align}
Here, $C_3,C_4>0$ are large constants and $C_5$ is sufficiently small. For the outputs $\{\tilde{V}_i\}_{i=1}^m$ from Algorithm \ref{alg:depthPCA}, it holds with probability at least $1-\exp(-t)-\exp(-\tilde{t})$,
\begin{align}\label{thm6.23}
\max_{i\in \Sc}\|\tilde{V}_i\tilde{V}_i^T-\vt{i}\|_F \lesssim &\min\Bigg\{\sqrt{\dfrac{K\kappa^2(r+\log m+ t)}{n}}, ~~\delta+\nonumber \\
&\sqrt{K}\kappa(\Sigma)\Bigg[\max_{i\in S}\frac{\|\Sigma-\Sigma_i\|}{\sigma_{min}(\Sigma_i)\wedge \sigma_{min}(\Sigma)}+\sqrt{\frac{B}{n}}\epsilon+\frac{B}{n}+\sqrt{\frac{d+\tilde{t}}{mn}}\Bigg]\Bigg\}.
\end{align}
When the rows of $\{X_i\}_{i\in S}$ are Gaussian, the term $\frac{B}{n}$ can be dropped from the above bound.  
\end{theorem}


The proof of Theorem \ref{theorem:depth space} is presented in Section \ref{proof:depth:method:bound}. 

\vspace{0.1cm}

\begin{remark}
The assumption of elliptical distribution provides spherical symmetry, making the depth function \eqref{new:depth:def} a meaningful notion with respect to the covariance matrix. The same assumption was made in \cite{chen2018robust}. The key assumption \eqref{key:cov:heter} imposes a restriction on the degree of heterogeneity of covariance among related tasks. To make it as weak as possible, we might choose $B\propto n$ (i.e. no grouping is performed). However, this would make the $B$-related error terms in \eqref{thm6.23} the largest. A favorable choice is 
\begin{align}
\label{group:choice:best}
B \propto \bigg(\max_{i\in S}\frac{\sqrt{n}\|\Sigma-\Sigma_i\|}{\sigma_{min}(\Sigma_i)\wedge \sigma_{min}(\Sigma)}\bigg)^2 \vee 1.
\end{align}
With this choice, it is direct to verify that the bound \eqref{thm6.23} holds for $B=1$ and  \eqref{key:cov:heter} can be replaced with the mild condition:
\begin{align}
\label{hetero:mild:con}
\sqrt{\frac{d+\tilde{t}}{m}}\Bigg(\min_{i\in S}\frac{\sigma_{min}(\Sigma_i)\wedge \sigma_{min}(\Sigma)}{\sqrt{n}\|\Sigma-\Sigma_i\|}\wedge 1\Bigg)\leq C_6,~\max_{i\in S}\frac{\|\Sigma-\Sigma_i\|}{\sigma_{min}(\Sigma_i)\wedge \sigma_{min}(\Sigma)}\leq C_7,
\end{align}
for sufficiently small constants $C_6,C_7>0$. This result reveals the importance of grouping in the construction of the depth function, a phenomenon not observed in the single-task learning setting.
\end{remark}

\vspace{0.05cm}

\begin{remark}
Referring to the upper bound \eqref{thm6.23}, we first observe that the convergence rate is no larger than the single-task learning rate $\sqrt{K\kappa^2r/n}$ (up to a logarithmic term). Moreover, compared to the upper bound for MTL-PCA (see Theorem \ref{theorem:upperbound}), the $\epsilon$-related term of \eqref{thm6.23} does not involve the dimensional factor anymore. To evaluate the tightness of the bound, we compare \eqref{thm6.23} with the minimax lower bound in Theorem \ref{theorem:minimax1} under the model $\mathcal{M}(\delta,\epsilon,\sigma)$ of Section \ref{section:minimax}. With the choice \eqref{group:choice:best}, \eqref{thm6.23} becomes
\begin{align*}
\max_{i\in \Sc}\|\tilde{V}_i\tilde{V}_i^T-\vt{i}\|_F \lesssim \frac{(\sigma+1)\sqrt{K}}{\sigma}\cdot \Bigg(\bigg[\sigma \delta +\sqrt{\frac{1}{n}}\epsilon+\sqrt{\frac{d}{mn}}\bigg] \wedge \sqrt{\frac{d+K\sigma+\log m}{n(\sigma+1)}}\Bigg).
\end{align*}
It matches (up to a $\log m$ factor) with the lower bound from Theorem \ref{theorem:minimax1} in the regime where $\max\{K,\sigma\}=O(1)$. The condition $K=O(1)$ is mild, since it is not uncommon to use only a small number of principal components in the PCA analysis. The other condition $\sigma=O(1)$ requires bounded eigengap, suitable for scenarios where the signal-to-noise ratio is not very high. This condition arises partly because the depth function has not yielded a dimension-free bound (involving the effective rank instead of the ambient dimension) for covariance matrix estimation -- the same issue occurs in \cite{chen2018robust}. We leave it for interesting future research to develop optimal multi-task learning procedures in terms of maximum error over related tasks, in broader regimes.

\end{remark}


\section{Simulations}\label{sec:numerical studies}

In this section, we conduct simulations to evaluate the empirical performance of the proposed MTL-PCA and DBMTL-PCA methods. 
Our experiments are designed to examine how the performance of the estimators is affected by the similarity among related tasks and the presence of outlier tasks. 
To run MTL-PCA, we apply the IRLS \cite{holland1977robust} algorithm  to solve the convex optimization problem \eqref{key:interp}. 
In the implementation of DBMTL-PCA, we adopt the iterative algorithm from \cite{chen2018robust}, whose code is available at \href{https://github.com/ChenMengjie/DepthDescent}{https://github.com/ChenMengjie/DepthDescent}, to compute the depth-based estimator $\hat{\Gamma}$ in \eqref{cov:center}.

The simulation study is divided into three parts.
In the first two parts, we assess the performance of MTL-PCA under two complementary settings.
In the first setting, the proportion of outlier tasks is fixed while the similarity among tasks is varied.
In the second setting, the task similarity is fixed while the proportion of outlier tasks is varied. These two settings allow us to evaluate the behavior of the method under different regimes of heterogeneity and contamination.
In the third part, we evaluate the performance of DBMTL-PCA under similar settings.

\subsection{Part I: Varying Data Similarity with a Fixed Outlier Proportion}
\label{subsec:sim part 1}

We generate $m=25$ tasks with data matrices
$X_i\in\mathbb{R}^{n_i\times d}$, where $n_i$ is the sample size
for task $i$. For simplicity, we set $n_i=n=1000$ for all tasks.
The population covariance matrix for each task follows the spiked model
\begin{equation}\label{spiked cov}
   \Sigma_i = \sigma V_i V_i^T + I_d, \quad i \in [m],
\end{equation}
where $V_i\in\mathcal{O}^{d\times K}$.
We set $d=100$, $K=3$, and $\sigma=10$, with $|\Sc|=20$ related
tasks and $|\Sc^c|=5$ outlier tasks. We vary $\delta$ to examine how
performance changes as the related tasks become more heterogeneous.

\begin{itemize}
\item \textbf{Eigenspace generation.}
We first generate a matrix $M\in\Rb{d\times K}$ with i.i.d.
standard Gaussian entries and apply QR decomposition to obtain a
common center $V_0\in\mathcal{O}^{d\times K}$.
For each related task, we independently generate a standard Gaussian
matrix $M_i\in\Rb{d\times K}$. We apply QR decomposition to $(I_d-V_0V_0^T)M_i$ and obtain $U_i\in\mathcal{O}^{d\times K}$ whose columns are
orthogonal to those of $V_0$.
For a given $\delta$, we set
$\theta=\arcsin (\frac{\delta}{\sqrt{2K}})$ and let
\begin{equation*}
    V_i=\sin (\theta)U_i+ \cos(\theta) V_0.
\end{equation*}
This construction yields $\|\vvT{i}{}-\vvT{0}{}\|_F=\delta$.
Thus, smaller values of $\delta$ correspond to greater similarity
among the related tasks.
For the outlier tasks $i\in\mathcal{S}^c$, we independently generate
$\tilde{V}_0$ in the same way as $V_0$ and set
$V_iV_i^T=\tilde{V}_0\tilde{V}_0^T$.
The outlier tasks therefore share a common eigenspace whose center
is generated independently of the related-task center.

\item \textbf{Data distribution.}
For each related task $i\in\Sc$, observations are independently
generated from $\mathcal{N}(0,\Sigma_i)$.
For outlier tasks $i\in\Sc^c$, we consider Gaussian observations
from $\mathcal{N}(0,\Sigma_i)$ and heavy-tailed observations from a
scaled multivariate $t$-distribution with $\nu=3$ degrees of freedom.
For the latter setting, let $Z\sim N(0,\Sigma_i)$ and
$W\sim\chi^2(\nu)$ be independent. We generate each observation as
\begin{equation*}
    \xi = \frac{Z}{\sqrt{W/(\nu-2)}}.
\end{equation*}
The scaling ensures that $\mathrm{Cov}(\xi)=\Sigma_i$.
Thus, the two settings have the same population covariance structure
but differ in the tail behavior of the outlier observations.

\item \textbf{Methods and evaluation.}
We compare MTL-PCA with four competing methods:
\begin{enumerate}
\item Selected Grassmannian barycenter PCA (sgbPCA):
the transfer learning method introduced in Algorithm~2 of
\cite{li2024knowledge}.
We treat each task in turn as the target and the remaining tasks as
candidate sources. Let
$\widetilde P_j=\widetilde V_j\widetilde V_j^\top$
denote the individual PCA projector for task $j$.
Starting from $P_i^{(0)}=\widetilde P_i$, we iterate
\begin{align*}
\mathcal I_i^{(t)}
&=
\left\{
j \ne i :
\operatorname{tr}\!\left(
P_i^{(t-1)} \widetilde P_j
\right) \ge \tau
\right\}, \\
P_i^{(t)}
&=
\Pi_K\!\Bigg(
n_i \widetilde P_i
+
\sum_{j \in \mathcal I_i^{(t)}} n_j \widetilde P_j
\Bigg),
\qquad t=1,\ldots,T,
\end{align*}
where $\Pi_K(A)$ denotes the orthogonal projector onto the leading
$K$-dimensional eigenspace of $A$, and $n_j$ is the sample size of
task $j$. The threshold $\tau$ controls source selection, and $T$
is the number of iterations. The final estimator is
$\widehat V_i\widehat V_i^\top=P_i^{(T)}$.

We parameterize the threshold as $\tau=\rho K$ and select $\rho$
by four-fold cross-validation over $\{0.80,0.81,\ldots,1.00\}$.
We examined $T\in\{1,\ldots,10\}$, tuning $\rho$ separately for each
value of $T$. Performance stabilized at $T=3$, with negligible
changes for larger values, so we fix $T=3$ throughout the experiments.

\item Single-task PCA (sPCA):
PCA applied separately to each task.

\item Distributed PCA (dPCA):
the distributed algorithm introduced by \cite{fan2019distributed}.
For each task $i$, let $\hat{V}_i\in\mathcal{O}^{d\times K}$ contain
the top $K$ eigenvectors of the sample covariance matrix
$n_i^{-1}X_i^T X_i$ as columns. The algorithm computes the top $K$
eigenvectors of $\sum_{i=1}^m\hat{V}_i\hat{V}_i^T$, yielding a
common eigenspace estimator.

\item Pooled PCA (pPCA):
PCA applied to the pooled observations from all tasks, yielding a
common eigenspace estimator.
\end{enumerate}

Our proposed MTL-PCA can be implemented in a distributed setting,
with a central server communicating with local data holders through
summary statistics. We consider two tuning strategies: one based
on a single round of communication, denoted by mPCA1, and the other
based on multiple rounds of communication, denoted by mPCA2.
Details are provided in Section~\ref{subsec: tuning stra}.

For each method, let $\hat{V}_i$ denote the eigenspace estimator for
task $i$. We evaluate performance using the maximum projector error
over related tasks (ME) and the root mean squared projector error
across all tasks (AE):
\begin{equation*}
\text{ME} = \max_{i \in \mathcal{S}} \|\hat{V}_i \hat{V}_i^T - V_i V_i^T\|_F,~~ \text{AE} = \sqrt{\frac{1}{m} \sum_{i=1}^m \|\hat{V}_i \hat{V}_i^T - V_i V_i^T\|_F^2}.
\end{equation*}
Each experimental configuration is repeated 100 times.
\end{itemize}

Figure~\ref{fig:sim1_gaus} presents the results for Gaussian outlier
tasks. When $\delta$ is small, all methods that share information
across tasks achieve substantially lower ME than sPCA.
In particular, sgbPCA, mPCA1, and mPCA2 have similar ME and outperform
dPCA and pPCA. As $\delta$ increases, the ME of dPCA and pPCA rises
steadily and eventually exceeds that of sPCA, illustrating the
limitations of non-adaptive average aggregation for increasingly heterogeneous tasks. The ME of sgbPCA follows a nonmonotonic pattern: it exceeds that of sPCA at intermediate values of $\delta$, peaks near $\delta=0.35$,
and then decreases toward the sPCA level.
This pattern suggests that source selection is most challenging at
intermediate levels of heterogeneity. By comparison, mPCA1 and mPCA2 have nearly identical ME, remain below sPCA throughout the displayed range, and approach the single-task error level smoothly as $\delta$ increases.

The AE of dPCA and pPCA is substantially larger than that of the other methods even when
$\delta$ is small, since the eigenspace estimators from the two methods are highly
biased towards related tasks. In contrast, sgbPCA has the lowest AE for small
$\delta$. Its advantage under this criterion can be attributed to its
ability to share information within both the related-task group and
the outlier-task group (with each task treated as the target in turn), as the latter also has a shared eigenspace in this experiment. At intermediate values of $\delta$, the AE of
sgbPCA rises above that of the two MTL-PCA variants, before becoming
slightly lower again at larger values of $\delta$.
Both mPCA1 and mPCA2 remain below sPCA and show a smooth increase in
AE as task heterogeneity grows.

\begin{figure}[h]
    \centering
    \includegraphics[width=\linewidth]{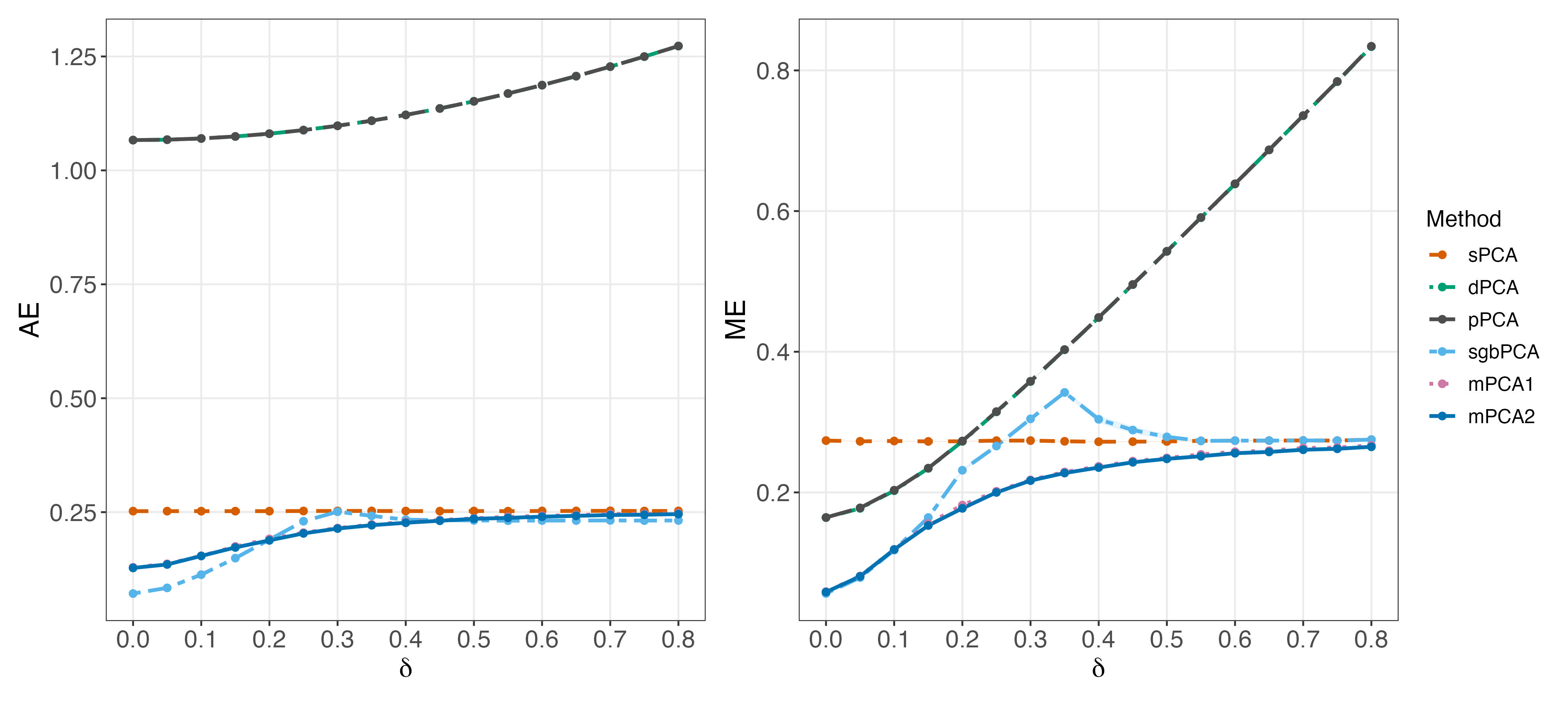}
    \caption{Effect of task heterogeneity $\delta$ with Gaussian
    outlier tasks. Left panel: Average error (AE).
    Right panel: Maximum error (ME). The outlier proportion is fixed at $5/25$.}
    \label{fig:sim1_gaus}
\end{figure}

Figure~\ref{fig:sim1_t3} presents the results when the outlier tasks
follow scaled multivariate $t_3$ distributions.
The ME patterns are broadly similar to those in the Gaussian setting:
the errors of dPCA and pPCA increase with $\delta$, sgbPCA exhibits
a peak at intermediate values of $\delta$, and the two MTL-PCA
variants remain close to one another and below sPCA. Unlike in the Gaussian setting, pPCA generally has larger errors
than dPCA. One explanation is that the “early aggregation” technique adopted by pPCA (computing eigenvectors after aggregation), compared with the “late aggregation” technique used by dPCA (computing eigenvectors before aggregation), tends to be more fragile to heavy-tailed data.

Regarding AE, mPCA1 consistently achieves the lowest error,
followed by sgbPCA, whereas mPCA2 remains much closer to sPCA.
Nevertheless, mPCA1 and mPCA2 have nearly identical ME over the
related tasks. As described in Section~\ref{subsec: tuning stra}, mPCA1 uses a
projector-distance cross-validation criterion, whereas mPCA2 uses
held-out explained variance. These results suggest that the
projector-distance criterion is more effective at controlling
overall projector error in this heavy-tailed setting, while both
strategies yield comparable performance for the related tasks.

\begin{figure}[h]
    \centering
    \includegraphics[width=\linewidth]{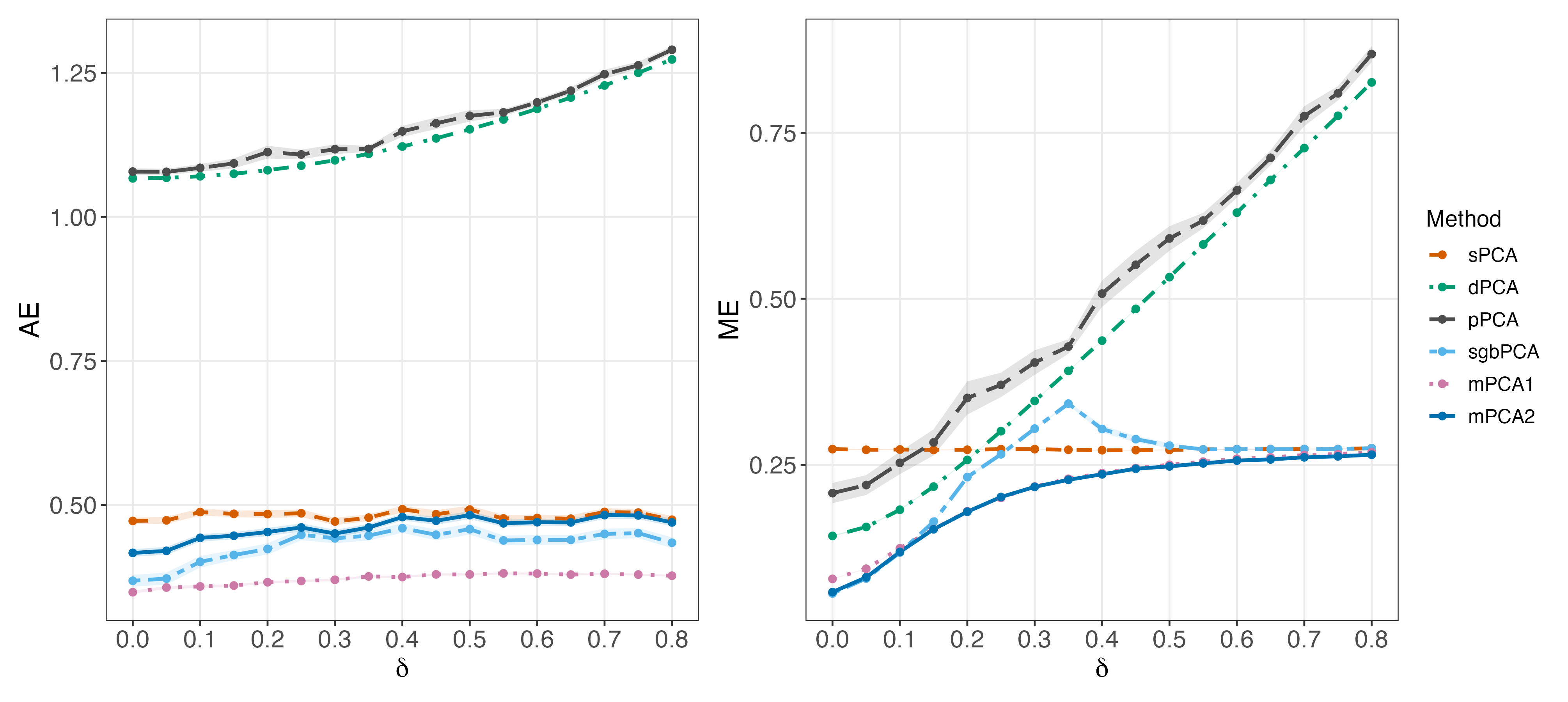}
    \caption{Effect of task heterogeneity $\delta$ with scaled
    multivariate $t_3$ outlier tasks. Left: Average error (AE). Right: Maximum error (ME). The outlier proportion is fixed
    at $5/25$.}
    \label{fig:sim1_t3}
\end{figure}

\subsection{Part II: Fixed Data Similarity with Varying Outlier Proportion}\label{subsec: sim part 2}

In this part, we examine the performance of the proposed method as the proportion of outlier tasks increases.
We fix the number of related tasks at $|\Sc|=20$ and vary the number of outlier tasks from $|\Sc^c|=0$ to $|\Sc^c|=25$.
We characterize the heterogeneity of the related tasks and outlier tasks by
\begin{equation*}
    \delta=\sup_{i\in \Sc}\|V_iV_i^T-V_0V_0^T\|_F,~~\tilde{\delta}= \sup_{i\in \Sc^c}\|V_iV_i^T-\tilde{V}_0 \tilde{V}_0^T\|_F.
\end{equation*}
We set $\delta=0.2$ and $\tilde{\delta}=1$ and generate the covariance matrices in the same manner as in Section \ref{subsec:sim part 1}.
All other simulation settings remain the same as in Part I. For sgbPCA, we set $\tau=\rho K$ and select $\rho$ from the grid $\{0.05, 0.1, \ldots, 0.95, 1\}$ using four-fold cross-validation.
We report the results for $T=3$, as performance remains largely unchanged for larger $T$.

Figure~\ref{fig:sim2_gauss} presents the results for Gaussian outlier tasks.
The AE and ME of pPCA and dPCA increase sharply as the number of outlier tasks grows, indicating substantial negative transfer.
In contrast, mPCA achieves lower AE and ME than sPCA throughout the considered range, demonstrating that it can effectively borrow information from related tasks while being robust to outlier tasks. The performance of sgbPCA is also relatively stable. It generally achieves lower errors than sPCA but higher errors than mPCA when the number of outlier tasks is small or moderate.

\begin{figure}[!htbp]
    \centering
    \includegraphics[width=1\linewidth]{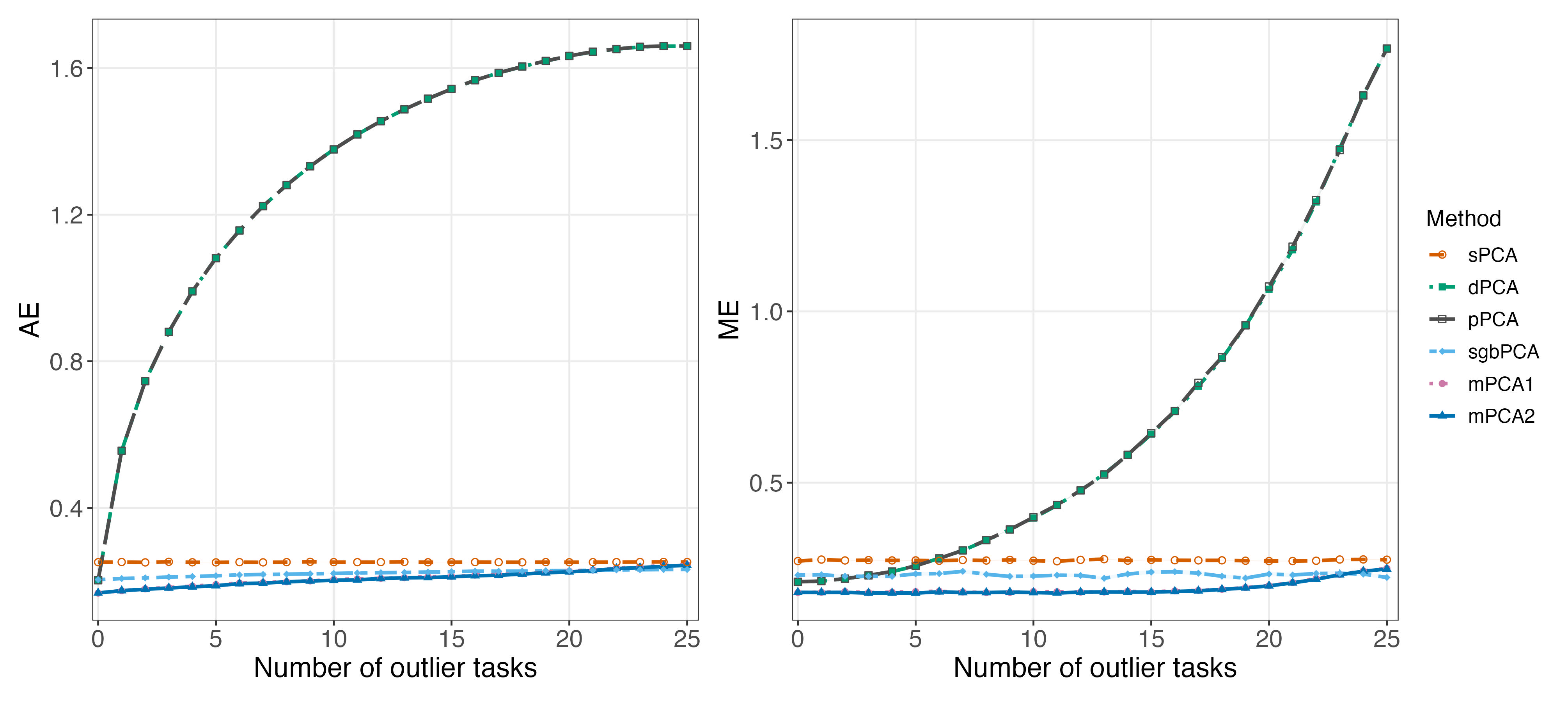}
    \caption{Effect of the number of Gaussian outlier tasks on estimation error.
    Left panel: Average error (AE).
    Right panel: Maximum error (ME).}
    \label{fig:sim2_gauss}
\end{figure}

\begin{figure}[!th]
    \centering
    \includegraphics[width=1\linewidth]{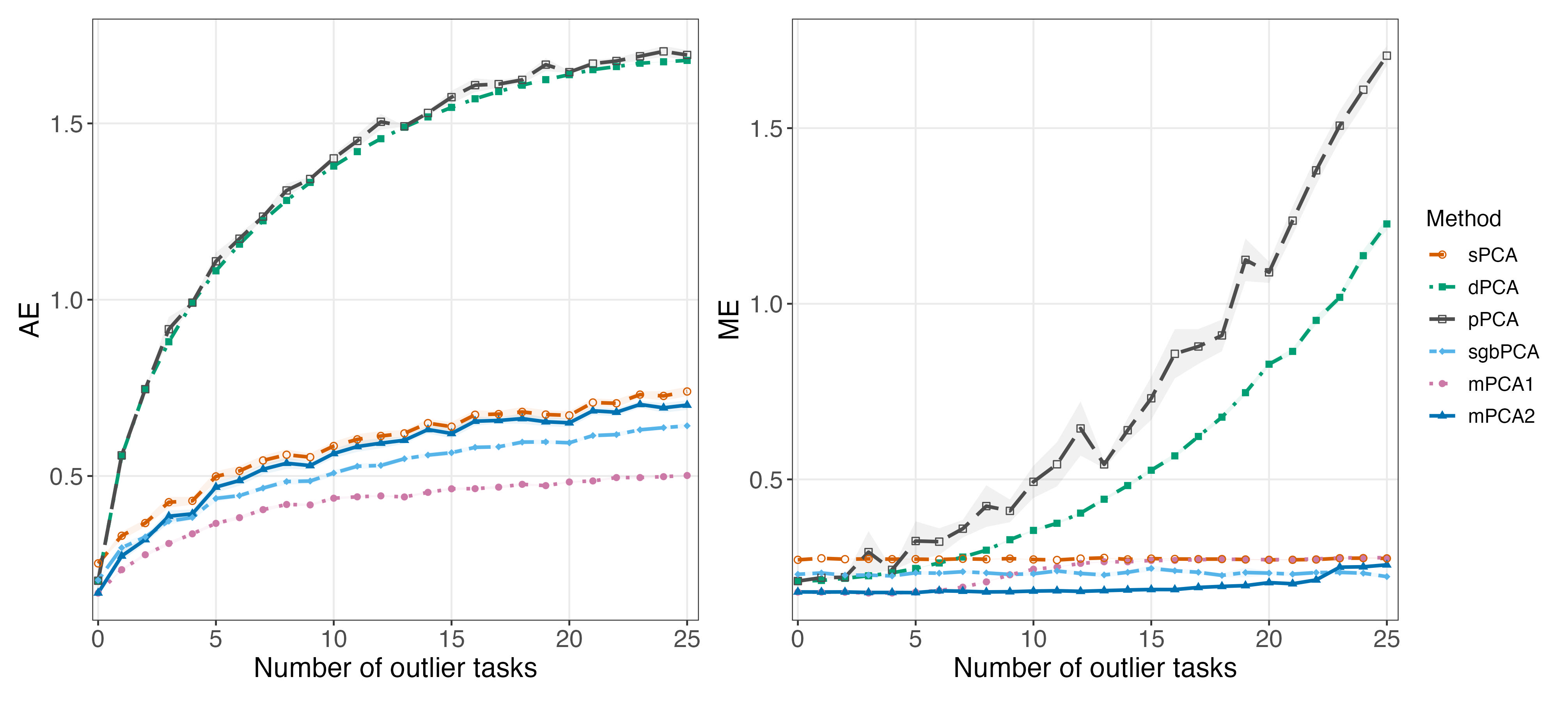}
    \caption{Effect of the number of scaled
    multivariate $t_3$ outlier tasks on estimation error. Left panel: Average error (AE).
    Right panel: Maximum error (ME).}
    \label{fig:sim2_t3}
\end{figure}

Figure~\ref{fig:sim2_t3} presents the corresponding results for heavy-tailed outlier tasks. While some of the differences become more pronounced, the general patterns remain similar to those observed in the Gaussian scenario: mPCA (the better variant) performs best overall, followed by sgbPCA, with both outperforming the three baseline methods. As observed in Part I, the two tuning strategies for mPCA exhibit different trade-offs in the heavy-tailed setting. mPCA1 provides better learning performance across all tasks, whereas mPCA2 performs better on the related tasks.

\subsection{Part III: Performance of DBMTL-PCA}
\label{subsec: sim part 3}

In this section, we evaluate the proposed depth-based multi-task learning procedure, DBMTL-PCA. Since this procedure is primarily designed to reduce the maximum error over related tasks (ME), we focus on ME in two experiments that vary the number of outlier tasks and the similarity among related tasks, respectively. We compare three methods:
\begin{enumerate}
\item mPCA: MTL-PCA with the tuning strategy based on multiple rounds of communication.
\item dbPCA: DBMTL-PCA with $B=1$, i.e., grouping all observations within each task together.
\item full dbPCA: DBMTL-PCA with $B=45$, i.e., treating each observation as a separate group.
\end{enumerate}
Both dbPCA and full dbPCA are tuned by four-fold cross-validation based on held-out explained variance.

In the first experiment, we fix the number of related tasks at $|\Sc|=30$ and vary the number
of outlier tasks over $|\Sc^c|\in\{5,10,15,20,25\}$. Data are generated as in Section \ref{subsec: sim part 2}, with parameters reset as $n=60, d=10, K=3, \sigma=1, \delta=0.1,\tilde{\delta}=0.4$. Each experimental configuration is repeated 100 times. Figure~\ref{fig:sim3 epsilon} presents the results for Gaussian
and $t_3$ outlier distributions.
In the Gaussian setting, dbPCA achieves the lowest ME for all
considered outlier counts except $|\Sc^c|=25$.
Full dbPCA performs similarly to mPCA, while the differences
among all three methods become small as the number of outlier
tasks approaches 25. The differences are more pronounced in the heavy-tailed setting.
Here, dbPCA consistently achieves the lowest ME, followed by mPCA
and full dbPCA.
The gaps between dbPCA and the other two methods widen as the
number of outlier tasks increases.
This comparison highlights the benefit of grouping observations
within tasks when constructing the matrix-depth estimator,
particularly in the presence of heavy-tailed outliers. It is consistent with the theory established in Section \ref{dbml:sec:main}.

\begin{figure}
    \centering
    \includegraphics[width=\linewidth]{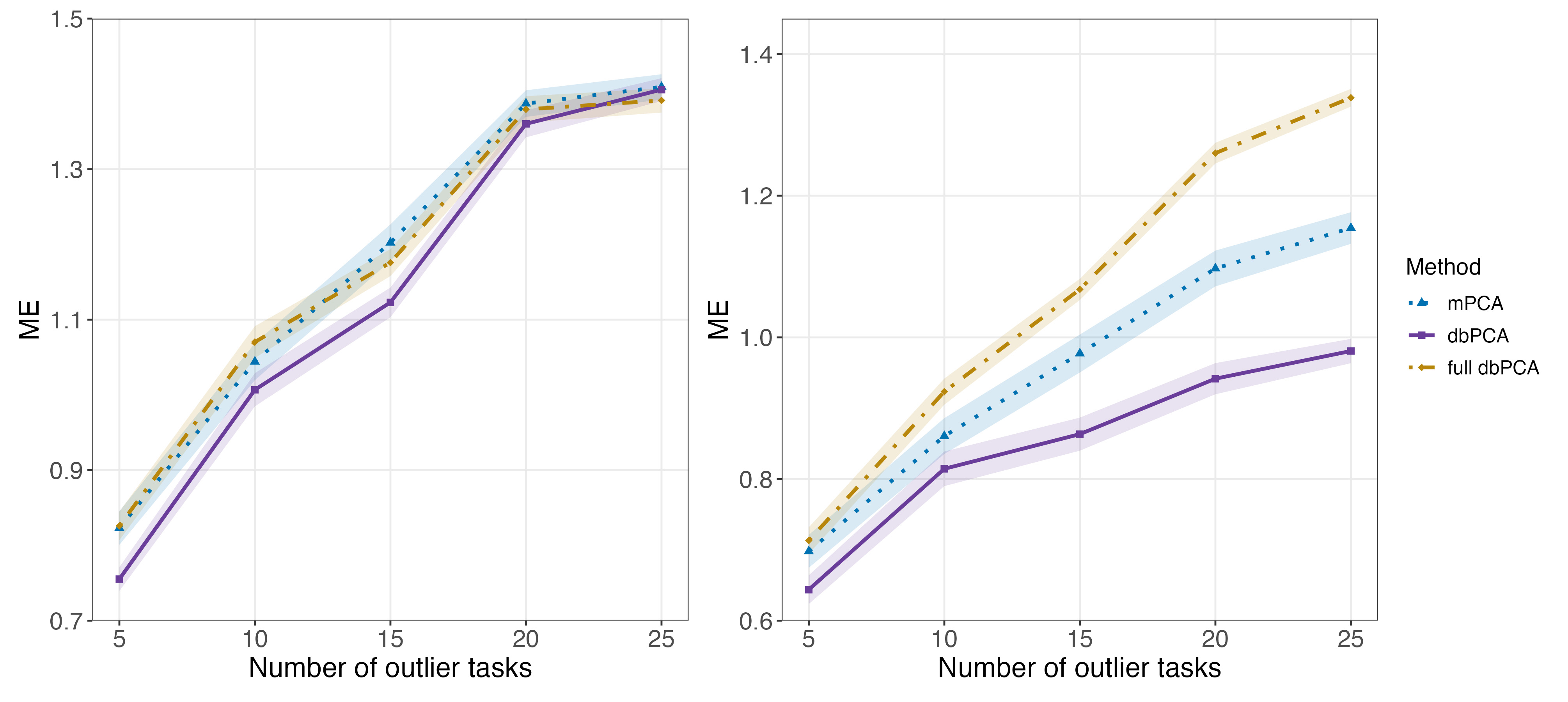}
    \caption{Effect of the number of outlier tasks on the maximum
    error over related tasks (ME).
    Left: Gaussian outlier distributions.
    Right: multivariate $t_3$ outlier distributions.
    The number of related tasks is fixed at 30.}
    \label{fig:sim3 epsilon}
\end{figure}

In the second experiment, we examine the effect of task heterogeneity in the absence
of outlier tasks. We set $|\Sc^c|=0$ and vary
$\delta\in\{0.05,0.1,\ldots,0.75,0.8\}$, while keeping all other configurations the same as in the first experiment. Figure~\ref{fig:sim3 delta} shows that ME increases with $\delta$
for all three methods as the tasks become more heterogeneous.
Nevertheless, dbPCA achieves the lowest ME throughout the
displayed range.
Its relative advantage is particularly evident at small and
moderate values of $\delta$, and it continues to outperform
mPCA and full dbPCA at larger values.
The errors of mPCA and full dbPCA are generally close, with
full dbPCA performing somewhat better at the largest values
of $\delta$.

\begin{figure}
    \centering
    \includegraphics[width=0.5\linewidth]{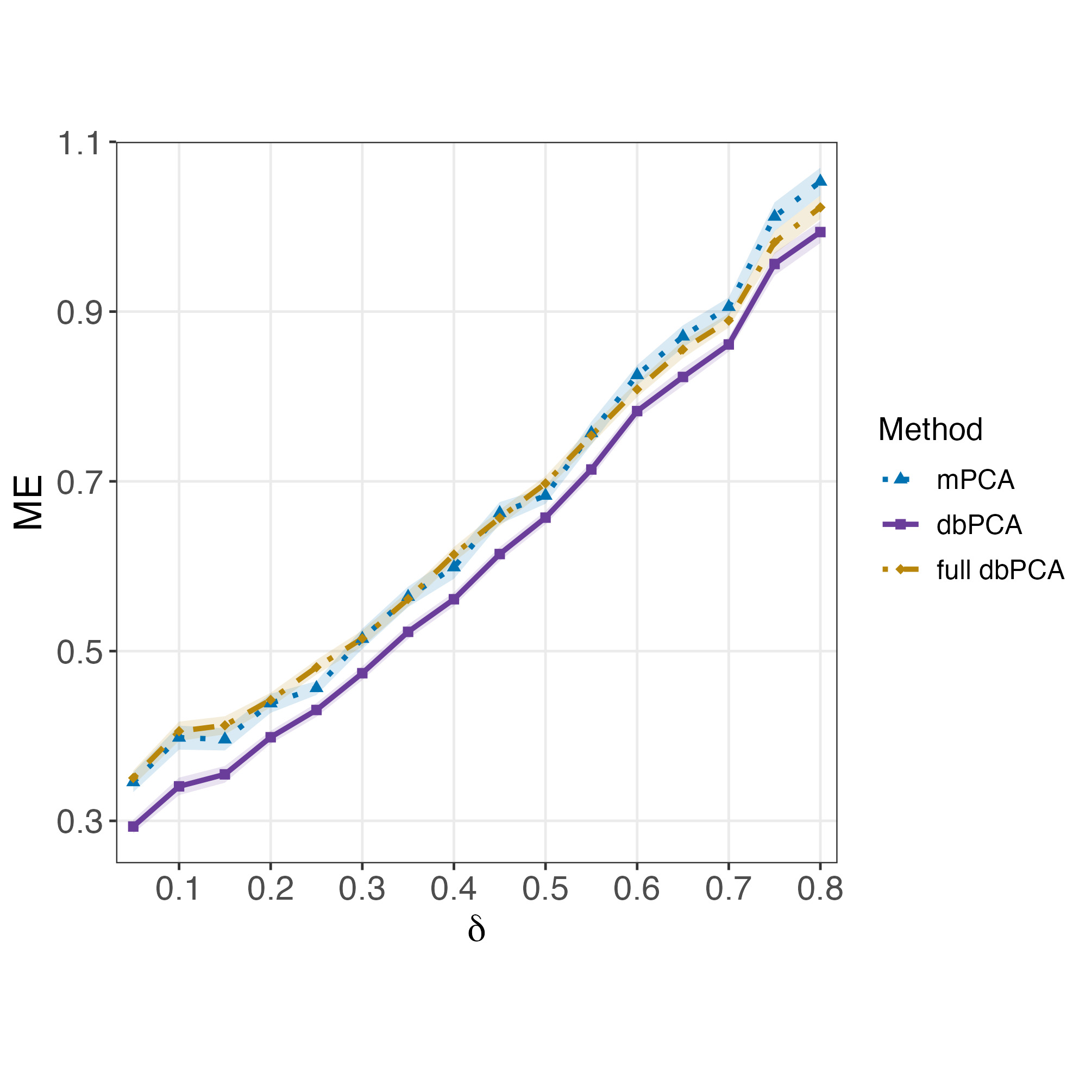}
    \caption{Effect of task heterogeneity $\delta$ on the maximum
    error over related tasks (ME) for 30 Gaussian tasks
    without outlier tasks.}
    \label{fig:sim3 delta}
\end{figure}

Together, these experiments demonstrate the effectiveness of DBMTL-PCA across different levels of task heterogeneity and contamination. The improvement of dbPCA over full dbPCA highlights the importance of
grouping in the matrix-depth construction, particularly under
heavy-tailed outlier distributions.
The advantage of dbPCA over mPCA is consistent with the theoretical analyses in Sections \ref{section:minimax} and \ref{dbml:sec:main}.

\subsection{Tuning strategy}\label{subsec: tuning stra}

MTL-PCA depends on task-specific regularization parameters
\(\lambda_1,\ldots,\lambda_m\). Theorem \ref{theorem:upperbound} shows that a common choice $\lambda_i=\lambda, i\in [m]$ can yield desirable convergence rates. We can treat $\lambda$ as a single tuning parameter and select it via cross validation. However, practically we may prefer a tuning scheme that is adaptive to the task heterogeneity. We therefore propose a more adaptive tuning approach. Specifically, for each task \(i\), we first construct a task-specific pilot scale \(\tilde\lambda_i\), and then tune a common multiplicative constant \(C\) over a grid \(\mathcal C\):
\[
    \lambda_i(C)=C\tilde\lambda_i,\qquad i=1,\ldots,m.
\]
This approach allows \(\tilde\lambda_i\) to vary across tasks, thereby adapting the amount of regularization to the task-specific difficulty. Recall that MTL-PCA is in fact a distributed algorithm which does not require full access to the data. Depending on the number of rounds of communications allowed between the central server and local datasets, we develop two different tuning strategies.

\paragraph{Tuning with one-round communication.} For each task \(i\), split the rows of \(X_i\) into \(B=4\) approximately equal parts. Let \(\hat V_{i,b}\in \mathbb{R}^{d\times K}\) denote the top \(K\) eigenvectors computed from the \(b\)-th split, \(b=1,\ldots,B\). The summary statistics $\{\hat V_{i,b}\}_{b=1}^B$ from each local dataset are then sent to the central server. The remaining steps are performed on the central server. 

For each validation fold \(b\in [B]\), compute the top $K$ eigenvectors,$\bar V_{i,-b}$, of 
\begin{equation*}
    \frac{1}{B-1}\sum_{l\neq b} \hat V_{i,l}(\hat V_{i,l})^T.
\end{equation*}
For a candidate value \(C\in\mathcal C\), we run the aggregation step \eqref{key:interp} with \(\hat{A}_i=\bar V_{i,-b}\bar V_{i,-b}^T\) and regularization parameters
\(\lambda_i(C)=C\tilde\lambda_i\) to obtain $\{\tilde{A}_i\}_{i\in [m]}$. Let \(\widetilde V_{i,-b}(C)\in \mathbb{R}^{d\times K}\) denote the top $K$ eigenvectors of $\tilde{A}_i$. We define the cross-validation error as
\[
    \operatorname{CV}_{\mathrm{or}}(C)
    =
    \frac{1}{mB}
    \sum_{i=1}^m \sum_{b=1}^B
    \left\|
        \widetilde V_{i,-b}(C) ( \widetilde V_{i,-b}(C))^T
        -
        \hat V_{i,b}(\hat V_{i,b})^T
    \right\|_F,
\]
and select
\[
    \widehat C_{\mathrm{or}}
    =
    \arg\min_{C\in\mathcal C}
    \operatorname{CV}_{\mathrm{or}}(C).
\]
Regarding the choice of $\tilde\lambda_i$, the proof of Theorem~\ref{theorem:upperbound} suggests that a suitable regularization level for task \(i\) should scale as $\sqrt{K\kappa_i^2 r_i/n_i}$ which is the statistical rate of single-task PCA for the \(i\)-th data matrix. We thus set
\[
    \tilde\lambda_i
    =
    \frac{1}{B}
    \sum_{b=1}^B
    \left\|
        \bar V_{i,-b}(\bar V_{i,-b})^T
        -
        \hat V_{i,b} (\hat V_{i,b})^T
    \right\|_F .
\]
Based on Theorem 4 of \cite{fan2019distributed} and the triangle inequality 
\[
\|\bar V_{i,-b}(\bar V_{i,-b})^T-\hat V_{i,b} (\hat V_{i,b})^T\|_F\leq \|
        \bar V_{i,-b}(\bar V_{i,-b})^T-V_iV_i^T
    \|_F+\|V_iV_i^T-\hat V_{i,b} (\hat V_{i,b})^T\|_F,
\]
we can see that the quantity $\tilde\lambda_i$ has the desirable order $\sqrt{K\kappa_i^2 r_i/n_i}$. It is larger for tasks whose single-task eigenspace estimates are less accurate, and therefore assigns a larger regularization scale to more difficult tasks. After \(\widehat C_{\mathrm{or}}\) is selected, we take the top $K$ eigenvectors, $\hat V_i^{\mathrm{agg}}$, of $\frac{1}{B}\sum_{b=1}^B \hat V_{i,b} (\hat V_{i,b})^T$ and run step \eqref{key:interp} with $\hat{A}_i=\hat V_i^{\mathrm{agg}}(\hat V_i^{\mathrm{agg}})^T$ and $\lambda_i=\widehat C_{\mathrm{or}}\tilde\lambda_i$ to obtain $\{\tilde{A}_i\}_{i\in [m]}$. The top $K$ eigenvectors of $\tilde{A}_i$ is finally sent back to the $i$ task.

\paragraph{Tuning with multiple-round communications.}

If multiple-round communications are allowed, some calculations for the cross-validation error can be conducted locally using local data instead of the summary statistics. For each task \(i\), split the rows of \(X_i\) into \(B=4\) folds. Let \(X_{i,b}\) denote the data in fold \(b\), and let \(X_{i,-b}\) denote the remaining data. Let \(\hat V_{i,-b}\) and \(\hat V_{i}\) be the top \(K\) eigenvectors computed from \(X_{i,-b}\) and \(X_{i}\), respectively. Each local server first computes a task-specific pilot scale based on the held-out unexplained variance:
\[
    \tilde\lambda_i
    =
    \frac{1}{B}
    \sum_{b=1}^B
    \left\{
    1-
    \frac{
    \operatorname{tr}\!\left(
        X_{i,b}^T X_{i,b}
        \hat V_{i,-b} \hat V_{i,-b}^T
    \right)
    }{
    \operatorname{tr}\!\left(
        X_{i,b}^T X_{i,b}
    \right)
    }
    \right\}.
\]
The summary statistics $\{\tilde\lambda_i,\hat{V}_i,\hat V_{i,-b},b\in [B]\}$ from each local dataset are sent
to the central server. For each candidate \(C\in\mathcal C\) and $b\in [B]$, the central server runs step \eqref{key:interp} with \(\hat{A}_i=\hat V_{i,-b}\hat V_{i,-b}^T\) and regularization parameters
\(\lambda_i(C)=C\tilde\lambda_i\), and let \(\widetilde V_{i,-b}(C)\) denote the resulting top $K$ eigenvectors. The central server sends back $\big\{\widetilde V_{i,-b}(C),C\in\mathcal C, b\in [B]\big\}$ to each local dataset $i$. The $i$th local server then calculates the $i$th part of the cross-validation error based on the proportion of variance explained:
\[
    \operatorname{CV}_{i}(C)
    =\frac{1}{B}\sum_{b=1}^B
    \frac{
    \operatorname{tr}\!\left(
        X_{i,b}^T X_{i,b}
        \widetilde V_{i,-b}(C) (\widetilde V_{i,-b}(C))^T
    \right)
    }{
    \operatorname{tr}\!\left(
        X_{i,b}^T X_{i,b}
    \right)
    }.
\]
Next, the error scores $\{\operatorname{CV}_{i}(C),C\in\mathcal C \}$ from the $ith$ local data are sent to the central server. The central server computes
\[
    \operatorname{CV}_{\mathrm{mr}}(C)
    =
    \frac{1}{m}
    \sum_{i=1}^m 
    \operatorname{CV}_{i}(C),
\]
and select
\[
    \widehat C_{\mathrm{mr}}
    =
    \arg\max_{C\in\mathcal C}
    \operatorname{CV}_{\mathrm{mr}}(C).
\]
This type of held-out explained variance criterion is commonly used to tune PCA-based estimators (e.g. \cite{yamane2016multitask}). After selecting \(\widehat C_{\mathrm{rm}}\), the central server runs step \eqref{key:interp} with \(\hat{A}_i=\hat V_{i}\hat V_{i}^T\) and \(\lambda_i=\widehat C_{\mathrm{rm}}\tilde\lambda_i\) to obtain $\{\tilde{A}_i\}_{i\in [m]}$. The top $K$ eigenvectors of $\tilde{A}_i$ is finally sent back to the $i$ task. Note that we may communicate the $b$- and $C$-related quantities using more rounds to reduce the communication cost per round.

\paragraph{Simulation comparison.} To demonstrate the advantage of our adaptive tuning approach over the non-adaptive one, we compare four tuning strategies: 
\begin{itemize}
\item Adaptive tuning with one-round communication (AT-or): our proposed strategy 
\item Non-adaptive tuning with one-round communication (NAT-or): set $\tilde{\lambda}_i=\lambda,i\in [m]$ and follow the remaining steps of AT-or 
\item Adaptive tuning with multiple-round communication (AT-mr): our proposed strategy 
\item Non-adaptive tuning with multiple-round communication (NAT-mr): set $\tilde{\lambda}_i=\lambda,i\in [m]$ and follow the remaining steps of AT-mr
\end{itemize}

We adopt the simulation setting of Section~\ref{subsec:sim part 1}, except that we replace the common
spike magnitude with task-specific spike magnitudes.
Specifically, we consider the spiked covariance model
\begin{equation*}
    \Sigma_i=V_i\Lambda_iV_i^T+I_d,
\end{equation*}
where
$\Lambda_i=\operatorname{diag}(\sigma_{i1},\ldots,\sigma_{iK})
\in\Rb{K\times K}$.
For each task $i$, we independently draw
$s_i\sim\operatorname{Unif}[5,20]$ and, conditional on $s_i$,
draw $\sigma_{ik}\sim\operatorname{Unif}[s_i/2,s_i]$
independently for $k\in[K]$. The task-specific spike magnitudes are held fixed across
replications, and each configuration is repeated 100 times.

Figure~\ref{fig:sim4 ME} shows the maximum error over related tasks. In the Gaussian setting, adaptive and non-adaptive tuning perform similarly when $\delta$ is small.
In the heavy-tailed setting, AT-mr and NAT-mr also perform similarly at small values of $\delta$, whereas NAT-or has lower ME than AT-or at the smallest values. Starting from $\delta=0.2$, each adaptive strategy achieves lower ME than its non-adaptive counterpart. The improvement is most pronounced at intermediate levels of
task heterogeneity, particularly for the one-round strategy. These results suggest that allowing the regularization scales
to vary according to task-specific information improves eigenspace estimation across a broad range of heterogeneity levels.

\begin{figure}
    \centering
    \includegraphics[width=1\linewidth]{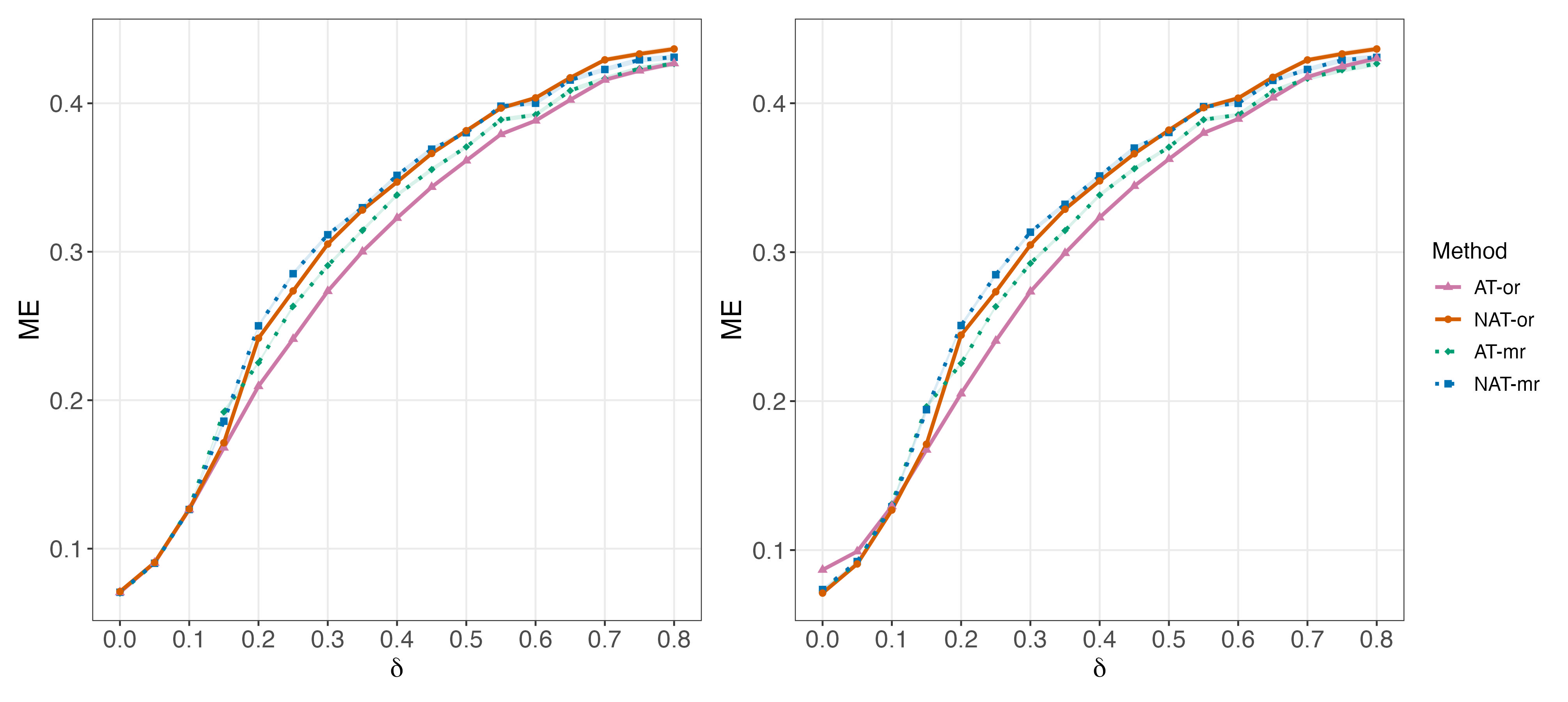}
    \caption{Comparison of adaptive and non-adaptive tuning strategies. Left: Gaussian outlier tasks. Right: scaled multivariate  $t_3$ outlier tasks.}
    \label{fig:sim4 ME}
\end{figure}

\section{Real data analysis}\label{sec:real data}
In this section, we verify the effectiveness of the proposed methods, MTL-PCA and DBMTL-PCA, on real datasets.

\subsection{Human activity recognition (HAR)}
\label{subsec: RDA HAR}

The \emph{Human Activity Recognition Using Smartphones} dataset
\cite{anguita2013public}, available from the
\href{https://archive.ics.uci.edu/dataset/240/human+activity+recognition+using+smartphones}{UCI Machine Learning Repository},
contains smartphone sensor measurements collected from 30 volunteers
performing six activities. Each observation is represented by a
561-dimensional feature vector, and the number of observations per
volunteer ranges from 281 to 409.
PCA and its variants have been widely used for dimension reduction
in human activity recognition, where sensor measurements are often
represented by high-dimensional feature vectors
\cite{chen2017robust,hassan2018robust,lerman2018overview}.
Dimension reduction is particularly important under a subject-specific
multi-task formulation, in which each volunteer is treated as a
separate task \cite{tian2022robust,duan2023adaptive,li2019online}.
In this setting, the number of observations per task is smaller than
the feature dimension, making accurate estimation of the task-specific
population eigenspaces challenging.

We treat the 30 volunteers as separate tasks and focus on two activities: standing and lying down. The sample size of each task varies from
95 to 179, while the feature dimension remains 561.
In each replication, we randomly choose 20\% of the observations
from each task for testing and use the remaining 80\% for estimation.
We fix the embedding dimension at $K=15$ and repeat the experiment
50 times.
For task \(i\), let \(X_i\) denote the test-data matrix and let $\hat{V}_i$ be the eigenspace estimator based on the training data. We use average explained variance and minimum explained variance to evaluate the performance of each method:
\begin{align*}
    \text{Avg. EV}
    &= \frac{1}{30}\sum_{i=1}^{30}\frac{\|X_i \hvvt{i}{}\|_F^2}{\|X_i\|_F^2},   \quad   \text{Min. EV}= \min_{1\leq i\leq 30}\frac{\|X_i \hvvt{i}{}\|_F^2}{\|X_i\|_F^2}.
\end{align*}
We compare the six methods from Section \ref{sec:numerical studies}: sPCA, dPCA, pPCA, sgbPCA, mPCA, and dbPCA. For sgbPCA, we parameterize the source-selection threshold as
$\tau=\rho K$ and select $\rho$ by four-fold cross-validation over
the grid $\{0.05,0.1,\ldots,0.95,1\}$.
Performance stabilized at $T=3$, with negligible changes for larger
iteration counts in our experiments. We therefore report the results for $T=3$. The tuning strategies for mPCA and dbPCA are the same as those in Section \ref{subsec: sim part 3}.

\begin{figure}[!t]
    \centering
    \includegraphics[height=7.cm, width=15.5cm]{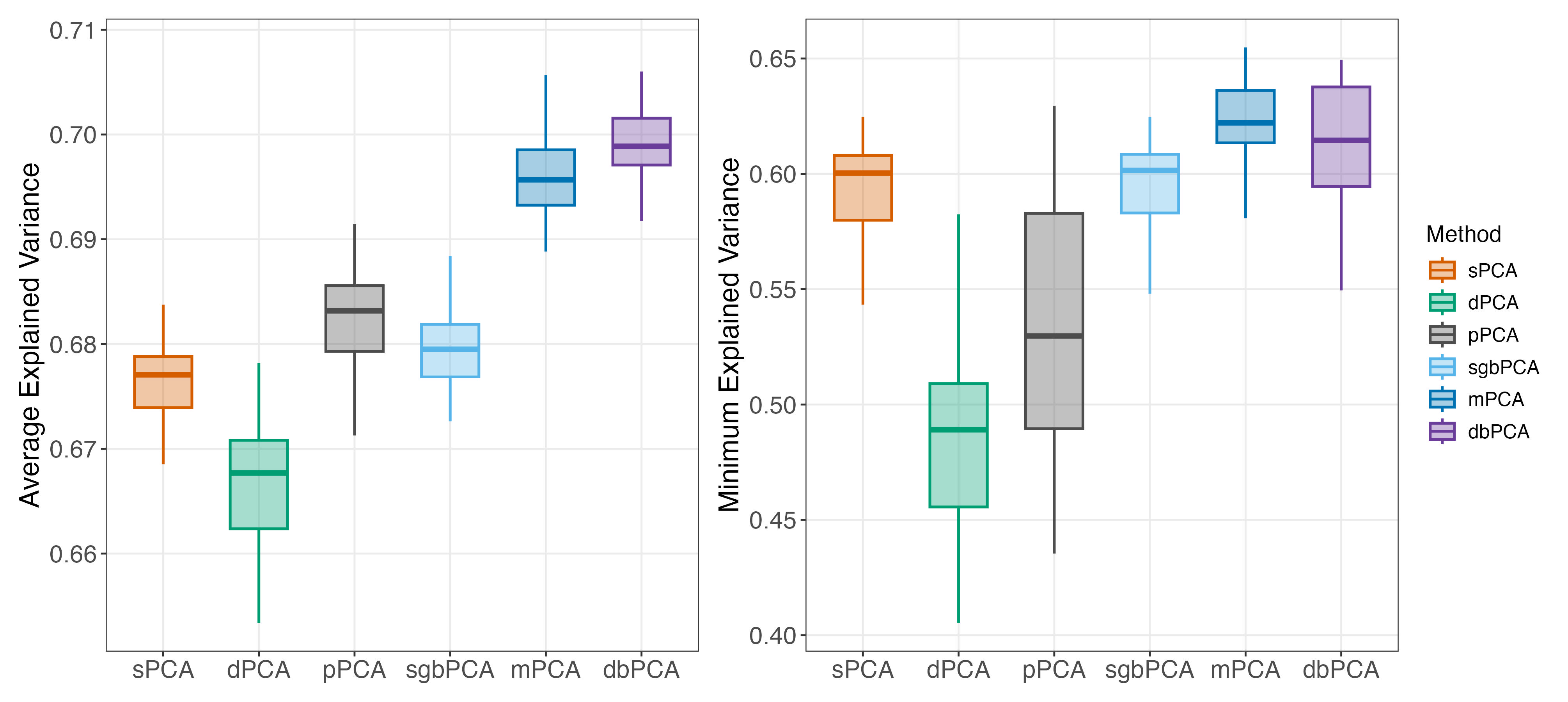}
    \caption{Average (left) and minimum (right) held-out explained
    variance across the 30 tasks in the two-activity HAR experiment.
    Each box plot summarizes the corresponding metric over
    50 replications.}
    \label{fig:AE_ME_HAR_2Act}
\end{figure}

\begin{table}[!thbp]
    \centering
    \caption{Average and minimum held-out explained variance in
    the two-activity HAR experiment. Entries are means $\pm$ standard errors over 50 replications.}
    \label{tab:har}
    \vspace{0.2cm}
    \begin{tabular}{lcc}
        \toprule
        Method & Avg. EV & Min. EV \\
        \midrule
        sPCA
        & $0.6766 \pm 0.0006$
        & $0.5915 \pm 0.0033$ \\
        dPCA
        & $0.6668 \pm 0.0008$
        & $0.4845 \pm 0.0058$ \\
        pPCA
        & $0.6827 \pm 0.0007$
        & $0.5346 \pm 0.0076$ \\
        sgbPCA
        & $0.6795 \pm 0.0006$
        & $0.5926 \pm 0.0033$ \\
        mPCA
        & $0.6962 \pm 0.0005$
        & $\mathbf{0.6214 \pm 0.0029}$ \\
        dbPCA
        & $\mathbf{0.6994 \pm 0.0005}$
        & $0.6135 \pm 0.0037$ \\
        \bottomrule
    \end{tabular}
\end{table}

Figure~\ref{fig:AE_ME_HAR_2Act} shows the distributions of average
and minimum explained variance over the 50 replications.
Table~\ref{tab:har} reports the corresponding means and standard errors. It can be seen that the two proposed methods, mPCA and dbPCA, outperform competing methods by a considerable margin. By contrast, competing methods show different trade-offs across the two criteria. Although pPCA improves average explained variance relative to sPCA, it has substantially lower minimum explained variance. sgbPCA provides a modest improvement in average explained variance while maintaining minimum explained variance close to that of sPCA.

Finally, Figure~\ref{fig:boxplot_TaskAE_HAR_2Act} shows the
distributions across the 30 task-specific explained variances averaged over the 50 replications. The distributions for mPCA and dbPCA have the highest medians and
are generally shifted upward relative to those of the competing
methods. This comparison complements the aggregate results of Figure \ref{fig:AE_ME_HAR_2Act} and Table \ref{tab:har}, by showing that the advantages of mPCA and dbPCA are also visible in
the distribution of performance across volunteers.

\begin{figure}[h]
    \centering
    \includegraphics[height=7.cm, width=14cm]{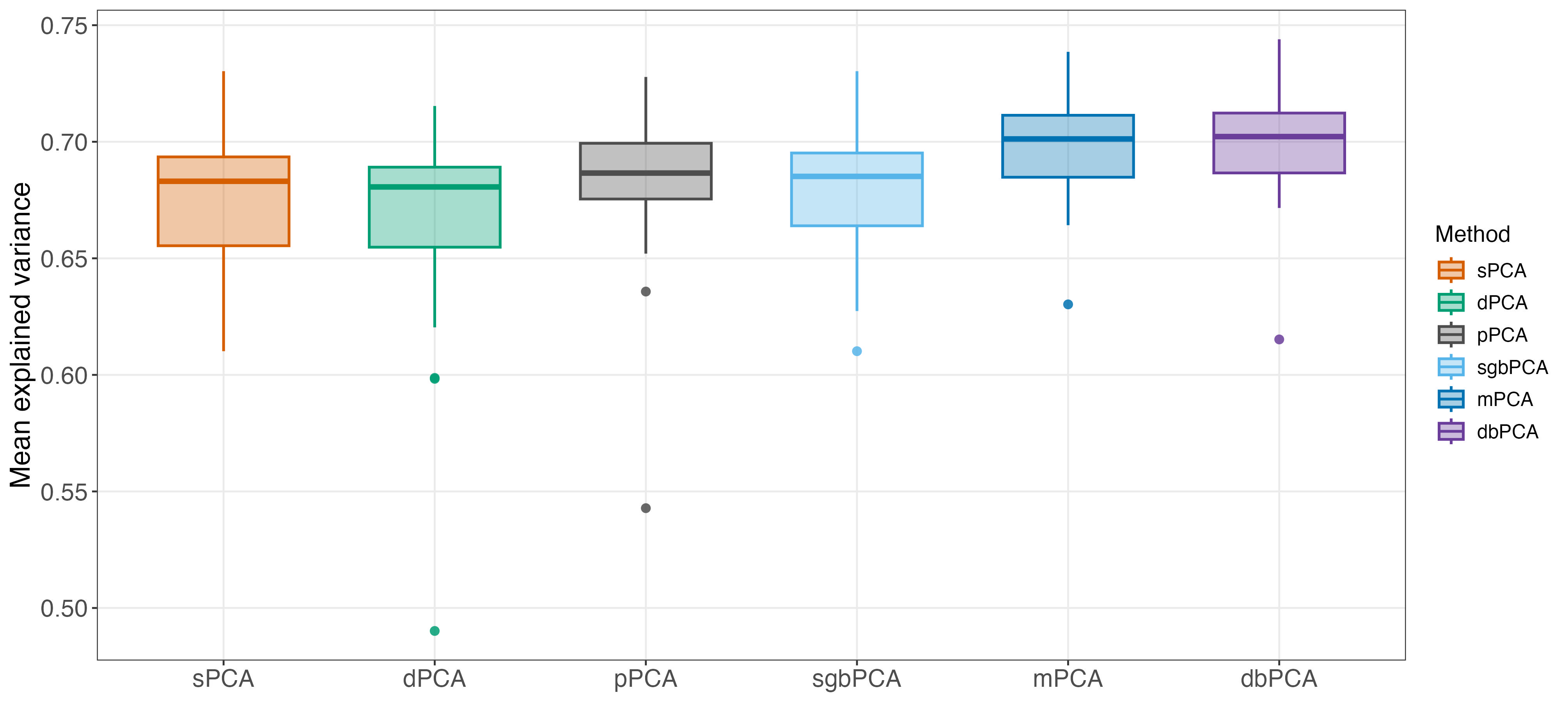}
    \caption{Distributions of task-specific held-out explained
    variance in the two-activity HAR experiment.
    Each box plot summarizes the 30 task-specific values obtained
    by averaging explained variance over 50 replications.}
    \label{fig:boxplot_TaskAE_HAR_2Act}
\end{figure}


\subsection{Federated Extended MNIST (FEMNIST)}
\label{subsec:RDA_FEMNIST}

The \emph{Federated Extended MNIST} (FEMNIST) dataset is a
writer-partitioned collection of handwritten digits and letters included
in the LEAF benchmark for federated learning \cite{caldas2018leaf}.
We use the Flower release
of FEMNIST (\href{https://huggingface.co/datasets/flwrlabs/femnist}{https://huggingface.co/datasets/flwrlabs/femnist}), in which each observation consists of a $28\times 28$
grayscale image and its associated writer and character labels.
The dataset contains 62 character classes: 10 digits, 26 uppercase
letters, and 26 lowercase letters. Its natural partition by writer
provides a multi-task setting in which the tasks may differ in both
handwriting style and character composition.

We select a fixed set of 25 writers and
treat each writer as a separate task. The resulting task-specific sample
sizes range from 308 to 436. We retain all available character classes
and vectorize each image, giving a feature dimension of $d=784$.
In each replication, we randomly reserve 20\% of the observations from
each writer for testing and use the remaining 80\% for estimation.
The split is constructed to approximately preserve the character
composition of each writer. We set the embedding dimension to $K=30$
and repeat the experiment 50 times, using the same train--test splits
for all methods.

We adopt the same error metrics as in Section \ref{subsec: RDA HAR} to evaluate the six methods. Figure~\ref{fig:AE_ME_FEMNIST} displays the distributions of average
and minimum explained variance over the 50 replications, and
Table~\ref{tab:femnist} reports the corresponding means and standard errors. Compared to the HAR dataset in Section \ref{subsec: RDA HAR}, the FEMNIST dataset exhibits greater task heterogeneity, causing non-adaptive aggregation methods, including pPCA and dPCA, to suffer from negative transfer. In contrast, the proposed methods, mPCA and dbPCA, outperform sgbPCA and yield considerable improvements over sPCA, in terms of average explained variance. 

\begin{figure}[!t]
    \centering
    \includegraphics[height=7.cm, width=15.5cm]{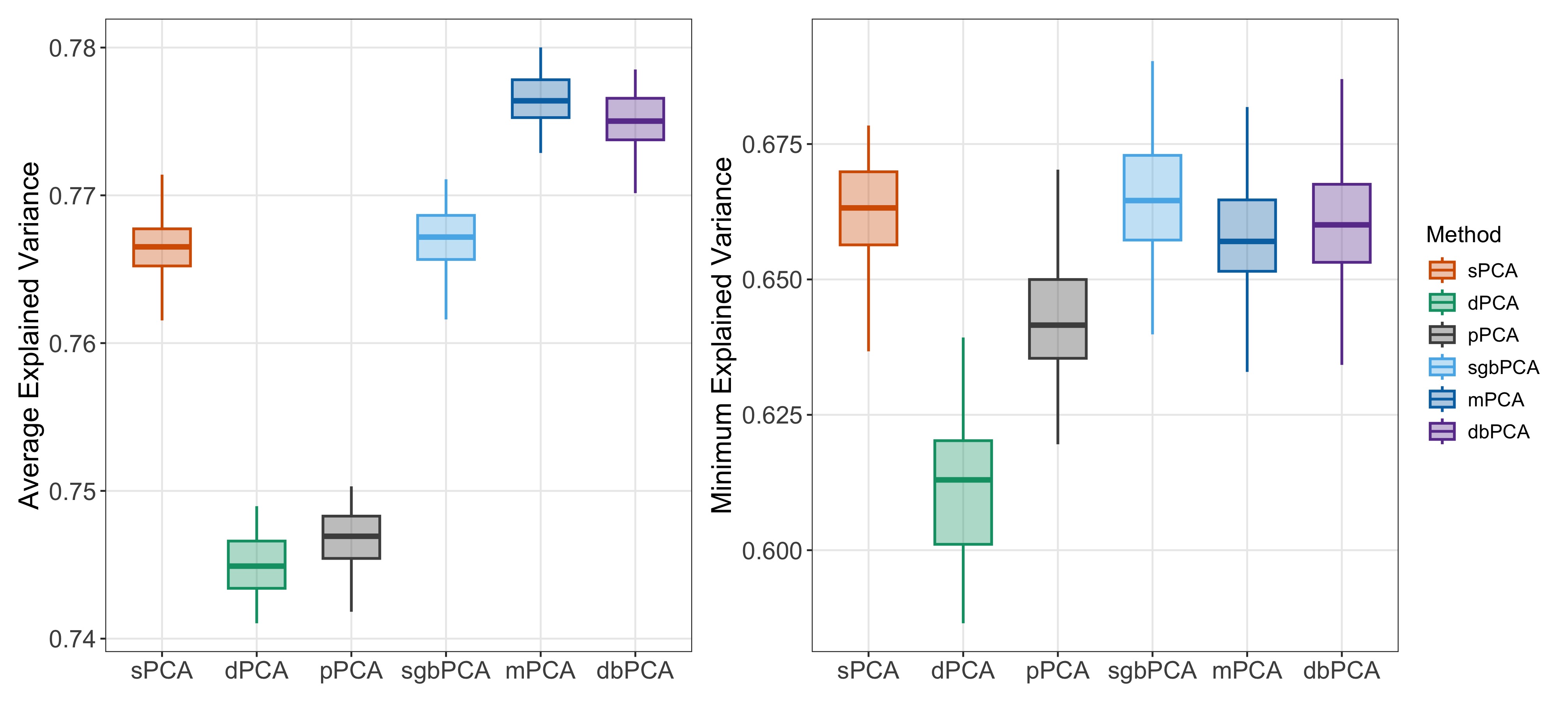}
    \caption{Average (left) and minimum (right) held-out explained
    variance across the 25 writers in the FEMNIST experiment.
    Each box plot summarizes the corresponding metric over
    50 matched replications.}
    \label{fig:AE_ME_FEMNIST}
\end{figure}

\begin{table}[htbp]
    \centering
    \caption{Average and minimum held-out explained variance in the FEMNIST experiment. Entries are means
    $\pm$ standard errors over 50 replications.}
    \label{tab:femnist}
    \vspace{0.2cm}
    \begin{tabular}{lcc}
        \toprule
        Method & Avg. EV & Min. EV \\
        \midrule
        sPCA
        & $0.7666 \pm 0.0003$
        & $0.6624 \pm 0.0014$ \\
        dPCA
        & $0.7449 \pm 0.0003$
        & $0.6115 \pm 0.0020$ \\
        pPCA
        & $0.7469 \pm 0.0003$
        & $0.6427 \pm 0.0019$ \\
        sgbPCA
        & $0.7672 \pm 0.0003$
        & $\mathbf{0.6649 \pm 0.0015}$ \\
        mPCA
        & $\mathbf{0.7763 \pm 0.0003}$
        & $0.6578 \pm 0.0019$ \\
        dbPCA
        & $0.7751 \pm 0.0003$
        & $0.6606 \pm 0.0019$ \\
        \bottomrule
    \end{tabular}
\end{table}

The plot of minimum explained variance (calculated across all tasks) suggests that dominant outlier tasks may exist and essentially drive the results. Indeed, the best performer, sgbPCA (treating those tasks as targets), yields only modest improvements over sPCA. While our methods do not achieve better worst-case task performance than sPCA, they perform better on most tasks, as shown in Figure \ref{fig:boxplot_TaskAE_FEMNIST}. The central portions of the mPCA and dbPCA
distributions are shifted upward relative to those of the competing methods, and these two methods have the highest medians across writers.

\begin{figure}[!thbp]
    \centering
    \includegraphics[height=7.cm, width=14cm]{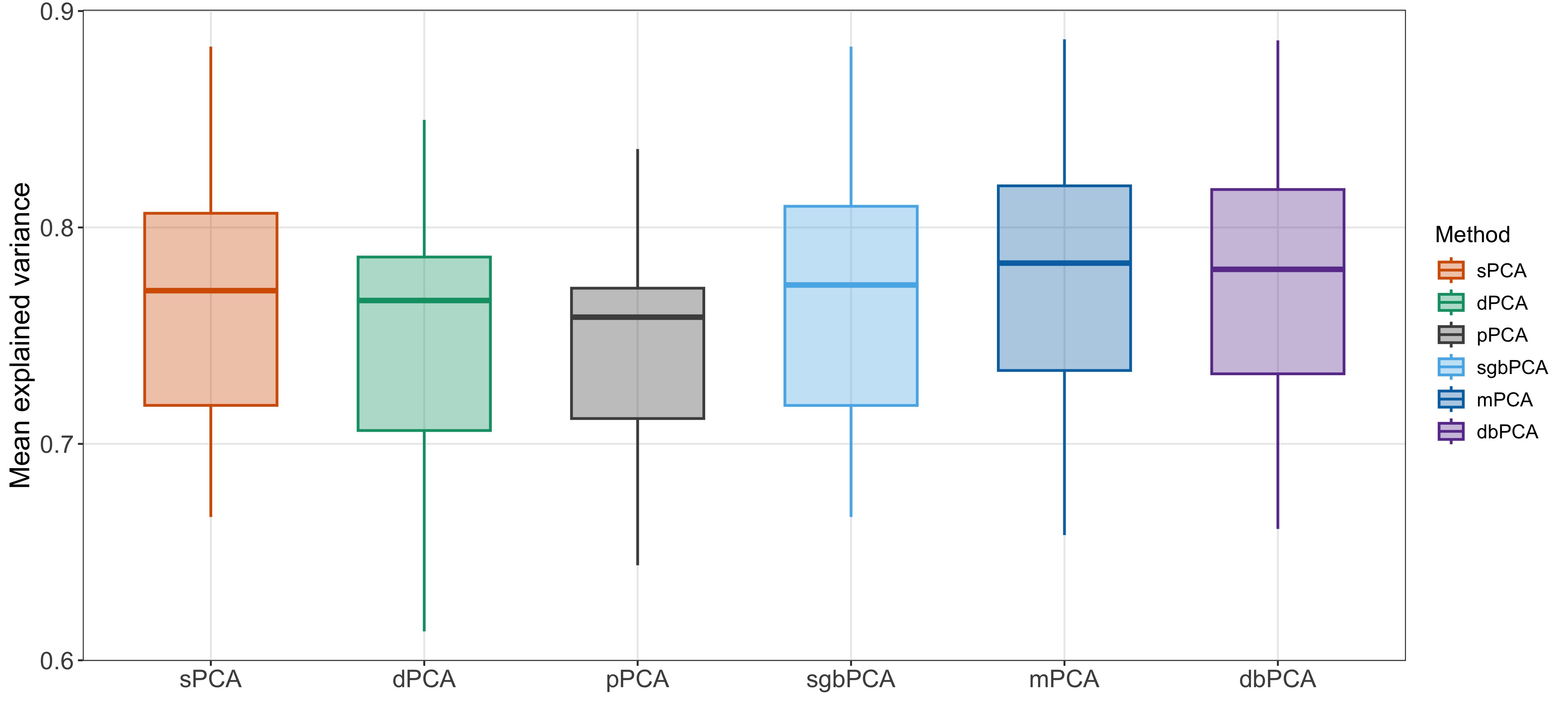}
    \caption{
    Distributions of task-specific held-out explained variance in the FEMNIST experiment.
    Each box plot summarizes the 25 writer-specific values obtained by averaging explained variance over 50 replications.
    }
    \label{fig:boxplot_TaskAE_FEMNIST}
\end{figure}


\section{Discussions}\label{sec:discussion}
In this paper, we develop two adaptive and robust multi-task PCA procedures based on a two-stage framework and further characterize their minimax optimality under suitable conditions. Extensive numerical studies demonstrate favorable eigenspace estimation performance of our approach. 

Several important directions remain for future investigation. First, our depth-based procedure has only been shown to attain minimax optimality in the bounded signal-to-noise ratio regime, primarily because the derived bound depends on the ambient dimension. Some recent work has studied trimmed-mean-type covariance matrix estimators that can achieve dimension-free bounds in single-task learning \cite{oliveira2024improved,abdalla2024covariance,minasyan2025statistically}. It is worthwhile to incorporate their ideas into our two-stage framework to develop multi-task PCA procedures that enjoy stronger optimality guarantees. Second, our depth-based method and the aforementioned trimmed-mean-type approaches all rely on algorithms with no global optimality guarantees. It is of fundamental interest to investigate polynomial-time algorithms that achieve minimax optimality under data contamination. Some recent progress has been made for other multi-task learning problems \cite{tian2026contaminated}, yet much remains to be explored for the PCA setting.

\section{Proofs of technical results}\label{sec:proofs}

\subsection{Proof of Proposition \ref{theorem:structure}}\label{subsec:prop1}

\begin{proof}
For a fixed $A$, define 
\begin{equation*}
 f_A(A_i)= \frac{1}{2}\|A_i-\hat{A}_i\|_F^2+\lambda_i\|A_i-A\|_F.   
\end{equation*}
Then the subdifferential takes the form
\begin{equation*}
 \partial f_A(A_i)=\begin{cases}
  A_i-\hat{A}_i+\lambda_i\frac{A_i-A}{\|A_i-A\|_F},A_i\neq A,\\
  A_i-\hat{A}_i+\lambda_i W, A_i=A,
 \end{cases} 
 \end{equation*}
where $W\in \mathbb{R}^{d\times d}$ is symmetric with $\|W\|_F\leq 1$. To get
\begin{equation*}
0\in \partial f_A(A_i),
\end{equation*}
we derive 
\begin{equation*}
 A_i=\begin{cases}
     A, \| A-\hat{A}_i\|_F\leq \lambda_i,\\
     (1-\alpha)\hat{A}_i+\alpha  A, \text{otherwise,}
 \end{cases}   
 \end{equation*}
 where  $\alpha=\frac{\lambda_i}{\|\hat{A}_i- A\|_F}$. Plugging the above solution into $f_A(A_i)$, we have
 \begin{equation*}
     \min_{A_i} f_A(A_i)=\rho_{\lambda_i}(\|A-\hat{A}_i\|_F).
 \end{equation*}
Thus, we can profile out $A_i$ in \eqref{key:interp} to obtain
\begin{align*}
\tilde{A}=  \argmin_{A} \sum_{i=1}^m\rho_{\lambda_i}(\|A-\hat{A}_i\|_F). 
\end{align*}    
\end{proof}

\subsection{Proof of Theorem \ref{theorem:upperbound}}\label{subsec:thm1}
Throughout the proof, the uppercase letters $C, C_1,C_2,\ldots$ denote constants that might only depend on the constant $M$ from Definition \ref{Def:subgaussian}, and their values can change at each occurrence. We recall that the notation $a_n\lesssim b_n$ means there exists a constant $C$ (which may depend on $M$) such that $a_n\leq C b_n$.

\subsubsection{Useful lemmas}
We first introduce several useful lemmas used in the proof of Theorem \ref{theorem:upperbound}.
\begin{lemma}[Weyl's inequality]\label{ine:weyl}
Given two compatible symmetric matrices $A$ and $B$, we have
\begin{equation}
|\sigma_i(A)-\sigma_i(B)|\leq \|A-B\|.
\end{equation}
\end{lemma}
The Davis-Kahan theorem bounds the distance between eigenspaces under perturbation. The following user-friendly version can be found in \cite{yu2015useful}.
\begin{lemma}[Davis-Kahan]\label{lemma: DK}
For a symmetric matrix $A\in \mathbb{R}^{d\times d}$, let $V_K(A)\in \mathbb{R}^{d\times K}$ denote the matrix consisting of the top K eigenvectors of $A$. Then, given two symmetric matrices $A, B\in \mathbb{R}^{d\times d}$, we have
\begin{equation*}
\| V_K(A)V_K(A)^{T}-V_K(B)V_K(B)^{T} \|_F\leq \dfrac{2\sqrt{2}\min(K^{1/2}\|A-B\|, \|A-B\|_F)}{\sigma_K(B)-\sigma_{K+1}(B)},
\end{equation*}
where $\sigma_{K+1}(B)=-\infty$ if $K=d$.
\end{lemma}

The following lemma characterizes the tail and expectation bounds of the difference between the sample covariance and the population covariance.
It can be found in \cite{koltchinskii2017concentration}.
\begin{lemma}\label{lemma:cov}
 Let $Z_i, i\in[n],$ be i.i.d.  subgaussian random vector   in $\mathbb R^d$ and have zero mean and covariance $\Sigma$.  
 Define 
 \begin{equation*}
     \hat{\Sigma}=\frac{1}{n}\sum_{i=1}^n Z_iZ_i^T.
 \end{equation*}
For all $t\geq 1$, with probability at least $1-\exp(-t)$, we have
\begin{equation}\label{ine: samplecov}
\|\hat{\Sigma}-\Sigma\|\leq C\|\Sigma\|\left( \dfrac{r(\Sigma)}{n}\vee \sqrt{\dfrac{r(\Sigma)}{n}}+\dfrac{t}{n}\vee \sqrt{\dfrac{t}{n}}  \right).
\end{equation}
And the moment satisfies 
\begin{equation*}
\Eb\|\hat{\Sigma}-\Sigma\|\leq C\|\Sigma\|   \left( \dfrac{r(\Sigma)}{n}\vee\sqrt{\dfrac{r(\Sigma)}{n}}   \right).
\end{equation*}
Here, $r(\Sigma)=\frac{{\rm Tr}(\Sigma)}{\|\Sigma\|}$ denotes the effective rank, and $C>0$ is a constant.
\end{lemma}

For a random variable $Y\in \mathbb{R}$, define its sub-exponential norm \cite{vershynin2018high} as 
\[
\|Y\|_{\psi_1}={\rm inf}\Big\{t>0: \mathbb{E}\exp(|Y|/t)\leq 2\Big\}.
\]
Recall that $\hat{V}_i\in \mathbb{R}^{d\times K}$ consists of the top $K$ eigenvectors of $\hat{\Sigma}_i=\frac{1}{n_i}X_i^TX_i$ and $V_i\in \mathbb{R}^{d\times K}$ has the top $K$ eigenvectors of $\Sigma_i$. Denote $\Sigma^*_i=\Eb[\hat{V}_i\hat{V}_i^T]$ and let  $V_i^*\in \mathbb{R}^{d\times K}$ represent the top $K$ eigenvectors of $\Sigma^*_i$.

\begin{lemma}[Lemma 1 in \cite{fan2019distributed}]\label{lemma: subexponentail norm}
Suppose Assumption \ref{assumption:subgaussian} holds. For each $i\in \Sc$, if $r_i\leq n_i$, we have
\begin{equation*}
\big\| \| \hat{V}_i\hat{V}_i^T-\Sigma_i^* \|_F\big \|_{\psi_1} \leq C\kappa_i \sqrt{\dfrac{Kr_i}{n_i}}.
\end{equation*}
\end{lemma}

\begin{lemma}[Theorem 2 in \cite{fan2019distributed}] \label{lemma: thm2 fan}
For each $i\in \Sc$, if the rows of $X_i$ have symmetric innovation, then $\Sigma_i^*$ and $\Sigma_i$ share the same set of eigenvectors. Furthermore, if $\left\|{\Sigma}_i^*-V_i V_i^T\right\|<1 / 2$, then  $\|V_i^*(V_i^*)^T-V_iV_i^T\|_F=0$.
\end{lemma}

\begin{lemma}[Theorem 3 in \cite{fan2019distributed}] \label{lemma: thm3 fan}
Under the same conditions of Lemma \ref{lemma: subexponentail norm}, there are constants $C_1$ and $C_2$ such that 
 \begin{equation*}
     \|V_i^*(V_i^*)^T-V_iV_i^T\|_F \leq C_1\left\|{\Sigma}^*_i-{V}_i {V}_i^T\right\|_F \leq C_2 \kappa^2_i \sqrt{K} r_i / n_i, ~\forall i\in \Sc.
 \end{equation*}
\end{lemma}

\begin{lemma} [Lemma 4 in \cite{fan2019distributed}] \label{lemma:sum of exponential norms}
If $\left\{Z_i\right\}_{i\in [n]}$ are independent random vectors in a separable Hilbert space (where the norm is denoted by $\|\cdot\|_H)$ with $\mathbb{E} [Z_i]=\mathbf{0}$ and  $\| \|Z_i\|_H\|_{\psi_1} \leq L_i<\infty$. We have
\begin{equation*}
    \Big\|\| \sum_{i=1}^n Z_i\|_H\Big\|_{\psi_1} \lesssim \sqrt{\sum_{i=1}^n L_i^2} .
\end{equation*}
\end{lemma}


\subsubsection{Deterministic analysis}
\begin{lemma}\label{lemma2}
Define $f(A)=\sum_{i\in \Sc}\rho_{\lambda_i}(\|A-\hat{A}_i\|_F)$, and 
\begin{equation*}
G_{\Sc}=\cap_{i\in \Sc} \mathbb{B}_F(\hat{A}_i, \lambda_i),
\end{equation*}
where $\mathbb B_F(\hat{A}_i,\lambda_i)=\{M\in\Rb{d\times d}:\|M-\hat{A}_i\|_F< \lambda_i\}$.
If 
\begin{equation}\label{cond:mean in G}
\frac{1}{|\Sc|}\sum_{i\in \Sc}\hat{A}_i\in G_{\Sc},
\end{equation}
then $\frac{1}{|\Sc|}\sum_{i\in \Sc} \hat{A}_i$ is the unique global minimizer of $f(A)$. 
\end{lemma}

\begin{proof}[Proof of Lemma \ref{lemma2}]
By the definition of the Huber loss $\rho_{\lambda}(\cdot)$, it is clear that 
\[
f(A)=\frac{1}{2}\sum_{i\in \Sc}\|A-\hat{A}_i\|_F^2, ~~{\rm ~for~}A\in G_{\Sc}.
\]
If (\ref{cond:mean in G}) is satisfied, then $\frac{1}{|\Sc|}\sum_{i\in \Sc} \hat{A}_i$ is a strict local minimizer of $f(A)$. Since $f(A)$ is convex, a strict local minimizer is also the unique global minimizer. 
\end{proof}
Now we introduce a useful lemma in \cite{duan2023adaptive}.
\begin{lemma}\label{lemma: WDF2} 
[Lemma F.2 in \cite{duan2023adaptive}] Let $f: \mathbb{R}^p \rightarrow \mathbb{R}$ be a convex function and $x_0=\operatorname{argmin}_x f({x})$. Suppose there exist $\nu>0$ and $\gamma>0$ such that $\nabla^2 f({x}) \succeq \nu\boldsymbol{I}, \forall {x} \in \bb({x}_0, \gamma)$. If $g: \mathbb{R}^p \rightarrow \mathbb{R}$ is convex and $\lambda$-Lipschitz for some $\lambda<\nu\gamma$, then $f({x})+g({x})$ has a unique minimizer and it belongs to $\mathbb{B}({x}_0, \lambda / \nu)$.
\end{lemma}

\begin{lemma}\label{lemma3}
Adopt the same notation $G_{\Sc}$ from Lemma \ref{lemma2}. Define 
\[
h=0.99\cdot \min_{i\in \Sc}\Big\{\lambda_i-\|\hat{A}_i-\frac{1}{|\Sc|}\sum_{i\in \Sc}\hat{A}_i\|_F \Big\}.
\]
If the following conditions are satisfied, 
\begin{align}
\label{lemma3: cond1}\frac{1}{|\Sc|}\sum_{i\in \Sc}\hat{A}_i\in G_\Sc   \\
\label{lemma3:cond2}\sum_{i\in \Sc^c}\lambda_i< |\Sc| \cdot h
\end{align}
 we have
\begin{equation}
\label{center:approx}
	\Big\|\tilde{A}_i-\frac{1}{|\Sc|}\sum_{i\in \Sc} \hat{A}_i\Big\|_F\leq \dfrac{\sum_{i\in \Sc^c} \lambda_i}{|\Sc|},\quad \forall i \in \Sc.
\end{equation}
\end{lemma}
\begin{proof}
Define the following two convex functions:
\begin{align*}
f(A)=\sum_{i\in \Sc}\rho_{\lambda_i}(\|A-\hat{A}_i\|_F), \quad g(A)=\sum_{i\in \Sc^c}\rho_{\lambda_i}(\|A-\hat{A}_i\|_F).
\end{align*}
By Proposition \ref{theorem:structure}, $\tilde{A}$ is the minimizer of $f(A)+g(A)$. By Lemma \ref{lemma2} and Condition \eqref{lemma3: cond1}, $\frac{1}{|\Sc|}\sum_{i\in \Sc} \hat{A}_i$ is the minimizer of the function $f(A)$. We aim to utilize Lemma \ref{lemma: WDF2} to obtain \eqref{center:approx}. To this end, we first note that $h>0$ under Condition \eqref{lemma3: cond1}. For any $A$ satisfying $\|A-\frac{1}{|\Sc|}\sum_{i\in \Sc} \hat{A}_i\|_F\leq h$, since
\begin{align*}
\|A-\hat{A}_i\|_F&\leq \Big\|A-\frac{1}{|\Sc|}\sum_{i\in \Sc} \hat{A}_i\Big\|_F+\Big\|\hat{A}_i-\frac{1}{|\Sc|}\sum_{i\in \Sc} \hat{A}_i\Big\|_F \\
&<\min_{i\in \Sc}\Big\{\lambda_i-\|\hat{A}_i-\frac{1}{|\Sc|}\sum_{i\in \Sc}\hat{A}_i\|_F \Big\}+\Big\|\hat{A}_i-\frac{1}{|\Sc|}\sum_{i\in \Sc} \hat{A}_i\Big\|_F\leq \lambda_i, ~~\forall i\in \Sc,
\end{align*}
we have
\begin{align}
\label{save:time:eq}
\Big\|A-\frac{1}{|\Sc|}\sum_{i\in \Sc} \hat{A}_i\Big\|_F\leq h ~~ \Rightarrow ~~A\in G_{\Sc}.
\end{align}
Given that $f(A)=\frac{1}{2}\sum_{i\in \Sc}\|A-\hat{A}_i\|_F^2, A\in G_{\Sc}$, we can thus conclude
\[
\nabla^2 f(A) \succeq |\Sc|\boldsymbol{I}, \quad  \forall A \in \bb\Big(\frac{1}{|\Sc|}\sum_{i\in \Sc} \hat{A}_i, h\Big).
\]
Moreover, the Lipschitz property of Huber loss function $\rho_{\lambda_i}(\cdot)$ implies that $g(A)$ is Lipschitz continuous with constant $\sum_{i\in \Sc^c}\lambda_i$. Therefore, invoking Lemma \ref{lemma: WDF2} under Condition \eqref{lemma3:cond2} gives 
\begin{align}
\label{center:ineq:form}
\Big\|\tilde{A}-\frac{1}{|\Sc|}\sum_{i\in \Sc} \hat{A}_i\Big\|_F\leq \dfrac{\sum_{i\in \Sc^c} \lambda_i}{|\Sc|} < h.
\end{align}
The above together with \eqref{save:time:eq} implies $\tilde{A}\in G_{\Sc}$, i.e., $\|\tilde{A}-\hat{A}_i\|_F< \lambda_i, \forall i\in \Sc$. Then combining Proposition \ref{theorem:structure} and \eqref{center:ineq:form} completes the proof.
\end{proof}

\begin{lemma}
\label{determin:general:error}
Let Assumption \ref{assumption:data simi} hold and the following conditions be satisfied,
\begin{align}
&\lambda_i\geq 2\bigg(\Big\|\frac{1}{|\Sc|}\sum_{i\in \Sc} (\hat{A}_i-V_iV_i^T)\Big\|_F+\|\hat{A}_i-V_iV_i^T\|_F+2\delta\bigg), ~~\forall i\in \Sc, \label{general:error:con1}\\
&0.495(1-\epsilon)m\cdot \min_{i\in \Sc}\lambda_i>\sum_{i\in \Sc^c}\lambda_i. \label{general:error:con2}
\end{align}
\begin{itemize}
\item[(i)] We have for each $i\in \Sc$,
\begin{align*}
\|\tilde{V}_i\tilde{V}_i^T-V_iV_i^T\|_F\leq 2\sqrt{2}\bigg(\Big\|\frac{1}{|\Sc|}\sum_{i\in \Sc} (\hat{A}_i-V_iV_i^T)\Big\|_F+\frac{\sum_{i\in \Sc^c}\lambda_i}{(1-\epsilon)m}+2\delta\bigg).
\end{align*}
\item[(ii)] For each $i\in \Sc$, if the data have symmetric innovation and satisfy $\|\mathbb{E}\hat{A}_i-V_iV_i^T\|_F\leq 1/4$, then we have
\begin{align*}
\|\tilde{V}_i\tilde{V}_i^T-V_iV_i^T\|_F\leq 4\sqrt{2}\bigg(\Big\|\frac{1}{|\Sc|}\sum_{i\in \Sc} (\hat{A}_i-\mathbb{E}\hat{A}_i)\Big\|_F+\frac{\sum_{i\in \Sc^c}\lambda_i}{(1-\epsilon)m}+(3+\sqrt{2}/8)\delta\bigg).
\end{align*}
\end{itemize}
\end{lemma}
\begin{proof}
Proof of Part (i). According to Lemma \ref{lemma: DK}, it holds that
\begin{align}
\label{start:point:dk}
\|\tilde{V}_i\tilde{V}_i^T-V_iV_i^T\|_F\leq 2\sqrt{2}\|\tilde{A}_i-V_iV_i^T\|_F.
\end{align}
We aim to employ Lemma \ref{lemma3} to help attain the bound. To this end, we first utilize Assumption \ref{assumption:data simi} to get
\begin{align*}
    \Big\|\hat{A}_i-\frac{1}{|\Sc|}\sum_{i\in \Sc} \hat{A}_i\Big\|_F&\leq \Big\|\hat{A}_i-V_iV_i^T\Big\|_F+\Big\|V_iV_i^T-\frac{1}{|\Sc|}\sum_{i\in \Sc} V_iV_i^T\Big\|_F+\Big\|\frac{1}{|\Sc|}\sum_{i\in \Sc} (\hat{A}_i-V_iV_i^T)\Big\|_F \\
    &\leq \Big\|\hat{A}_i-V_iV_i^T\Big\|_F+\Big\|\frac{1}{|\Sc|}\sum_{i\in \Sc} (\hat{A}_i-V_iV_i^T)\Big\|_F +2\delta.
\end{align*}
Under Assumption \ref{assumption:data simi}, it is direct to verify that the above inequality together with Conditions \eqref{general:error:con1}-\eqref{general:error:con2} implies Conditions \eqref{lemma3: cond1}-\eqref{lemma3:cond2}. We can then apply Lemma \ref{lemma3} to obtain
\begin{align}
\label{pre:lemma:out}
	\Big\|\tilde{A}_i-\frac{1}{|\Sc|}\sum_{i\in \Sc} \hat{A}_i\Big\|_F\leq \dfrac{\sum_{i\in \Sc^c} \lambda_i}{|\Sc|}\leq \dfrac{\sum_{i\in \Sc^c} \lambda_i}{(1-\epsilon)m},\quad \forall i \in \Sc,
\end{align}
leading to 
\begin{align}
&\|\tilde{A}_i-V_iV_i^T\|_F \nonumber \\
\leq &\Big \|\tilde{A}_i-\frac{1}{|\Sc|}\sum_{i\in \Sc} \hat{A}_i\Big \|_F+\Big\|\frac{1}{|\Sc|}\sum_{i\in \Sc} (\hat{A}_i-V_iV_i^T)\Big\|_F+\Big\|\frac{1}{|\Sc|}\sum_{i\in \Sc} V_iV_i^T-V_iV_i^T\Big\|_F \nonumber \\
\leq &  \dfrac{\sum_{i\in \Sc^c} \lambda_i}{(1-\epsilon)m}+\Big\|\frac{1}{|\Sc|}\sum_{i\in \Sc} (\hat{A}_i-V_iV_i^T)\Big\|_F+2\delta, ~~\forall i \in \Sc. \label{alternative:bound}
\end{align}
Putting together \eqref{start:point:dk} and \eqref{alternative:bound} completes the proof of Part (i).

Proof of Part (ii). Let the columns of $\bar{V}_i\in \mathbb{R}^{d\times (d-K)}$ denote the smallest $(d-K)$ eigenvectors of $\Sigma_i$. Under the additional conditions of Part (ii), we can invoke Lemma \ref{lemma: thm2 fan} to write the eigen-decomposition of $\mathbb{E}\hat{A}_i$ in the form:
\[
\mathbb{E}\hat{A}_i=V_i\Lambda_iV_i^T+\bar{V}_i\bar{\Lambda}_i\bar{V}_i^T,
\]
where the diagonal matrices $\Lambda_i,\bar{\Lambda}_i$ contain the top $K$ and remaining eigenvalues of $\mathbb{E}\hat{A}_i$ respectively. Moreover, Lemma \ref{ine:weyl} together with the condition $\|\mathbb{E}\hat{A}_i-V_iV_i^T\|_F\leq 1/4$ shows that $\sigma_{\min}(\Lambda_i)\geq 3/4$ and $\sigma_{\max}(\bar{\Lambda}_i)\leq 1/4$. For each $i\in \Sc$, there exist orthogonal matrices $O_i\in \mathbb{R}^{K\times K}, \bar{O}_i\in \mathbb{R}^{(d-K)\times (d-K)}$ such that
\begin{align}
\label{small:err:approx}
\|V_i-VO_i\|_F\leq \|V_iV_i^T-VV^T\|_F\leq \delta, \quad \|\bar{V}_i-\bar{V}\bar{O}_i\|_F\leq \|\bar{V}_i\bar{V}_i^T-\bar{V}\bar{V}^T\|_F\leq \delta,
\end{align}
where $V\in \mathbb{R}^{d\times K}$ is the ``center" matrix in Assumption \ref{assumption:data simi} and $\bar{V}\in \mathbb{R}^{d\times (d-K)}$ satisfies $VV^T+\bar{V}\bar{V}^T=I_d$. Construct the following matrices:
\[
H_i=(VO_i)\Lambda_i (VO_i)^T+(\bar{V}\bar{O}_i)\bar{\Lambda}_i (\bar{V}\bar{O}_i)^T, \quad i\in \Sc.
\]
It is clear that $H_i$ shares the same top $K$ eigenspace ${\rm Col}(V)$. Define the average matrix 
\[
\frac{1}{|\Sc|}\sum_{i\in \Sc}H_i=V\Big(\frac{1}{|\Sc|}\sum_{i\in \Sc}O_i\Lambda_i O_i^T\Big)V^T+\bar{V}\Big(\frac{1}{|\Sc|}\sum_{i\in \Sc}\bar{O}_i\bar{\Lambda}_i \bar{O}_i^T\Big)\bar{V}^T.
\]
Since 
\begin{align*}
&\sigma_{\min}\Big(\frac{1}{|\Sc|}\sum_{i\in \Sc}O_i\Lambda_i O_i^T\Big)\geq \frac{1}{|S|}\sum_{i\in \Sc}\sigma_{\min}(O_i\Lambda_i O_i^T)\geq 3/4, \\
&\sigma_{\max}\Big(\frac{1}{|\Sc|}\sum_{i\in \Sc}\bar{O}_i\bar{\Lambda}_i \bar{O}_i^T\Big)\leq \frac{1}{|\Sc|}\sum_{i\in \Sc}\sigma_{\max}(\bar{O}_i\bar{\Lambda}_i \bar{O}_i^T)\leq 1/4,
\end{align*}
we know that $\frac{1}{|\Sc|}\sum_{i\in \Sc}H_i$ has the same top $K$ eigenspace ${\rm Col}(V)$, and it has eigengap (between $K$-th and $(K+1)$-th eigenvalues) at least $1/2$. Based on these results, we now bound $\|\tilde{V}_i\tilde{V}_i^T-V_iV_i^T\|_F$ as follows: $\forall i\in \Sc$,
\begin{align}
\label{key:bias:start}
&~~~~\|\tilde{V}_i\tilde{V}_i^T-V_iV_i^T\|_F \nonumber \\ 
&\leq \|\tilde{V}_i\tilde{V}_i^T-VV^T\|_F+\|VV^T-V_iV_i^T\|_F\nonumber \\
&\leq 4\sqrt{2}\Big\|\tilde{A}_i-\frac{1}{|\Sc|}\sum_{i\in \Sc}H_i\Big\|_F+\delta \nonumber \\
&\leq 4\sqrt{2}\Big\|\tilde{A}_i-\frac{1}{|\Sc|}\sum_{i\in \Sc}\hat{A}_i\Big\|_F+4\sqrt{2}\Big\|\frac{1}{|\Sc|}\sum_{i\in \Sc}(\hat{A}_i-H_i)\Big\|_F+\delta \nonumber \\
&\leq 4\sqrt{2}\Big\|\frac{1}{|\Sc|}\sum_{i\in \Sc}(\hat{A}_i-H_i)\Big\|_F+\delta+\dfrac{4\sqrt{2}\sum_{i\in \Sc^c} \lambda_i}{(1-\epsilon)m} \nonumber \\
&\leq 4\sqrt{2}\Big\|\frac{1}{|\Sc|}\sum_{i\in \Sc}(\mathbb{E}\hat{A}_i-H_i)\Big\|_F+4\sqrt{2}\Big\|\frac{1}{|S|}\sum_{i\in \Sc}(\hat{A}_i-\mathbb{E}\hat{A}_i)\Big\|_F+\delta+\dfrac{4\sqrt{2}\sum_{i\in \Sc^c} \lambda_i}{(1-\epsilon)m},
\end{align}
where the second inequality is by Lemma \ref{lemma: DK}, and the fourth inequality is due to \eqref{pre:lemma:out}. It remains to bound $\big\|\frac{1}{|\Sc|}\sum_{i\in \Sc}(\mathbb{E}\hat{A}_i-H_i)\big\|_F$:
\begin{align}
\label{key:bias:term}
&\Big\|\frac{1}{|\Sc|}\sum_{i\in \Sc}(\mathbb{E}\hat{A}_i-H_i)\Big\|_F \nonumber \\
\leq &\frac{1}{|\Sc|}\sum_{i\in \Sc} \Big(\|(VO_i)\Lambda_i (VO_i)^T-V_i\Lambda_i V_i^T\|_F+\|(\bar{V}\bar{O}_i)\bar{\Lambda}_i (\bar{V}\bar{O}_i)^T-\bar{V}_i\bar{\Lambda}_i\bar{V}_i^T\|_F\Big) \nonumber \\
\leq &\frac{2}{|\Sc|}\sum_{i\in \Sc} \Big(\sigma_{\max}(\Lambda_i)\|VO_i-V_i\|_F+\sigma_{\max}(\bar{\Lambda}_i)\|\bar{V}\bar{O}_i-\bar{V}_i\|_F\Big) \leq 3\delta.
\end{align}
Here, the second inequality holds since $\|A\Delta A^T-B\Delta B^T\|_F\leq \|(A-B)\Delta A^T\|_F+\|B\Delta(A-B)^T\|_F\leq 2 \sqrt{\sigma_{\max}(\Delta^2)}\|A-B\|_F$ holds for orthonormal matrices $A,B$ and diagonal matrix $\Delta$, and the last inequality is due to \eqref{small:err:approx}. Combining \eqref{key:bias:start}-\eqref{key:bias:term} finishes the proof of Part (ii). 
\end{proof}
\subsubsection{Completion of the proof}

\begin{proof}

We first state several useful probabilistic results:
\begin{enumerate}
\item[(a)] Given any $t\in [(1-\log m)_+,\min_{i\in \Sc}n_i/2]$, with probability at least $1-\exp(-t)$, it holds that
\begin{align}
\label{single:learning:rate:union}
\max_{i\in \Sc}\|\hat{A}_i-V_iV_i^T\|_F\leq C_1 \max_{i\in \Sc} \sqrt{\frac{K\kappa_i^2(r_i+\log m+t)}{n_i}}.
\end{align}
This result is a direct application of Lemmas \ref{lemma: DK} and \ref{lemma:cov}, with the union bound. Note that under the conditions $\max_{i\in \Sc}\frac{\sqrt{K}\kappa_i^2r_i}{n_i}\leq c_1, \min_{i\in \Sc}n_i\geq 2\log m$, the choice of $t$ ensures $1\leq t+\log m \leq n_i, r_i\leq n_i$ so that Lemma \ref{lemma:cov} can be employed properly.
\item[(b)] For any given $\tilde{t}\geq 0$, the following holds with probability at least $1-2\exp(-\tilde{t})$,
\begin{align}
\label{sub:exp:bound:aver}
\Big\|\frac{1}{|\Sc|}\sum_{i\in \Sc} (\hat{A}_i-\mathbb{E}\hat{A}_i)\Big\|_F\leq C_2 \tilde{t}\sqrt{\frac{1}{|\Sc|^2}\sum_{i\in \Sc}\frac{K\kappa_i^2r_i}{n_i}}.
\end{align}
The result is directly due to Lemmas \ref{lemma: subexponentail norm} and \ref{lemma:sum of exponential norms}. Further applying Lemma \ref{lemma: thm3 fan} shows that with the same probability,
\begin{align}
\label{another:bound:aver}
\Big\|\frac{1}{|\Sc|}\sum_{i\in \Sc} (\hat{A}_i-V_iV_i^T)\Big\|_F\leq C_2 \tilde{t}\sqrt{\frac{1}{|\Sc|^2}\sum_{i\in \Sc}\frac{K\kappa_i^2r_i}{n_i}}+C_3\frac{1}{|\Sc|}\sum_{i\in \Sc}\frac{\sqrt{K}\kappa_i^2r_i}{n_i}.
\end{align}
\end{enumerate}

To prove Theorem \ref{theorem:upperbound}, we consider two cases. 

\begin{itemize}
\item[] Case I: $\delta\leq \max_{i\in \Sc} \sqrt{\frac{K\kappa_i^2(r_i+\log m+t)}{n_i}}$. In this case, we aim to apply Lemma \ref{determin:general:error} to obtain the desirable bound. Since all the $\lambda_i$'s are identical and $\epsilon\leq 1/4$, it is straightforward to confirm that Condition \eqref{general:error:con2} is satisfied. Results \eqref{single:learning:rate:union} and \eqref{another:bound:aver} imply that with probability at least $1-2\exp(-\tilde{t})-\exp(-t)$, it holds that $\forall i\in \Sc$,
\begin{align*}
&\Big\|\frac{1}{|\Sc|}\sum_{i\in \Sc} (\hat{A}_i-V_iV_i^T)\Big\|_F+\|\hat{A}_i-V_iV_i^T\|_F+2\delta \\
\leq &C_1 \max_{i\in \Sc} \sqrt{\frac{K\kappa_i^2(r_i+\log m+t)}{n_i}}+C_2 \tilde{t}\sqrt{\frac{1}{|\Sc|^2}\sum_{i\in \Sc}\frac{K\kappa_i^2r_i}{n_i}}+C_3\frac{1}{|\Sc|}\sum_{i\in \Sc}\frac{\sqrt{K}\kappa_i^2r_i}{n_i}+2\delta.
\end{align*}
Given that $\max_{i\in \Sc}\frac{\sqrt{K}\kappa_i^2r_i}{n_i}\leq c_1$ and $\tilde{t}\leq c_2c_3\sqrt{|\Sc|}$, each term on the right-hand side of the above inequality is smaller than $\lambda_i/8$ as long as the constant $c_2$ in $\lambda_i$ is chosen sufficiently large. Hence, Condition \eqref{general:error:con1} is also satisfied. We can invoke Part (i) of Lemma \ref{determin:general:error} to obtain
\begin{align}
\label{proof:bound:form1}
\max_{i\in \Sc}\|\tilde{V}_i\tilde{V}_i^T-V_iV_i^T\|_F &\lesssim \tilde{t}\sqrt{\frac{1}{|\Sc|^2}\sum_{i\in \Sc}\frac{K\kappa_i^2 r_i}{n_i}}+\delta\nonumber \\
&+\epsilon \max_{i\in \Sc}\sqrt{\dfrac{K\kappa_i^2(r_i+\log m+ t)}{n_i}}+\frac{1}{|\Sc|}\sum_{i\in \Sc}\frac{K^{1/2}\kappa_i^2r_i}{n_i}.
\end{align}
If the data on $\Sc$ have symmetric innovation, since the condition $\max_{i\in \Sc}\frac{\sqrt{K}\kappa_i^2r_i}{n_i}\leq c_1$ guarantees $\max_{i\in \Sc}\|\mathbb{E}\hat{A}_i-V_iV_i^T\|_F\leq 1/4$ by Lemma \ref{lemma: thm3 fan}, we can invoke Part (ii) of Lemma \ref{determin:general:error} to drop the last term in \eqref{proof:bound:form1}. 

\item[] Case II: $\delta > \max_{i\in \Sc} \sqrt{\frac{K\kappa_i^2(r_i+\log m+t)}{n_i}}$. We use Proposition \ref{theorem:structure} and \eqref{single:learning:rate:union} to obtain
\begin{align}
\label{single:task:error}
\|\tilde{V}_i\tilde{V}_i^T-V_iV_i^T\|_F&\leq 2\sqrt{2}\|\tilde{A}_i-V_iV_i^T\|_F\leq 2\sqrt{2}\|\tilde{A}_i-\hat{A}_i\|_F+2\sqrt{2}\|\hat{A}_i-V_iV_i^T\|_F \nonumber \\
\leq &2\sqrt{2} \lambda_i+2\sqrt{2}\|\hat{A}_i-V_iV_i^T\|_F \lesssim \max_{i\in \Sc}\sqrt{\dfrac{K\kappa_i^2(r_i+\log m+ t)}{n_i}}.
\end{align}
\end{itemize}
Combining the bounds in both cases completes the proof of Part (i) of Theorem \ref{theorem:upperbound}. Result \eqref{bound:outlier:task} of Part (ii) can be derived similarly as in Case II above. 

\end{proof}

\subsection{Proof of Theorems \ref{theorem:minimax1} and \ref{theorem:minimax2}}
\label{proof:them:2and3}

Throughout this section, we use $\mathbb{V}_{d,K}$ to denote the Stiefel manifold, i.e., $\mathbb{V}_{d,K}=\big\{V\in \mathbb{R}^{d\times K}: V^TV=I_K\big\}$. For two distributions $\mathbb{P}_1,\mathbb{P}_2$, $\kl (\mathbb{P}_1|\mathbb{P}_2  )$ and $\tv(\Pb_1,\Pb_2)$ denote the Kullback–Leibler divergence and total variation distance between $\Pb_1$ and $\Pb_2$, respectively.

\subsubsection{Proof of Theorem \ref{theorem:minimax1}}
\begin{proof}
We first obtain the minimax lower bound over the submodel $\mathcal{M}(\delta,0,\sigma)$. We denote the parameter of the model $\mathcal{M}(\delta,0,\sigma)$ by $\theta=\{V_jV_j^T\}_{j\in [m]}$ and define the pseudo-metric as
\[
\rho(\theta,\tilde{\theta})=\|V_1V_1^T-\tilde{V}_1\tilde{V}_1^T\|_F,~ {\rm~for~}\theta=\{V_jV_j^T\}_{j\in [m]}, \tilde{\theta}=\{\tilde{V}_j\tilde{V}_j^T\}_{j\in [m]}.
\]
Take the subset $\{W_i\}_{i\in [N]}\subseteq \mathbb{V}_{d,K}$ specified in Lemma \ref{lemma:packing construction}, and construct the set of parameters: 
\begin{align*}
\theta^{(i)}=\{V^{(i)}_j(V^{(i)}_j)^T\}_{j\in [m]} {\rm ~with~} V^{(i)}_1(V^{(i)}_1)^T=W_iW_i^T, V^{(i)}_j(V^{(i)}_j)^T=W_1W_1^T, j=2,\ldots, m,
\end{align*}
for $i\in [N]$. Lemma \ref{lemma:packing construction} implies 
\begin{align}
& \log N\geq Kd/3, \label{pack:set:dis}\\
&\rho(\theta^{(l)},\theta^{(k)})=\|W_lW_l^T-W_kW_k^T\|_F\geq c_1\sqrt{K}\Delta, {\rm~for~all~}1\leq l\neq k \leq N, \label{lower:dist:error} \\
& \max_{i\in [N],j,\tilde{j}\in [m]}\|V^{(i)}_j(V^{(i)}_j)^T-V^{(i)}_{\tilde{j}}(V^{(i)}_{\tilde{j}})^T\|_F=\max_{i\in[N]}\|W_1W_1^T-W_iW_i^T\|_F\leq c_2\sqrt{K}\Delta, \label{upper:dist:error}
\end{align}
where $c_1, c_2>0$ are some absolute constants and $\Delta\in (0,1/2]$ is any given constant. We consider the following specific choice of $\Delta$, 
\begin{align}
\label{delta:choice:set}
\Delta=\frac{1}{\sqrt{K}}\Bigg[\frac{\delta}{c_2}\wedge \sqrt{\frac{(\sigma+1)Kd}{10c_2^2n\sigma^2}}\Bigg].
\end{align}
Note that $\Delta \leq \sqrt{\frac{(\sigma+1)d}{10c_2^2n\sigma^2}}\leq 1/2$ as long as $\frac{(\sigma+1)d}{\sigma^2n}\leq 5c_2^2/2$. We have the following observations:
\begin{itemize}
\item \eqref{upper:dist:error}-\eqref{delta:choice:set} show that $\max_{i\in [N],j,\tilde{j}\in [m]}\|V^{(i)}_j(V^{(i)}_j)^T-V^{(i)}_{\tilde{j}}(V^{(i)}_{\tilde{j}})^T\|_F\leq \delta$. Hence, the set $\{\theta^{(i)}\}_{i\in[N]}$ belongs to the parameter space of $\mathcal{M}(\delta,0,\sigma)$.
\item \eqref{upper:dist:error}-\eqref{delta:choice:set} also yield that
\begin{align*}
\kl (\Pb_{\theta^{(l)}}|\Pb_{\theta^{(k)}}  )=\frac{n\sigma^2}{4(1+\sigma)}\|W_lW_l^T-W_kW_k^T\|_F^2\leq \frac{Kd}{40}, {\rm~for~all~}l\neq k.
\end{align*}
\item \eqref{lower:dist:error} and \eqref{delta:choice:set} together give
\begin{align*}
   \min_{1\leq l\neq k\leq N}\rho(\theta^{(l)},\theta^{(k)})\geq \frac{c_1\delta}{c_2}\wedge \sqrt{\frac{c_1^2(\sigma+1)Kd}{10c_2^2n\sigma^2}}.
\end{align*}
So $\{\theta^{(i)}\}_{i\in [N]}$ is a $\frac{c_1\delta}{c_2}\wedge \sqrt{\frac{c_1^2(\sigma+1)Kd}{10c_2^2n\sigma^2}}$-separated set.
\end{itemize}
The above results enable us to apply a standard Fano's argument (see Lemma \ref{lemma: Fano}) to conclude
\begin{align}
&\inf_{\hat{V}_1\in\mathbb{V}_{d,K}} \sup_{\mathbb{P}\in \mathcal{M}(\delta,0,\sigma)} \mathbb{P}\Bigg( \|\hat{V}_1\hat{V}_1^T-\vt{1}\|_F\geq \frac{c_1\delta}{2c_2}\wedge \sqrt{\frac{c_1^2(\sigma+1)Kd}{40c_2^2n\sigma^2}}\Bigg) \nonumber \\
\geq& 1-\frac{\log 2 +Kd/40}{Kd/3}\geq 0.7, \label{no:outlier:model}
\end{align}
where we have used \eqref{pack:set:dis} in the first inequality and the condition $d\geq 10 K$ in the second inequality. Alternatively, we consider $\Delta=\sqrt{\frac{(\sigma+1)d}{10c_2^2nm\sigma^2}}$ and parameter values:
\begin{align}
\label{both:case:suit}
\theta^{(i)}=\{V^{(i)}_j(V^{(i)}_j)^T\}_{j\in [m]} {\rm ~with~} V^{(i)}_j(V^{(i)}_j)^T=W_iW_i^T, j\in [m],i\in [N].
\end{align}
In a similar way, we can verify that $\Delta \leq 1/2$ and the set $\{\theta^{(i)}\}_{i\in [N]}$ belongs to the parameter space of $\mathcal{M}(\delta,0,\sigma)$. Moreover,
\begin{align*}
&\max_{l\neq k}\kl (\Pb_{\theta^{(l)}}|\Pb_{\theta^{(k)}}  )=\frac{mn\sigma^2}{4(1+\sigma)}\|W_lW_l^T-W_kW_k^T\|_F^2\leq \frac{Kd}{40}, \\
&\min_{1\leq l\neq k\leq N}\rho(\theta^{(l)},\theta^{(k)})\geq \sqrt{\frac{c_1^2(\sigma+1)Kd}{10c_2^2nm\sigma^2}}.
\end{align*}
Applying again the Fano's argument (Lemma \ref{lemma: Fano}) shows
\begin{align}
\label{both:case:suit:bound}
\inf_{\hat{V}_1\in\mathbb{V}_{d,K}} \sup_{\mathbb{P}\in \mathcal{M}(\delta,0,\sigma)} \mathbb{P}\Bigg( \|\hat{V}_1\hat{V}_1^T-\vt{1}\|_F\geq  \sqrt{\frac{c_1^2(\sigma+1)Kd}{40c_2^2nm\sigma^2}}\Bigg) \geq 0.7. 
\end{align}
Combining the above with \eqref{no:outlier:model} gives us
\begin{align}
\label{no:outlier:final:bound}
&\inf_{\{\hat{V}_i\}_{i\in \Sc}\subseteq \mathbb{V}_{d,K}} \sup_{\mathbb{P}\in \mathcal{M}(\delta,0,\sigma)} \mathbb{P}\Bigg( \max_{i\in \Sc}\|\hat{V}_i\hat{V}_i^T-\vt{i}\|_F\geq  \frac{c_1\delta}{4c_2}\wedge \sqrt{\frac{c_1^2(\sigma+1)Kd}{160c_2^2n\sigma^2}}+\sqrt{\frac{c_1^2(\sigma+1)Kd}{160c_2^2nm\sigma^2}}\Bigg) \nonumber \\
\geq  &\inf_{\hat{V}_1\in\mathbb{V}_{d,K}} \sup_{\mathbb{P}\in \mathcal{M}(\delta,0,\sigma)} \mathbb{P}\Bigg( \|\hat{V}_1\hat{V}_1^T-\vt{1}\|_F\geq  \frac{c_1\delta}{4c_2}\wedge \sqrt{\frac{c_1^2(\sigma+1)Kd}{160c_2^2n\sigma^2}}+\sqrt{\frac{c_1^2(\sigma+1)Kd}{160c_2^2nm\sigma^2}}\Bigg) \geq 0.7
\end{align}

We next obtain the minimax lower bound over the submodel $\mathcal{M}(0,\epsilon,\sigma)$. Since $\delta=0$, we may write $\Sigma_i=\Sigma=\sigma VV^T+I_d$ for $i\in \Sc$, so that each task $i\in \Sc$ has $n$ i.i.d. samples drawn from $\mathcal{N}(0,\Sigma)$. The model $\mathcal{M}(0,\epsilon,\sigma)$ is closely related to the Huber's $\epsilon$-contamination model: 
\begin{equation}
\mathbb{P}_{(\epsilon,VV^T,Q)}:= (1-\epsilon)\mathcal{N}^n(0,\Sigma)+\epsilon \mathbb{Q},\label{random huber contamination} 
\end{equation}
where $\mathcal{N}^n(0,\Sigma)$ denotes the $n$-fold product of $\mathcal{N}(0,\Sigma)$ and $Q$ is an arbitrary distribution. Theorem 5.1 in \cite{chen2018robust} provides a general minimax lower bound theory for the Huber's $\epsilon$-contamination model. To invoke the theory, let $\Sigma_i=\sigma \vt{i} +I_d, \Pb_i=\mathcal{N}^n(0, \Sigma_i),i=1,2$, and we need to compute 
\begin{align*}
&\sup\left\{ \|\vt{1}-\vt{2}\|_F :V_1,V_2\in \mathbb{V}_{d,K},\tv(\Pb_1,\Pb_2)\leq \frac{\epsilon}{1-\epsilon}\right\} \\
\geq &    \sup\left\{ \| \vt{1}-\vt{2}\|_F: V_1,V_2\in \mathbb{V}_{d,K},\kl(\Pb_1|\Pb_2)\leq \frac{2\epsilon^2}{(1-\epsilon)^2}\right\}\\
=& \epsilon\sqrt{\frac{8(\sigma+1)}{\sigma^2n(1-\epsilon)^2}} \geq \epsilon\sqrt{\frac{8(\sigma+1)}{\sigma^2n}},
\end{align*}
where the first inequality holds since $\tv(\Pb_1,\Pb_2)\leq \sqrt{\frac{1}{2}\kl(\Pb_1|\Pb_2)} $, and the first equality is due to the identity $\kl(\Pb_1,\Pb_2)=\frac{n\sigma ^2}{4(\sigma+1)}\|\vt{1}-\vt{2}\|^2_F$. Then by Theorem 5.1 of \cite{chen2018robust}, it holds that 
\begin{equation}
\label{huber:lower:bound}
 \inf_{\hat{V}\in\mathbb{V}_{d,K}} \sup_{V\in \mathbb{V}_{d,K}}\sup_Q \mathbb{P}_{(\epsilon,VV^T,Q)}\lb \|\hat{V}\hat{V}^T-\vt{}\|_F \geq \epsilon\sqrt{\frac{2(\sigma+1)}{\sigma^2n}} \rb\geq  0.5,
\end{equation}
where $\hat{V}\in\mathbb{V}_{d,K}$ is any measurable function of independent samples from $\mathbb{P}_{(\epsilon,VV^T,Q)}$. Now we aim to convert the above lower bound to one under our model $\mathcal{M}(0,\epsilon,\sigma)$. To this end, let  $b_1,\ldots, b_m$ be i.i.d. Bernoulli random variable with success probability $\epsilon$, $Z_i|b_i=0 \overset{i.i.d.}{\sim} \mathcal{N}^n(0,\Sigma)$ and $Z_i|b_i=1 \overset{i.i.d.}{\sim} Q$. We thus have $Z_1,\ldots, Z_m \overset{i.i.d.}{\sim} \mathbb{P}_{(\epsilon,VV^T,Q)}$. Define the events $H_k=\{ \sum_{i=1}^m b_i=k\}$ and $B=\{ \|\hat{V}\hat{V}^T-VV^T\|_F \geq a\}$ where $\hat{V}$ is a measurable function of $\{Z_i\}_{i\in [m]}$ and $a$ is a fixed constant. We then have
\begin{align*}
    \mathbb{P}(B)=&\sum_{k> 2\epsilon m} \mathbb{P}\left(H_k\right)\mathbb{P}(B|H_k) +\sum_{k\leq 2\epsilon m} \mathbb{P}\left(H_k\right)\mathbb{P}(B|H_k)\\
    \leq &\Pb\left(\sum_{i=1}^m b_i> 2\epsilon m\right)+\sup_{k\leq 2\epsilon m }\mathbb{P}(B|H_k).
\end{align*}
Note that $\mathbb{P}(B)$ is evaluated under $\mathbb{P}=\mathbb{P}_{(\epsilon,VV^T,Q)}$, and $\sup_{k\leq 2\epsilon m }\mathbb{P}(B|H_k)\leq \sup_{\mathbb{P}\in \mathcal{M}(0,2\epsilon,\sigma)} \mathbb{P}(B)$. Hence, the above inequality leads to
\begin{equation}
    \label{convert:middle}
  \sup_{\mathbb{P}\in \mathcal{M}(0,2\epsilon,\sigma)} \mathbb{P}(B) \geq \sup_{V\in \mathbb{V}_{d,K}}\sup_Q \mathbb{P}_{(\epsilon,VV^T,Q)}(B)-\Pb\left(\sumi b_i\geq 2\epsilon m\right).
\end{equation}
By Hoeffding's inequality,
\begin{equation}
\label{hoeffding:eq}
 \Pb\Big(\frac{1}{m}\sum_{i=1}^mb_i-\epsilon\geq t \Big)\leq \exp(-2mt^2) \Longrightarrow  \Pb\Big(\sum_{i=1}^n b_i\geq  2\epsilon m \Big)   \leq \exp(-2m\epsilon^2).  
\end{equation}
\eqref{convert:middle}-\eqref{hoeffding:eq} together show
\begin{equation*}
 \sup_{\mathbb{P}\in \mathcal{M}(0,2\epsilon,\sigma)} \mathbb{P}(B) \geq \sup_{V\in \mathbb{V}_{d,K}}\sup_Q \mathbb{P}_{(\epsilon,VV^T,Q)}(B)-0.0001,
\end{equation*}
when $\epsilon^2m\geq 5$. Taking the infimum over $\hat{V}\in\mathbb{V}_{d,K}$ on both sides of the above inequality and applying \eqref{huber:lower:bound} gives
\begin{align}
\label{outlier:involve:err}
&\inf_{\{\hat{V}_i\}_{i\in \Sc}\subseteq \mathbb{V}_{d,K}} \sup_{\mathbb{P}\in \mathcal{M}(0,\epsilon,\sigma)}  \mathbb{P}\Bigg( \max_{i\in \Sc}\|\hat{V}_i\hat{V}_i^T-\vt{i}\|_F\geq \epsilon\sqrt{\frac{(\sigma+1)}{2\sigma^2n}}\Bigg)\nonumber \\
\geq &\inf_{\hat{V}\in\mathbb{V}_{d,K}} \sup_{\mathbb{P}\in \mathcal{M}(0,\epsilon,\sigma)} \mathbb{P}\Bigg( \|\hat{V}\hat{V}^T-\vt{}\|_F\geq \epsilon\sqrt{\frac{(\sigma+1)}{2\sigma^2n}}\Bigg)\geq 0.4999.
\end{align}
Combining \eqref{no:outlier:final:bound} and \eqref{outlier:involve:err}, we obtain
\begin{align}\label{minimax1:e3}
\inf_{\{\hat{V}_i\}_{i\in \Sc}\subseteq \mathbb{V}_{d,K}} \sup_{\mathbb{P}\in \mathcal{M}(\delta,\epsilon,\sigma)}  \mathbb{P}\Bigg( \max_{i\in \Sc}\|\hat{V}_i\hat{V}_i^T-\vt{i}\|_F\geq  &\frac{c_1\delta}{8c_2}\wedge \sqrt{\frac{c_1^2(\sigma+1)Kd}{640c_2^2n\sigma^2}}\nonumber \\
&+\sqrt{\frac{c_1^2(\sigma+1)Kd}{640c_2^2nm\sigma^2}} +\epsilon\sqrt{\frac{(\sigma+1)}{8\sigma^2n}}\Bigg) \geq 0.4999,
\end{align}
given $\epsilon^2m\geq 5$. When $\epsilon^2m\leq 5$, we have
\begin{equation*}
  \epsilon\sqrt{\frac{(\sigma+1)}{\sigma^2n}}\leq \sqrt{\frac{(\sigma+1)Kd}{2nm\sigma^2}}
\end{equation*}
under the condition $Kd\geq 10K^2\geq 10$. Thus, \eqref{no:outlier:final:bound} directly implies \eqref{minimax1:e3} (by adjusting the absolute constants in the lower bound).

\end{proof}

\subsubsection{Proof of Theorem \ref{theorem:minimax2}}

\begin{proof}

We first obtain the minimax lower bound over the submodel $\mathcal{A}(\delta,0,\sigma)$. As in the proof of Theorem \ref{theorem:minimax1}, we would like to apply the Fano's argument (Lemma \ref{lemma: Fano}). Define the metric
\begin{align}
\label{real:metric:def}
\rho(\theta,\tilde{\theta})=\sqrt{\frac{1}{m}\sum_{j=1}^m\|V_jV_j^T-\tilde{V}_j\tilde{V}_j^T\|_F^2},~ {\rm~for~}\theta=\{V_jV_j^T\}_{j\in [m]}, \tilde{\theta}=\{\tilde{V}_j\tilde{V}_j^T\}_{j\in [m]}.
\end{align}
We again utilize the subset $\{W_i\}_{i\in [N]}\subseteq \mathbb{V}_{d,K}$ from Lemma \ref{lemma:packing construction}. Let $\mathcal{H}=\{h_{l}\}_{l\in M}\subseteq[N]^m$ be the subset stated in Lemma \ref{pack:set:construct:ref}, hence it satisfies
\begin{align}
\label{code:book:set}
M\geq N^{m/4}, \quad \sum_{j=1}^m\mathbb{I}(h_{lj}\neq h_{\tilde{l}j})\geq \lceil m/4 \rceil,~~ \forall l\neq \tilde{l},
\end{align}
as long as $N\geq 4e+1$ and $m\geq 8$. We construct the set of parameter values $\{\theta^{(l)}\}_{l\in [M]}$ as follows:
\begin{align*}
\theta^{(l)}=\{V^{(l)}_j(V^{(l)}_j)^T\}_{j\in [m]} {\rm ~with~} V^{(l)}_j(V^{(l)}_j)^T=W_{h_{lj}}W_{h_{lj}}^T, j\in [m], l\in [M].
\end{align*}
Consider the case $\delta \leq \sqrt{\frac{(\sigma+1)dK}{10n\sigma^2}}$ and set $\Delta=\frac{\delta}{c_2\sqrt{K}}$. Based on Lemma \ref{lemma:packing construction} and \eqref{code:book:set}, we have the following observations:
\begin{itemize}
\item $\Delta\leq 1/2$ under the given condition $\frac{(\sigma+1)d}{\sigma^2n}\leq 5c_2^2/2$. So the choice of $\Delta$ satisfies the requirement of Lemma \ref{lemma:packing construction}.
\item $\max_{l\in [M],j,\tilde{j}\in[m]}\|V^{(l)}_j(V^{(l)}_j)^T-V^{(l)}_{\tilde{j}}(V^{(l)}_{\tilde{j}})^T\|_F\leq \max_{i,j\in [N]}\|W_iW_i^T-W_jW_j^T\|_F\leq c_2\sqrt{K}\Delta=\delta$. Hence, the set $\{\theta^{(l)}\}_{l\in [M]}$ belongs to the parameter space of $\mathcal{A}(\delta,0,\sigma)$.
\item For all $l,k\in [M]$, it holds that 
\begin{align*}
\kl (\Pb_{\theta^{(l)}}|\Pb_{\theta^{(k)}})=&\sum_{j=1}^m\frac{n\sigma^2}{4(1+\sigma)}\|V^{(l)}_j(V^{(l)}_j)^T-V^{(k)}_j(V^{(k)}_j)^T\|_F^2 \\
\leq &\frac{mn\sigma^2}{4(1+\sigma)}\max_{i,j\in [N]}\|W_iW_i^T-W_jW_j^T\|_F^2\leq \frac{nm\sigma^2c_2^2K\Delta^2}{4(1+\sigma)} \\
= & \frac{nm\sigma^2}{4(1+\sigma)}\delta^2\leq \frac{mdk}{40}.
\end{align*}

\item For all $1\leq l \neq k\leq M$, the following hold
\begin{align*}
\rho(\theta^{(l)},\theta^{(k)})=&\sqrt{\frac{1}{m}\sum_{j=1}^m\|W_{h_{lj}}W_{h_{lj}}^T-W_{h_{kj}}W_{h_{kj}}^T\|_F^2} \\
\geq &\frac{1}{2}\min_{1\leq i\neq j\leq N}\|W_iW_i^T-W_jW_j^T\|_F \geq \frac{c_1\sqrt{K}\Delta}{2}=\frac{c_1\delta}{2c_2},
\end{align*}
where the first inequality is due to \eqref{code:book:set}.
\end{itemize}
Given the above results, we are ready to apply Lemma \ref{lemma: Fano} to conclude 
\begin{align}
\label{first:fano:type:br}
&\inf_{\{\hat{V}_i\}_{i\in [m]}\subseteq \mathbb{V}_{d,K}} \sup_{\mathbb{P}\in \mathcal{A}(
\delta,0,\sigma)}  \mathbb{P}\Bigg( \sqrt{\frac{1}{m}\sum_{i=1}^m\|\hat{V}_i\hat{V}_i^T-\vt{i}\|_F^2}\geq \frac{c_1\delta}{4c_2}\Bigg) \nonumber \\
\geq&1-\frac{\log 2+mdK/40}{\log M}\geq 1-\frac{\log 2+mdK/40}{mdK/12}\geq 0.59,
\end{align}
given $\delta \leq \sqrt{\frac{(\sigma+1)dK}{10n\sigma^2}}, m\geq 8, N\geq 4e+1, dK\geq 10$. Note that $dK\geq 10$ implies $N\geq 4e+1$ since $\log N\geq dK/3$. Moreover, with exactly the same set of parameter values in \eqref{both:case:suit} and the choice of $\Delta=\sqrt{\frac{(\sigma+1)d}{10c_2^2nm\sigma^2}}$ from the proof of Theorem \ref{theorem:minimax1}, it is straightforward to verify that the minimax lower bound \eqref{both:case:suit:bound} continues to hold here,
\begin{align}
\label{second:fano:type:br}
\inf_{\{\hat{V}_i\}_{i\in [m]}\subseteq \mathbb{V}_{d,K}}  \sup_{\mathbb{P}\in \mathcal{A}(\delta,0,\sigma)} \mathbb{P}\Bigg( \sqrt{\frac{1}{m}\sum_{i=1}^m\|\hat{V}_1\hat{V}_i^T-\vt{i}\|_F^2}\geq  \sqrt{\frac{c_1^2(\sigma+1)Kd}{40c_2^2nm\sigma^2}}\Bigg) \geq 0.7. 
\end{align}
Combining \eqref{first:fano:type:br}-\eqref{second:fano:type:br} yields
\[
\inf_{\{\hat{V}_i\}_{i\in [m]}\subseteq \mathbb{V}_{d,K}}  \sup_{\mathbb{P}\in \mathcal{A}(\delta,0,\sigma)} \mathbb{P}\Bigg( \sqrt{\frac{1}{m}\sum_{i=1}^m\|\hat{V}_1\hat{V}_i^T-\vt{i}\|_F^2}\geq  \frac{c_1\delta}{8c_2}+\sqrt{\frac{c_1^2(\sigma+1)Kd}{160c_2^2nm\sigma^2}}\Bigg) \geq 0.59,
\]
given $\delta \leq \sqrt{\frac{(\sigma+1)dK}{10n\sigma^2}}, m\geq 8$. When $\delta >\sqrt{\frac{(\sigma+1)dK}{10n\sigma^2}}:=\bar{\delta}$, since $\mathcal{A}(\bar{\delta},0,\sigma)\subseteq \mathcal{A}(\delta,0,\sigma)$, the minimax lower bound under $\mathcal{A}(\delta,0,\sigma)$ is larger than that under $\mathcal{A}(\bar{\delta},0,\sigma)$. We can thus conclude
\begin{align}
\label{final:zeroep:bound}
\inf_{\{\hat{V}_i\}_{i\in [m]}\subseteq \mathbb{V}_{d,K}}  \sup_{\mathbb{P}\in \mathcal{A}(\delta,0,\sigma)} \mathbb{P}\Bigg( \sqrt{\frac{1}{m}\sum_{i=1}^m\|\hat{V}_1\hat{V}_i^T-\vt{i}\|_F^2}\geq & \frac{c_1\delta}{8c_2} \wedge \sqrt{\frac{c_1^2(\sigma+1)dK}{640c_2^2n\sigma^2}} \nonumber \\
&+\sqrt{\frac{c_1^2(\sigma+1)Kd}{160c_2^2nm\sigma^2}}\Bigg) \geq 0.59,
\end{align}
given $m\geq 8$. When $m<8$, \eqref{second:fano:type:br} directly implies \eqref{final:zeroep:bound} (after properly adjusting the absolute constants in \eqref{final:zeroep:bound}), hence \eqref{final:zeroep:bound} in fact holds for all $m\geq 1$.

We next obtain the minimax lower bound over the submodel $\mathcal{A}(0,\epsilon,\sigma)$. The arguments are similar to those in the preceding proof. Let $\tilde{\mathcal{H}}=\{\tilde{h}_{l}\}_{l\in \tilde{M}}\subseteq[N]^{\epsilon m}$\footnote{For simplicity, we treat $\epsilon m$ as an integer. The non-integer case can be handled in a similar way.} be the subset stated in Lemma \ref{pack:set:construct:ref}. Adopt the same metric \eqref{real:metric:def}, set $\Delta=\sqrt{\frac{(\sigma+1)d}{10c_2^2\sigma^2n}}$, and consider the set $\{\theta^{(l)}\}_{l\in[\tilde{M}]}$:
\begin{align*}
\theta^{(l)}=\{V^{(l)}_j(V^{(l)}_j)^T\}_{j\in [m]} &{\rm ~with~} V^{(l)}_j(V^{(l)}_j)^T=W_{\tilde{h}_{lj}}W_{\tilde{h}_{lj}}^T, j\in [\epsilon m], \\
&{\rm and~} V^{(l)}_j(V^{(l)}_j)^T=W_1W_1^T, j=\epsilon m+1,\ldots m.
\end{align*}
Like before, we can directly confirm the following:
\begin{itemize}
\item $\Delta\leq 1/2$ under the given condition $\frac{(\sigma+1)d}{\sigma^2n}\leq 5c_2^2/2$. So the choice of $\Delta$ satisfies the requirement of Lemma \ref{lemma:packing construction}.
\item Let $\Sc=\{\epsilon m+1,\epsilon m+2,\ldots, m\}$. Then, $\max_{l\in [\tilde{M}],j,\tilde{j}\in \Sc}\|V^{(l)}_j(V^{(l)}_j)^T-V^{(l)}_{\tilde{j}}(V^{(l)}_{\tilde{j}})^T\|_F=0$. Hence, the set $\{\theta^{(l)}\}_{l\in [\tilde{M}]}$ belongs to the parameter space of $\mathcal{A}(0,\epsilon,\sigma)$.
\item For all $l,k\in [\tilde{M}]$, it holds that 
\begin{align*}
\kl (\Pb_{\theta^{(l)}}|\Pb_{\theta^{(k)}})=&\sum_{j=1}^m\frac{n\sigma^2}{4(1+\sigma)}\|V^{(l)}_j(V^{(l)}_j)^T-V^{(k)}_j(V^{(k)}_j)^T\|_F^2 \\
\leq &\frac{\epsilon mn\sigma^2}{4(1+\sigma)}\max_{i,j\in [N]}\|W_iW_i^T-W_jW_j^T\|_F^2\leq \frac{n\epsilon m\sigma^2c_2^2K\Delta^2}{4(1+\sigma)} =\frac{\epsilon mK d}{40}.
\end{align*}

\item For all $1\leq l \neq k\leq \tilde{M}$, the following hold
\begin{align*}
\rho(\theta^{(l)},\theta^{(k)})=&\sqrt{\frac{1}{m}\sum_{j=1}^m\|W_{h_{lj}}W_{h_{lj}}^T-W_{h_{kj}}W_{h_{kj}}^T\|_F^2} \\
\geq &\frac{\sqrt{\epsilon}}{2}\min_{1\leq i\neq j\leq N}\|W_iW_i^T-W_jW_j^T\|_F \geq \sqrt{\frac{c_1^2\epsilon (\sigma+1)Kd}{40c_2^2\sigma^2 n}}.
\end{align*}
\end{itemize}
Therefore, Lemma \ref{lemma: Fano} gives
\begin{align*}
&\inf_{\{\hat{V}_i\}_{i\in [m]}\subseteq \mathbb{V}_{d,K}} \sup_{\mathbb{P}\in \mathcal{A}(
0,\epsilon,\sigma)}  \mathbb{P}\Bigg( \sqrt{\frac{1}{m}\sum_{i=1}^m\|\hat{V}_i\hat{V}_i^T-\vt{i}\|_F^2}\geq \sqrt{\frac{c_1^2\epsilon (\sigma+1)Kd}{160c_2^2\sigma^2 n}}\Bigg) \nonumber \\
\geq&1-\frac{\log 2+\epsilon mdK/40}{\epsilon mKd/12} \geq 0.59,
\end{align*}
given $\epsilon m\geq 8$. Combining the above bound with \eqref{final:zeroep:bound}, we obtain
\begin{align}
    \inf_{\{\hat{V}_i\}_{i=1}^m} \sup_{\Pb \in \mathcal{A}(\delta, \epsilon,\sigma)} \Pb \Bigg(&\frac{1}{m}\sum_{i=1}^m \|\hat{V}_i\hat{V}_i^T-\vt{i}\|_F^2 \nonumber \\
    &\geq C \lb\frac{(\sigma+1)Kd}{\sigma^2nm}+\delta^2\wedge\frac{(\sigma+1)Kd}{\sigma^2n}+\frac{(\sigma+1)\epsilon Kd}{\sigma^2n}\rb\Bigg)\geq 0.59, \label{final:aver:together}
\end{align}
where $C>0$ is an absolute constant, and $\epsilon m\geq 8$. When $\epsilon m < 8$, it is clear that \eqref{final:zeroep:bound} directly implies \eqref{final:aver:together} (by adjusting the constant $C$). The proof is thus completed. 
\end{proof}

\subsubsection{Technical lemmas}
The following lemma is often used to obtain minimax lower bounds, which can be found in standard textbooks (e.g., Proposition 15.12 in \cite{wainwright2019high}).
\begin{lemma}\label{lemma: Fano}
Consider a distribution class $\mathbb{P}=\{\Pb_{\theta}\mid \theta \in \Theta\}$ where $\theta$ is some parameter.
Let $\left\{\theta_1,\ldots,\theta_M\right\}$ be a $2\delta$-separated set in a pseudo-metric $\rho$ on $\Theta$.
Then  the minimax risk is lower bounded as 
\begin{equation*}
  \inf_{\hat{\theta}}\sup_{\Pb_{\theta}\in \mathbb{P}} \Pb_{\theta}(\rho(\hat{\theta},\theta)\geq \delta)\geq 1-\frac{\frac{1}{M^2}\sum_{i,j=1}^M \kl(\Pb_{\theta_i}|\Pb_{\theta_j})+\log 2}{\log M},
\end{equation*}
where $\kl(\Pb_{\theta_i}|\Pb_{\theta_j})$ denotes the Kullback–Leibler divergence between $\Pb_{\theta_i}$ and $\Pb_{\theta_j}$.
   
\end{lemma}

\begin{lemma}\label{lemma:packing construction}
Assume $1\leq K \leq d/3$. For any given $\Delta\in (0,1/2]$, there exists a subset $\{W_i\}_{i\in [N]}  \subseteq \mathbb{V}_{d, K}$ satisfying the following properties:
\begin{itemize}
    \item[1.] $\log N\geq Kd/3$,
    \item[2.] For $1\leq i \neq j\leq N$, $c_1\sqrt{K}\Delta\leq \|W_iW_i^T-W_jW_j^T\|_F\leq c_2\sqrt{K}\Delta$,
\end{itemize}
where $c_1,c_2>0$ are some absolute constants.
\end{lemma}
\begin{proof}

According to Lemma 6 from \cite{vu2013minimax} (see also Proposition 8 in \cite{pajor1998metric}) under the condition $1\leq K \leq d/3$, there exists a subset $\{J_i\}_{i\in [N]}\subseteq \mathbb{V}_{d-K, K}$ satisfying
\begin{align}
&\log N \geq Kd/3, \label{entropy:bound:one}\\
&\|J_iJ_i^T-J_jJ_j^T\|_F\geq \tilde{c}_1\sqrt{K} {\rm~~ for~all~} i\neq j,
\end{align}
where $\tilde{c}_1>0$ is an absolute constant. Define 
\begin{equation*}
W_i
=
\left[
\begin{array}{cc}
(1 - \Delta^2)^{1/2} I_K  \\
\Delta J_i 
\end{array}
\right]\in \mathbb{R}^{d\times K}, \quad i\in [N].
\end{equation*}
It is straightforward to confirm that $\{W_i\}_{i\in [N]}\subseteq \mathbb{V}_{d, K}$. Moreover, we apply Lemma 3 in \cite{vu2013minimax} to obtain 
\begin{align}
\label{convert:form:one}
\sqrt{2}\Delta \sqrt{1-\Delta^2}\|J_i-J_j\|_F &\leq \|W_iW_i^T-W_jW_j^T\|_F\leq \sqrt{2}\Delta \|J_i-J_j\|_F \nonumber \\
&\leq \sqrt{2}\Delta(\|J_i\|_F+\|J_j\|_F)\leq 2\sqrt{2}\sqrt{K}\Delta,
\end{align}
and Proposition 2.2 in \cite{vu2013minimax} to have
\begin{align}
\label{convert:form:two}
\sqrt{2}\Delta \sqrt{1-\Delta^2}\|J_i-J_j\|_F\geq \sqrt{3/2}\Delta \|J_i-J_j\|_F\geq \sqrt{3}/2\Delta\|J_iJ_i^T-J_jJ_j^T\|_F.
\end{align}
Combining \eqref{entropy:bound:one}-\eqref{convert:form:two} completes the proof. 
\end{proof}

\begin{lemma}
\label{pack:set:construct:ref}
Consider integers $N, m$ with $N\geq 4e+1$ and $m\geq 8$. Then there exists a subset $\mathcal{H}\subseteq [N]^m$ satisfying the following properties:
\begin{itemize}
\item[(i)] $|\mathcal{H}|\geq N^{m/4}$,
\item[(ii)] $\sum_{i=1}^m\mathbb{I}(h_i\neq \tilde{h}_i)\geq \lceil m/4 \rceil $ for any $h\neq \tilde{h}\in \mathcal{H}$.
\end{itemize}
\end{lemma}

\begin{proof}
The Gilbert-Varshamov bound states that there exists a subset $\mathcal{H}\subseteq [N]^m$ such that Part (ii) is satisfied and 
\begin{align}
\label{gv:explicit}
|\mathcal{H}|\geq \frac{N^m}{\sum_{j=0}^{\lceil m/4 \rceil -1}{m\choose j}(N-1)^j}.
\end{align}
We aim to simplify the above lower bound. Based on the binomial coefficient bounds $\big(\frac{n}{k}\big)^k\leq {n \choose k}\leq \big(\frac{en}{k}\big)^k$ for $1\leq k\leq n$, we obtain the following bounds
\begin{align*}
\sum_{j=0}^{\lceil m/4 \rceil -1}{m\choose j}(N-1)^j&\leq 1+\sum_{j=1}^{\lceil m/4 \rceil}\Big(\frac{em}{j}\Big)^j(N-1)^j\leq 1+\sum_{j=1}^{\lceil m/4 \rceil}\Big(\frac{\lceil m/4 \rceil }{j}\Big)^j[4e(N-1)\big]^j \\
&\leq 1+\sum_{j=1}^{\lceil m/4 \rceil}{\lceil m/4 \rceil \choose j}[4e(N-1)\big]^j=[4e(N-1)+1\big]^{\lceil m/4 \rceil} \leq N^{3m/4},
\end{align*}
where the last inequality is by the conditions $N\geq 4e+1$ and $m\geq 8$. The above result combined with \eqref{gv:explicit} completes the proof.
\end{proof}

\subsection{Proof of Theorem \ref{theorem:depth space}}
\label{proof:depth:method:bound}

Recall the depth function defined in \eqref{new:depth:def}. We define a few related quantities in the following:
\begin{align}
   \mathcal{D}\left(\Gamma, \{X_i\}_{i\in S} \right) =
\inf_{\|u\|_2=1} \min &\Bigg\{
\frac{1}{|S|} \sum_{i\in S} \frac{1}{B_i}\sum_{j=1}^{B_i}\mathbb{I} \Big\{ \frac{u^T(X_i^{(j)})^TX_i^{(j)}u}{n_{ij}}\leq u^{T} \Gamma u \Big\}, \;  \nonumber \\
&
~~~\frac{1}{|S|} \sum_{i\in S} \frac{1}{B_i}\sum_{j=1}^{B_i}\mathbb{I} \Big\{ \frac{u^T(X_i^{(j)})^TX_i^{(j)}u}{n_{ij}}> u^{T} \Gamma u \Big\}
\Bigg\},  \label{fn:s:def}\\
   \mathcal{D}\left(\Gamma, \{\mathbb{P}_i\}_{i\in S} \right) =
\inf_{\|u\|_2=1} \min &\Bigg\{
\frac{1}{|S|} \sum_{i\in S} \frac{1}{B_i}\sum_{j=1}^{B_i}\mathbb{P} \Big\{ \frac{u^T(X_i^{(j)})^TX_i^{(j)}u}{n_{ij}}\leq u^{T} \Gamma u \Big\}, \;  \nonumber  \\
&
~~~\frac{1}{|S|} \sum_{i\in S} \frac{1}{B_i}\sum_{j=1}^{B_i}\mathbb{P} \Big\{ \frac{u^T(X_i^{(j)})^TX_i^{(j)}u}{n_{ij}}> u^{T} \Gamma u \Big\}
\Bigg\},  \label{f:pop:def} \\
   \tilde{\mathcal{D}}\left(\Gamma, \{\mathbb{P}_i\}_{i\in S} \right) =
\inf_{\|u\|_2=1} \min &\Bigg\{
\frac{1}{|S|} \sum_{i\in S} \frac{1}{B_i}\sum_{j=1}^{B_i}\mathbb{P} \Big\{ \frac{u^T(X_i^{(j)})^TX_i^{(j)}u}{n_{ij}u^T\Sigma_iu}\leq \frac{u^{T} \Gamma u}{u^T\Sigma u} \Big\}, \;  \nonumber \\
&
~~~\frac{1}{|S|} \sum_{i\in S} \frac{1}{B_i}\sum_{j=1}^{B_i}\mathbb{P} \Big\{ \frac{u^T(X_i^{(j)})^TX_i^{(j)}u}{n_{ij}u^T\Sigma u}> \frac{u^{T} \Gamma u}{u^T\Sigma u} \Big\} \label{g:pop:def}
\Bigg\},
\end{align}
where $\Sigma \succ 0$ is any given positive definite matrix. Since the depth-based estimator $\hat{\Gamma}$ in \eqref{cov:center} plays an important role in the depth-based multi-task learning procedure, we first prove an error bound for $\hat{\Gamma}$ below.

\subsubsection{A general bound for $\hat{\Gamma}$}

Introduce the notation:
\[
\underline{n}=\min_{i\in S}n_i,~\bar{n}=\max_{i\in S}n_i, ~\underline{B}=\min_{i\in S}B_i,~\bar{B}=\max_{i\in S}B_i.
\]

\begin{theorem}\label{theorem:depth}
Let Assumption \ref{assumption:data simi} hold. Additionally, we assume $\epsilon\leq \frac{1}{4}$, $\frac{\underline{n}}{\bar{B}}\geq C_1$ for sufficiently large constant $C_1$, and there exists $\Sigma\succ 0$ such that 
\begin{align}
\sqrt{\frac{\bar{n}}{\underline{B}}}\max_{i\in S}\frac{\|\Sigma-\Sigma_i\|}{\sigma_{min}(\Sigma_i)\wedge \sigma_{min}(\Sigma)}+\sqrt{\frac{d+t}{|S|\underline{B}}}\leq C_2, \label{con:small:gap}
\end{align}
for some sufficiently small constant $C_2>0$.
\begin{itemize}
\item[(i)] If the rows of $X_i$ are Gaussian for all $i\in S$, then it holds with probability at least $1-\exp(-t)$ that
\begin{align*}
\|\hat{\Gamma}-\beta\Sigma\|\leq C_3e^{C_4(\frac{\bar{n}\bar{B}}{\underline{n}\underline{B}})^{\frac{3}{2}}} \|\Sigma\|\cdot\Bigg[\max_{i\in S}\frac{\|\Sigma-\Sigma_i\|}{\sigma_{min}(\Sigma_i)\wedge \sigma_{min}(\Sigma)}+\sqrt{\frac{\underline{B}}{\bar{n}}}\epsilon+\sqrt{\frac{d+t}{|S|\bar{n}}}\Bigg],
\end{align*}
where $\beta \in [\frac{1}{2},\frac{3}{2}]$ depends on the sample size $\{n_i\}_{i\in S}$, and $C_3,C_4>0$ are absolute constants.
\item[(ii)] If the rows of $X_i$ are elliptical with finite sixth moments for all $i\in S$, then it holds with probability at least $1-\exp(-t)$ that
\begin{align*}
\|\hat{\Gamma}-\Sigma\|\leq C_5\|\Sigma\|\cdot\Bigg[\sqrt{\frac{\bar{n}\bar{B}}{\underline{n}\underline{B}}}\max_{i\in S}\frac{\|\Sigma-\Sigma_i\|}{\sigma_{min}(\Sigma_i)\wedge \sigma_{min}(\Sigma)}+\sqrt{\frac{\bar{B}}{\underline{n}}}\epsilon+\frac{\bar{B}}{\underline{n}}+\sqrt{\frac{(d+t)\bar{B}}{|S|\underline{n}\underline{B}}}\Bigg],
\end{align*}
where $C_5>0$ depends on the six moments of the elliptical distributions. 
\end{itemize}
\end{theorem}

\begin{proof}
For notational simplicity, we assume $N_i:=\frac{n_i}{B_i}$ is an integer for all $i\in S$ throughout the proof. The non-integer case can be handled in a similar way. We use $c_i,i=1,2,\ldots$ to denote absolute constants.

Proof of Part (i). Let $F_{N_i}(\cdot)$ denote the CDF of $\frac{1}{N_i}\chi^2(N_i)$, and $\beta$ satisfy 
\begin{align}
\label{gauss:beta:def}
\frac{1}{|S|}\sum_{i\in S}F_{N_i}(\beta)=\frac{1}{2}.
\end{align}
It is straightforward to verify the following for $\tilde{\mathcal{D}}\left(\Gamma, \{\mathbb{P}_i\}_{i\in S} \right) $ in \eqref{g:pop:def}:
\begin{align}
 \tilde{\mathcal{D}}\left(\Gamma, \{\mathbb{P}_i\}_{i\in S} \right) &=\inf_{\|u\|_2=1}\min\Bigg\{\frac{1}{|S|}\sum_{i\in S}F_{N_i}\Big(\frac{u^T\Gamma u}{u^T\Sigma u}\Big), 1-\frac{1}{|S|}\sum_{i\in S}F_{N_i}\Big(\frac{u^T\Gamma u}{u^T\Sigma u}\Big)\Bigg\}, \label{g:reform} \\
\beta\Sigma &=\arg\max_{\Gamma \succeq 0} \tilde{\mathcal{D}}\left(\Gamma, \{\mathbb{P}_i\}_{i\in S} \right). \label{g:maximizer}
\end{align}
We have the following useful decomposition,
\begin{align}
\label{key:decomp:D}
&\tilde{\mathcal{D}}\left(\beta\Sigma, \{\mathbb{P}_i\}_{i\in S} \right) -\tilde{\mathcal{D}}(\hat{\Gamma}, \{\mathbb{P}_i\}_{i\in S} )  \nonumber \\
=&\Big[\tilde{\mathcal{D}}(\beta\Sigma, \{\mathbb{P}_i\}_{i\in S} )-\mathcal{D}\left(\beta\Sigma, \{\mathbb{P}_i\}_{i\in S} \right)\Big]+\Big[\mathcal{D}(\hat{\Gamma}, \{\mathbb{P}_i\}_{i\in S} )-\tilde{\mathcal{D}}(\hat{\Gamma}, \{\mathbb{P}_i\}_{i\in S} )\Big] \nonumber \\
&+\Big[\mathcal{D}(\beta\Sigma, \{\mathbb{P}_i\}_{i\in S} )-\mathcal{D}(\beta\Sigma, \{X_i\}_{i\in S} )\Big]+\Big[\mathcal{D}(\hat{\Gamma}, \{X_i\}_{i\in S} )-\mathcal{D}(\hat{\Gamma}, \{\mathbb{P}_i\}_{i\in S} )\Big] \nonumber\\
& +\Big[\mathcal{D}(\beta\Sigma, \{X_i\}_{i\in S} )-\mathcal{D}(\hat{\Gamma}, \{X_i\}_{i\in S} )\Big] \nonumber \\
\leq & 2\sup_{\Gamma \succeq 0} \Big|\mathcal{D}(\Gamma, \{\mathbb{P}_i\}_{i\in S} )-\tilde{\mathcal{D}}(\Gamma, \{\mathbb{P}_i\}_{i\in S} )\Big|+2\sup_{\Gamma \succeq 0} \Big|\mathcal{D}(\Gamma, \{X_i\}_{i\in S})-\mathcal{D}(\Gamma, \{\mathbb{P}_i\}_{i\in S} )\Big| \nonumber \\
&+\Big[\mathcal{D}(\beta\Sigma, \{X_i\}_{i\in S} )-\mathcal{D}(\hat{\Gamma}, \{X_i\}_{i\in S} )\Big],
\end{align}
where  $\mathcal{D}\left(\Gamma, \{X_i\}_{i\in S}  \right)$ and $\mathcal{D}\left(\Gamma, \{\mathbb{P}_i\}_{i\in S} \right)$ are introduced in \eqref{fn:s:def} and \eqref{f:pop:def} respectively. We now bound each of the three terms in \eqref{key:decomp:D}. Note that $\mathcal{D}\left(\Gamma, \{\mathbb{P}_i\}_{i\in S} \right) $ can be rewritten as 
\[
 \mathcal{D}\left(\Gamma, \{\mathbb{P}_i\}_{i\in S} \right) =\inf_{\|u\|_2=1}\min\Bigg\{\frac{1}{|S|}\sum_{i\in S}F_{N_i}\Big(\frac{u^T\Gamma u}{u^T\Sigma_i u}\Big), 1-\frac{1}{|S|}\sum_{i\in S}F_{N_i}\Big(\frac{u^T\Gamma u}{u^T\Sigma_i u}\Big)\Bigg\}.
\]
Hence,
\begin{align}
&\sup_{\Gamma \succeq 0} \Big|\mathcal{D}(\Gamma, \{\mathbb{P}_i\}_{i\in S} )-\tilde{\mathcal{D}}(\Gamma, \{\mathbb{P}_i\}_{i\in S} )\Big| \nonumber\\
\leq &\sup_{\Gamma \succeq 0, \|u\|_2=1}\frac{1}{|S|}\sum_{i\in S}\Big|F_{N_i}\Big(\frac{u^T\Gamma u}{u^T\Sigma_i u}\Big)-F_{N_i}\Big(\frac{u^T\Gamma u}{u^T\Sigma u}\Big)\Big| \nonumber\\
=&\sup_{\Gamma \succeq 0, \|u\|_2=1}\frac{1}{|S|}\sum_{i\in S}\Bigg[\Big(u^T\Gamma u\Big(\frac{t_i}{u^T\Sigma_i u}+\frac{1-t_i}{u^T\Sigma u}\Big)\Big)f_{N_i}\Big(u^T\Gamma u\Big(\frac{t_i}{u^T\Sigma_i u}+\frac{1-t_i}{u^T\Sigma u}\Big)\Big)\Bigg] \cdot \frac{\Big|\frac{1}{u^T\Sigma_i u}-\frac{1}{u^T\Sigma u}\Big|}{\Big(\frac{t_i}{u^T\Sigma_i u}+\frac{1-t_i}{u^T\Sigma u}\Big)} \nonumber \\
\leq & \sup_{\|u\|_2=1}\frac{1}{|S|}\sum_{i\in S} c_1\sqrt{N_i}\frac{|u^T(\Sigma_i-\Sigma)u|}{u^T\Sigma_i u \wedge u^T\Sigma u}\leq c_1 \sqrt{\frac{\bar{n}}{\underline{B}}}\max_{i\in S}\frac{\|\Sigma_i-\Sigma\|}{\sigma_{min}(\Sigma_i)\wedge \sigma_{min}(\Sigma)}. \label{decom:part1:bound}
\end{align}
Here, the equality is by mean value theorem and $f_{N_i}$ denotes the density function of $\frac{1}{N_i}\chi^2(N_i)$, and the second inequality is due to Part (i) of Lemma \ref{Lemma:lbd for derivative}. The second term in \eqref{key:decomp:D} has been bounded in Lemma \ref{lemma:EPbound}. Regarding the third one, we have
\begin{align}
\label{decom:part3:bound}
&\mathcal{D}(\beta\Sigma, \{X_i\}_{i\in S} )-\mathcal{D}(\hat{\Gamma}, \{X_i\}_{i\in S} )\nonumber \\
\leq &\mathcal{D}(\beta\Sigma, \{X_i\}_{i\in S} )-\frac{m}{|S|}\mathcal{D}(\hat{\Gamma}, \{X_i\}_{i=1}^m )+\frac{m-|S|}{|S|} \nonumber\\
\leq & \mathcal{D}(\beta\Sigma, \{X_i\}_{i\in S} )-\frac{m}{|S|}\mathcal{D}(\beta\Sigma, \{X_i\}_{i=1}^m )+\frac{m-|S|}{|S|}\leq \frac{\epsilon}{1-\epsilon},
\end{align}
because $m\mathcal{D}(\Gamma, \{X_i\}_{i=1}^m )-m+|S|\leq |S| \mathcal{D}(\Gamma, \{X_i\}_{i\in S} )\leq m\mathcal{D}(\Gamma, \{X_i\}_{i=1}^m ),\forall \Gamma \succeq 0$ and $\hat{\Gamma}$ is the maximizer of $\mathcal{D}(\Gamma, \{X_i\}_{i=1}^m )$. Combining \eqref{key:decomp:D}-\eqref{decom:part3:bound} and Lemma \ref{lemma:EPbound}, we obtain that with probability at least $1-\exp(t)$,
\begin{align*}
&\tilde{\mathcal{D}}\left(\beta\Sigma, \{\mathbb{P}_i\}_{i\in S} \right) -\tilde{\mathcal{D}}(\hat{\Gamma}, \{\mathbb{P}_i\}_{i\in S} ) \nonumber \\
\leq &c_2 \sqrt{\frac{\bar{n}}{\underline{B}}}\max_{i\in S}\frac{\|\Sigma-\Sigma_i\|}{\sigma_{min}(\Sigma_i)\wedge \sigma_{min}(\Sigma)}+c_3\sqrt{\frac{d+t}{|S|\underline{B}}}+\frac{\epsilon}{1-\epsilon}.
\end{align*}
Since $\tilde{\mathcal{D}}\left(\beta\Sigma, \{\mathbb{P}_i\}_{i\in S} \right)=\frac{1}{2}$, the above inequality can be equivalently written as 
\begin{align}
\label{fun:err:bound}
& \Big|\frac{1}{2}-\frac{1}{|S|}\sum_{i\in S}F_{N_i}\Big(\frac{u^T\hat{\Gamma} u}{u^T\Sigma u}\Big)\Big|=\Big|\frac{1}{|S|}\sum_{i\in S}\Big(F_{N_i}\Big(\frac{u^T\hat{\Gamma} u}{u^T\Sigma u}\Big)-F_{N_i}(\beta)\Big)\Big|\nonumber \\
\leq &c_2 \sqrt{\frac{\bar{n}}{\underline{B}}}\max_{i\in S}\frac{\|\Sigma-\Sigma_i\|}{\sigma_{min}(\Sigma_i)\wedge \sigma_{min}(\Sigma)}+c_3\sqrt{\frac{d+t}{|S|\underline{B}}}+\frac{\epsilon}{1-\epsilon}, \quad \forall u{\rm ~with~} \|u\|_2=1.
\end{align}
Then, under the conditions $\epsilon\leq \frac{1}{4}$ and \eqref{con:small:gap}, we can have
\begin{align}
\label{slow:rate:part}
\Big|\frac{1}{2}-\frac{1}{|S|}\sum_{i\in S}F_{N_i}\Big(\frac{u^T\hat{\Gamma} u}{u^T\Sigma u}\Big)\Big|\leq \frac{5}{12}.
\end{align}
Referring to the concentration inequality for chi-squared random variable (e.g. Example 2.11 in \cite{wainwright2019high}):
\[
\mathbb{P}\Big(\Big|\frac{1}{N_i}\chi^2(N_i)-1\Big|\geq t\Big)\leq 2\exp\Big(-\frac{N_it^2}{8}\Big), ~~\forall t\in (0,1),
\]
we can conclude that 
\begin{align}
\label{concen:around1}
\Big|1-\frac{u^T\hat{\Gamma} u}{u^T\Sigma u}\Big|\leq \sqrt{\frac{8\bar{B}\log 25}{\underline{n}}}\leq \frac{1}{2},
\end{align}
as long as $\frac{\underline{n}}{\bar{B}}\geq 32\log 25$. Otherwise, $\frac{1}{|S|}\sum_{i\in S}F_{N_i}\Big(\frac{u^T\hat{\Gamma} u}{u^T\Sigma u}\Big)$ would be either too large or too small due to the concentration, hence contradicting with \eqref{slow:rate:part}. Moreover, Lemma \ref{lemma:median} shows that 
\begin{align}
\label{concen:beta:one}
|\beta-1|\leq \frac{c_4\bar{B}}{\underline{n}}\leq \frac{1}{2},
\end{align}
when $\frac{\underline{n}}{\bar{B}}$ is lower bounded by a large enough constant. Based on \eqref{concen:around1}-\eqref{concen:beta:one}, we apply mean value theorem to continue from \eqref{fun:err:bound},
\begin{align*}
&c_2 \sqrt{\frac{\bar{n}}{\underline{B}}}\max_{i\in S}\frac{\|\Sigma-\Sigma_i\|}{\sigma_{min}(\Sigma_i)\wedge \sigma_{min}(\Sigma)}+c_3\sqrt{\frac{d+t}{|S|\underline{B}}}+\frac{\epsilon}{1-\epsilon}\\
\geq &
\Big|\frac{u^T\hat{\Gamma} u}{u^T\Sigma u}-\beta\Big|\cdot \frac{1}{|S|}\sum_{i\in S}f_{N_i}\Big(t_i\frac{u^T\hat{\Gamma} u}{u^T\Sigma u}+(1-t_i)\beta\Big), \quad t_i\in [0,1] \\
\geq & c_5 e^{-c_6(\frac{\bar{n}\bar{B}}{\underline{n}\underline{B}})^{\frac{3}{2}}} \sqrt{\frac{\underline{n}}{\bar{B}}} \Big|\frac{u^T\hat{\Gamma} u}{u^T\Sigma u}-\beta\Big|,  \quad \forall u{\rm ~with~} \|u\|_2=1,
\end{align*}
where the second inequality is by Part (ii) of Lemma \ref{Lemma:lbd for derivative}. Multiplying both sides of the above inequality and using $u^T\Sigma u\leq \|\Sigma\|$ finishes the proof of Part (i).

For Part (ii), we first define for each $i\in S$,
\begin{align*}
z_{ij}=\frac{u^TX_{i,j}}{\sqrt{u^T\Sigma_iu}}, ~~\sigma_i^2={\rm Var}(z_{ij}^2), ~~\tau_i=\frac{\mathbb{E}|z_{ij}^2-1|^3}{\sigma_i^3},
\end{align*}
where $X_{i,j}\in \mathbb{R}^d$ denotes the $j$th row of $X_i$ and $u\in \mathbb{R}^d$ satisfies $\|u\|_2=1$. Note that the distribution of $z_{ij}$ is independent of $u$ since $X_{i,j}$ is elliptical. We adopt a slightly different decomposition from \eqref{key:decomp:D}:
\begin{align}
\label{key:decomp:newD}
&\bar{\mathcal{D}}\left(\Sigma, \{\mathbb{P}_i\}_{i\in S} \right) -\bar{\mathcal{D}}(\hat{\Gamma}, \{\mathbb{P}_i\}_{i\in S} )  \nonumber \\
\leq & 2\sup_{\Gamma \succeq 0} \Big|\mathcal{D}(\Gamma, \{\mathbb{P}_i\}_{i\in S} )-\bar{\mathcal{D}}(\Gamma, \{\mathbb{P}_i\}_{i\in S} )\Big|+2\sup_{\Gamma \succeq 0} \Big|\mathcal{D}(\Gamma, \{X_i\}_{i\in S})-\mathcal{D}(\Gamma, \{\mathbb{P}_i\}_{i\in S} )\Big| \nonumber \\
&+\Big[\mathcal{D}(\Sigma, \{X_i\}_{i\in S} )-\mathcal{D}(\hat{\Gamma}, \{X_i\}_{i\in S} )\Big],
\end{align}
where 
\begin{align*}
&\bar{\mathcal{D}}\left(\Gamma, \{\mathbb{P}_i\}_{i\in S} \right)\\
=&\inf_{\|u\|_2=1} \min \Bigg\{
\frac{1}{|S|} \sum_{i\in S} \Phi\bigg(\sqrt{\frac{N_i}{\sigma_i^2}}\Big(\frac{u^T\Gamma u}{u^T\Sigma u}-1\Big)\bigg),1-\frac{1}{|S|} \sum_{i\in S} \Phi\bigg(\sqrt{\frac{N_i}{\sigma_i^2}}\Big(\frac{u^T\Gamma u}{u^T\Sigma u}-1\Big)\bigg) 
\Bigg\},
\end{align*}
with $\Phi(\cdot)$ being the CDF of $\mathcal{N}(0,1)$. The second and third terms in \eqref{key:decomp:newD} can be bounded using exactly the same arguments as before. But we need a new bound for the first one. Denote
\begin{align*}
&\vec{\mathcal{D}}\left(\Gamma, \{\mathbb{P}_i\}_{i\in S} \right)\\
=&\inf_{\|u\|_2=1} \min \Bigg\{
\frac{1}{|S|} \sum_{i\in S} \Phi\bigg(\sqrt{\frac{N_i}{\sigma_i^2}}\Big(\frac{u^T\Gamma u}{u^T\Sigma_i u}-1\Big)\bigg),1-\frac{1}{|S|} \sum_{i\in S} \Phi\bigg(\sqrt{\frac{N_i}{\sigma_i^2}}\Big(\frac{u^T\Gamma u}{u^T\Sigma_i u}-1\Big)\bigg) 
\Bigg\}.
\end{align*}
We have
\begin{align}
\label{ellip:first:bound}
&\sup_{\Gamma \succeq 0} \Big|\mathcal{D}(\Gamma, \{\mathbb{P}_i\}_{i\in S} )-\bar{\mathcal{D}}(\Gamma, \{\mathbb{P}_i\}_{i\in S} )\Big| \nonumber \\
\leq & \sup_{\Gamma \succeq 0} \Big|\mathcal{D}(\Gamma, \{\mathbb{P}_i\}_{i\in S} )-\vec{\mathcal{D}}(\Gamma, \{\mathbb{P}_i\}_{i\in S} )\Big| + \sup_{\Gamma \succeq 0} \Big|\vec{\mathcal{D}}(\Gamma, \{\mathbb{P}_i\}_{i\in S} )-\bar{\mathcal{D}}(\Gamma, \{\mathbb{P}_i\}_{i\in S} )\Big| \nonumber\\
\leq & \sup_{\|u\|_2=1,\Gamma \succeq 0} \frac{1}{|S|}\sum_{i\in S}\Bigg|\mathbb{P} \Big (\frac{1}{N_i}\sum_{j=1}^{N_i}z_{ij}^2\leq \frac{u^{T} \Gamma u}{u^T\Sigma_i u} \Big)-\Phi\bigg(\sqrt{\frac{N_i}{\sigma_i^2}}\Big(\frac{u^T\Gamma u}{u^T\Sigma_i u}-1\Big)\bigg)\Bigg| \nonumber\\
&+\sup_{\|u\|_2=1,\Gamma \succeq 0} \frac{1}{|S|}\sum_{i\in S}\Bigg|\Phi\bigg(\sqrt{\frac{N_i}{\sigma_i^2}}\Big(\frac{u^T\Gamma u}{u^T\Sigma_i u}-1\Big)\bigg)-\Phi\bigg(\sqrt{\frac{N_i}{\sigma_i^2}}\Big(\frac{u^T\Gamma u}{u^T\Sigma u}-1\Big)\bigg)\Bigg|\nonumber \\
\leq & c_7\max_{i\in S} \tau_i \sqrt{\frac{\bar{B}}{\underline{n}}}+c_8\Big(\sqrt{\frac{\bar{n}}{\underline{B}\min_{i\in S}\sigma_i^2}}+1\Big)\max_{i\in S}\frac{\|\Sigma_i-\Sigma\|}{\sigma_{min}(\Sigma_i)\wedge \sigma_{min}(\Sigma)}.
\end{align}
Here, in the last inequality we have used Berry-Essen Theorem (see details in the proof of Lemma \ref{lemma:median}) to bound the first sup and mean value theorem together with the simple fact $\sup_{x\geq 0}|t|xe^{-\frac{t^2}{2}(x-1)^2}\leq c_9(|t|+1)$ to bound the second sup. We then combine \eqref{key:decomp:newD} with \eqref{decom:part3:bound},\eqref{ellip:first:bound} and Lemma \ref{lemma:EPbound} to obtain that with probability at least $1-\exp(-t)$,
\begin{align}
&\bar{\mathcal{D}}\left(\Sigma, \{\mathbb{P}_i\}_{i\in S} \right) -\bar{\mathcal{D}}(\hat{\Gamma}, \{\mathbb{P}_i\}_{i\in S} ) \nonumber \\
\leq & \frac{\epsilon}{1-\epsilon}+c_{10}\sqrt{\frac{d+t}{|S|\underline{B}}}+c_7\max_{i\in S} \tau_i \sqrt{\frac{\bar{B}}{\underline{n}}}+c_8\Big(\sqrt{\frac{\bar{n}}{\underline{B}\min_{i\in S}\sigma_i^2}}+1\Big)\max_{i\in S}\frac{\|\Sigma_i-\Sigma\|}{\sigma_{min}(\Sigma_i)\wedge \sigma_{min}(\Sigma)}. \label{final:ellip:upper}
\end{align}
On the other hand, since $\Sigma=\arg\min_{\Gamma \succeq 0}\bar{\mathcal{D}}(\Gamma, \{\mathbb{P}_i\}_{i\in S} )$, we can write $\forall u {\rm~with~}\|u\|_2=1$, 
\begin{align}
&\bar{\mathcal{D}}\left(\Sigma, \{\mathbb{P}_i\}_{i\in S} \right) -\bar{\mathcal{D}}(\hat{\Gamma}, \{\mathbb{P}_i\}_{i\in S} )\nonumber \\
\geq & \bigg|\frac{1}{|S|}\sum_{i\in S}\bigg(\Phi\bigg(\sqrt{\frac{N_i}{\sigma_i^2}}\Big(\frac{u^T\hat{\Gamma} u}{u^T\Sigma u}-1\Big)\bigg)-\Phi(0)\bigg)\bigg|\geq \sqrt{\frac{\underline{n}}{\bar{B}\max_{i\in S}\sigma^2_i}}\cdot \psi(t^*)\Big|\frac{u^T\hat{\Gamma} u}{u^T\Sigma u}-1\Big|, \label{final:ellip:lower}
\end{align}
where the last inequality is by mean value theorem and $\psi(\cdot) $ denotes the density of $\mathcal{N}(0,1)$. Under the conditions $\epsilon\leq \frac{1}{4}$, $\frac{\underline{n}}{\bar{B}}\geq C_1$ and \eqref{con:small:gap}, the upper bound in \eqref{final:ellip:upper} is no larger than $\frac{5}{12}$, hence $t^*$ in \eqref{final:ellip:lower} is bounded by an absolute constant. This fact together with \eqref{final:ellip:upper}-\eqref{final:ellip:lower} completes the proof for Part (ii).
\end{proof}

\subsubsection{Technical lemmas}

\begin{lemma}\label{lemma:EPbound}
For any given $t>0$, the following uniform bound holds with probability at least $1-\exp(-t)$,
\begin{align*}
\sup_{\Gamma\succeq 0}|\mathcal{D}\left(\Gamma, \{X_i\}_{i\in S} \right)-\mathcal{D}\left(\Gamma, \{\mathbb{P}_i\}_{i\in S} \right)| \leq c\sqrt{\frac{d+t}{|S|\underline{B}}},
\end{align*}
where $c>0$ is an absolute constant.

\end{lemma}
\begin{proof}
First, it is straightforward to confirm that
\begin{align*}
 &\sup_{\Gamma\succeq 0}|\mathcal{D}\left(\Gamma, \{X_i\}_{i\in S} \right)-\mathcal{D}\left(\Gamma, \{\mathbb{P}_i\}_{i\in S} \right)| \\
 \leq &  \sup_{\|u\|_2=1,t\in\mathbb{R}}\bigg|\sum_{i\in S}\sum_{j=1}^{B_i}\frac{1}{|S|B_i}\lb\mathbb{I}\Big\{u^T\hat{\Sigma}_i^{(j)}u\leq t\Big\}-\Pb(u^T\hat{\Sigma}_i^{(j)}u\leq t)\rb\bigg|:=\Delta,
\end{align*}
where $\hat{\Sigma}_i^{(j)}=\frac{(X_i^{(j)})^TX_i^{(j)}}{n_i}$. We now focus on bounding the empirical process on the right-hand side of the above inequality. Note that the empirical process consists of independent (but non-identically distributed) random variables $\hat{\Sigma}_i^{(j)}$'s. We follow standard arguments to bound $\Delta-\mathbb{E}\Delta$ and $\mathbb{E}\Delta$ respectively. By the bounded difference inequality (e.g. Theorem 6.2 in \cite{bouch13non}), we immediately obtain
\begin{align}
\label{bdi:give}
\mathbb{P}\Bigg(\Delta\geq \mathbb{E}(\Delta)+\sqrt{\frac{t}{2|S|\underline{B}}}\Bigg)\leq e^{-t}, ~~\forall t>0.
\end{align}
To bound $\mathbb{E}\Delta$, a standard symmetrization gives
\begin{align}
\mathbb{E}\Delta\leq 2\mathbb{E}\sup_{\|u\|_2=1,t\in\mathbb{R}}\bigg|\sum_{i\in S}\sum_{j=1}^{B_i}\frac{\epsilon_{ij}}{|S|B_i}\mathbb{I}\Big\{u^T\hat{\Sigma}_i^{(j)}u\leq t\Big\}\bigg|\leq 2\mathbb{E}\sup_{f\in \mathcal{F}}\bigg|\underbrace{\sum_{i\in S}\sum_{j=1}^{B_i}\frac{\epsilon_{ij}}{|S|B_i}f(\hat{\Sigma}_i^{(j)})}_{Z_f}\bigg|, \label{sym:red}
\end{align}
where $\epsilon_{ij}$'s are independent symmetric Bernoulli variables, and $\mathcal{F}$ is the boolean function class $\mathcal{F}=\big\{M\in \mathbb{M}_d\rightarrow \mathbb{I}\{u^TMu\leq t \},\|u\|_2=1,t\in \mathbb{R}\big\}$ with $\mathbb{M}_d$ being the space of symmetric $d$-by-$d$ matrices. We now condition on $\{\hat{\Sigma}_i^{(j)}\}$ and aim to use Dudley inequality for $\mathbb{E}\sup_{f\in \mathcal{F}}|Z_f|$. To this end, we first show that the increments of the stochastic process $(Z_f)_{f\in \mathcal{F}}$ are subgaussian: $\forall f,g\in \mathcal{F}$,
\[
\|Z_f-Z_g\|_{\psi_2}\leq c_1 \frac{1}{\sqrt{|S|\underline{B}}}\cdot \underbrace{\sqrt{\frac{\sum_{i\in S}\sum_{j=1}^{B_i}\frac{1}{|S|^2B_i^2}(f(\hat{\Sigma}_i^{(j)})-g(\hat{\Sigma}_i^{(j)}))^2}{\sum_{i\in S}\sum_{j=1}^{B_i}\frac{1}{|S|^2B_i^2}}}}_{=\|f-g\|_{L^2(\mu)}},
\]
where $\|f-g\|_{L^2(\mu)}$ is the $L_2$ norm based on a weighted empirical measure, and $c_1>0$ is an absolute constant. Hence, we can apply Dudley inequality (e.g., Theorem 8.1.3 in \cite{vershynin2018high}) to obtain
\begin{align}
\mathbb{E}\sup_{f\in \mathcal{F}}|Z_f|&\leq \frac{c_2}{\sqrt{|S|\underline{B}}} \mathbb{E}\int_0^1 \sqrt{\log \mathcal{N}(\mathcal{F},L^2(\mu), t)}dt \nonumber \\
&\leq \frac{c_2}{\sqrt{|S|\underline{B}}} \mathbb{E}\int_0^1 \sqrt{c_3 {\rm vc}(\mathcal{F})\log(c_3/t)}dt \leq c_4\sqrt{\frac{d}{|S|\underline{B}}}, \label{dudley:bound}
\end{align}
where the second inequality bounds the metric
entropy $\log \mathcal{N}(\mathcal{F},L^2(\mu), t)$ via VC dimension ${\rm vc}(\mathcal{F})$ \cite{dudley78cover}, and the last inequality holds since ${\rm vc}(\mathcal{F})\leq c_5 d$ according to Lemma 1 and Proposition 1 in \cite{depersin2024robust}. Here, $c_i,i=2,\dots, 5$ are absolute constants. Putting together \eqref{bdi:give}-\eqref{dudley:bound} completes the proof. 
\end{proof}

\begin{lemma}\label{lemma:median}
Let $\{g_{ij}\}_{1\leq i\leq T, 1\leq j\leq N_i}$ be i.i.d. random variables with $\mathbb{E}(g_{ij}^2)=1$ and ${\rm Var}(g_{ij}^2)=\sigma^2<\infty$. Assume 
\[
C_g:=\frac{\mathbb{E}|g_{ij}^2-1|^3}{\sigma^3}<\infty,
\]
and $\frac{C_g}{\sqrt{\min_{i}N_i}}\leq c$. If there exits $\beta$ such that
\begin{align}
\label{eq:beta:type}
\frac{1}{T}\sum_{i=1}^T\mathbb{P}\Big(\frac{1}{N_i}\sum_{j=1}^{N_i}g_{ij}^2\leq \beta\Big)=\frac{1}{2},
\end{align}
then $\beta$ satisfies 
\begin{align}
\label{beta:concen:bound}
|\beta-1|\leq \frac{\tilde{c} C_g\sigma }{\min_{i}N_i}.
\end{align}
Here, $c,\tilde{c}>0$ are absolute constants.
\end{lemma}
\begin{proof}
Throughout the proof, we use $c_i$ to denote absolute constants. We first write
\begin{equation}
\label{reform:bethm}
    \Pb\lb\frac{1}{N_i}\sum_{j=1}^{N_i} g_{ij}^2\leq \beta\rb=\Pb\lb \frac{\sum_{j=1}^{N_i} (g_{ij}^2-1)}{\sqrt{\sigma^2 N_i}}\leq \sqrt\frac{N_{i}}{\sigma^2}(\beta-1)\rb.
\end{equation}
By Berry-Esseen Theorem (e.g. Theorem 3.4.9 in \cite{durrett2010probability}), we have
\begin{align}
\label{berry:esseen:exp}
\sup_{t\in \mathbb{R}}\left| \Pb\lb \frac{\sum_{j=1}^{N_i} (g_{ij}^2-1)}{\sqrt{\sigma^2 N_i}}\leq t\rb-\Phi\lb t\rb \right|\leq  \frac{c_1C_g}{\sqrt{N_i}},
\end{align}
where $\Phi(\cdot)$ is the cumulative distribution function of $\mathcal{N}(0,1)$. Suppose $\beta\geq 1$, then combining \eqref{eq:beta:type} and \eqref{reform:bethm}-\eqref{berry:esseen:exp} gives
\begin{align*}
\frac{c_1C_g}{\sqrt{\min_i N_i}}\geq  \frac{1}{T}\sum_{i=1}^T \frac{c_1C_g}{\sqrt{N_i}} &\geq \frac{1}{T}\sum_{i=1}^T\Phi\lb \sqrt\frac{N_{i}}{\sigma^2}(\beta-1)\rb-\frac{1}{2}\\
&\geq \Phi\lb \sqrt\frac{\min_{i}N_{i}}{\sigma^2}(\beta-1)\rb-\Phi(0)=\psi(t^*)\cdot \sqrt\frac{\min_{i}N_{i}}{\sigma^2}(\beta-1).
\end{align*}
Here, the last equality is by mean value theorem and $\psi(\cdot)$ denotes the density function of $\mathcal{N}(0,1)$. As long as $\frac{c_1C_g}{\sqrt{\min_i N_i}}\leq \frac{1}{4}$, the value of $t^*$ will be bounded by an absolute constant so that we can conclude 
\[
0\leq \beta-1\leq \frac{c_2C_g\sigma}{\min_i N_i}.
\]
If $\beta<1$, similar arguments show
\[
\frac{c_1C_g}{\sqrt{\min_i N_i}}\geq \Phi(0)-\Phi\lb \sqrt\frac{\min_{i}N_{i}}{\sigma^2}(\beta-1)\rb=\psi(t^{**})\cdot \sqrt\frac{\min_{i}N_{i}}{\sigma^2}(1-\beta),
\]
yielding $0\leq 1-\beta\leq \frac{c_3C_g\sigma}{\min_i N_i}$. In both cases, we have proved the result \eqref{beta:concen:bound}.
\end{proof}

\begin{lemma}\label{Lemma:lbd for derivative}
Let $f_n(x)$ be the probability density  function of $X/n$ where $X\sim \chi^2(n)$. The following hold:
\begin{itemize}
\item[(i)] $\sup_{x\geq 0}xf_n(x)\leq c\sqrt{n}$ for some absolute constant $c>0$.
\item[(ii)] Let $\kappa>0$ be some fixed constant. As long as $\frac{\kappa}{\sqrt{n}}\leq \frac{1}{2}$, we have
\begin{equation*}
    f_n(x) \geq c_{\kappa}\sqrt{n}, ~~\forall x\in \big[1-\kappa/\sqrt{n},1+\kappa/\sqrt{n} \big],
\end{equation*}
where $c_{\kappa}>0$ is a constant that only depends on $\kappa$ and decreases with $\kappa$.
\end{itemize}
\end{lemma}

\begin{proof}
The density function $f_n(x)$ takes the form
 \begin{equation*}
f_n(x)
= \frac{n^{\frac n2}}{2^{\frac n2}\,\Gamma\!\bigl(\frac n2\bigr)}
\,x^{\frac n2 - 1}
\exp\lb-\frac{n x}{2}\rb,
~  x>0.
 \end{equation*}
 By Stirling‘s formula, 
 \[
 \sqrt{\frac{2\pi}{z}}\Big(\frac{z}{e}\Big)^z\leq \Gamma(z)\leq \sqrt{\frac{2\pi}{z}}\Big(\frac{z}{e}\Big)^ze^{\frac{1}{12z}}, ~~\forall z>0.
 \]
We thus obtain
 \begin{equation}
c_1\Big(\frac{n}{2e}\Big)^{\frac{n}{2}}n^{-\frac{1}{2}}\leq  \Gamma\Big(\frac{n}{2}\Big) \leq c_2\Big(\frac{n}{2e}\Big)^{\frac{n}{2}}n^{-\frac{1}{2}}, \label{gamma:sim}
 \end{equation}
 where $c_1,c_2>0$ are absolute constants. This implies
 \begin{align*}
xf_n(x)\leq c_1^{-1}\sqrt{n}\exp\lb\frac{n}{2}\lb 1-x+\log (x)\rb\rb\leq c_1^{-1}\sqrt{n}, ~~\forall x>0,
 \end{align*}
 where the second inequality is due to the fact that $1-x+\log x\leq 0$ for all $x>0$. This finishes the proof of Part (i).

Regarding Part (ii), from \eqref{gamma:sim} we have
\begin{align}
f_n(x)&\geq c_2^{-1}\sqrt{n}x^{-1}\exp\lb\frac{n}{2}\lb 1-x+\log (x)\rb\rb \nonumber \\
&\geq \frac{2}{3}c_2^{-1}\sqrt{n}\exp\lb\frac{n}{2}\lb 1-x+\log (x)\rb\rb, \quad \forall x\in \big[1-\kappa/\sqrt{n},1+\kappa/\sqrt{n} \big], \label{pdf:lower:eq1}
\end{align}
under the condition $\frac{\kappa}{\sqrt{n}}\leq \frac{1}{2}$. Using Taylor’s Theorem, it is direct to verify
\begin{align}
\label{pdf:lower:eq2}
\big|\log(1+t)-t+\frac{1}{2}t^2\big|\leq \frac{8}{3}|t|^3,\quad \forall t\geq -\frac{1}{2}.
\end{align}
Since $x\geq 1/2$ for each $x\in \big[1-\kappa/\sqrt{n},1+\kappa/\sqrt{n} \big]$, we can use \eqref{pdf:lower:eq2} to approximate $\log x$ in \eqref{pdf:lower:eq1}, obtaining $\forall x\in \big[1-\kappa/\sqrt{n},1+\kappa/\sqrt{n} \big]$,
\begin{align*}
f_n(x)&\geq \frac{2}{3}c_2^{-1}\sqrt{n}\exp \Big(\frac{n}{2}\Big(-\frac{1}{2}(x-1)^2-\frac{8}{3}|x-1|^3\Big)\Big) \\
&\geq \frac{2}{3}c_2^{-1}\sqrt{n}\exp \Big(-\frac{\kappa^2}{4}-\frac{4\kappa^3}{3}\Big).
\end{align*}
The proof of Part (ii) is completed.

\end{proof}

\subsection{Completion of the proof}

Under Assumption \ref{assumption:data simi}, we first prove the following deterministic and non-asymptotic result: set $\lambda_i=\lambda\geq 2\max_{i\in S}\|\hat{V}_i\hat{V}_i^T-V_iV_i^T\|_F$, then it holds that 
\begin{align}
\label{general:nonasymp}
\max_{i\in S}\|\tilde{V}_i\tilde{V}_i^T-V_iV_i^T\|_F\leq 6\sqrt{2}\cdot \Big[\big(\|\tilde{A}-VV^T\|_F+\delta\big)\wedge \frac{\lambda}{2}\Big].
\end{align}
When $\|\tilde{A}-VV^T\|_F+\delta\leq \frac{\lambda}{2}$, we know that 
\[
\|\tilde{A}-\hat{A}_i\|_F\leq \|\tilde{A}-VV^T\|_F+\|VV^T-V_iV_i^T\|_F+\|V_iV_i^T-\hat{V}_i\hat{V}_i^T\|_F\leq \lambda, ~\forall i\in S.
\]
Hence, for each $i\in S$, $\tilde{A}_i=\tilde{A}$ and
\begin{align}
\|\tilde{V}_i\tilde{V}_i^T-V_iV_i^T\|_F\leq & 2\sqrt{2}\|\tilde{A}-V_iV_i^T\|_F\leq 2\sqrt{2}(\|\tilde{A}-VV^T\|_F+\|V_iV_i^T-VV^T\|_F) \nonumber \\
\leq& 2\sqrt{2}(\|\tilde{A}-VV^T\|_F+\delta), \label{case1:nonasymp}
\end{align}
where the first inequality is by Lemma \ref{lemma: DK}. On the other hand, it is always true that $\|\tilde{A}_i-\hat{A}_i\|_F\leq \lambda$, which implies
\begin{align}
\label{single:task:side}
&\|\tilde{V}_i\tilde{V}_i^T-V_iV_i^T\|_F\leq 2\sqrt{2}\|\tilde{A}_i-V_iV_i^T\|_F \nonumber \\
\leq &2\sqrt{2}(\|\tilde{A}_i-\hat{A}_i\|_F+\|\hat{A}_i-V_iV_i^T\|_F)\leq 3\sqrt{2}\lambda, \quad \forall i\in S.
\end{align}
Putting together \eqref{case1:nonasymp}-\eqref{single:task:side} gives \eqref{general:nonasymp}. The proof is then completed by combining \eqref{general:nonasymp} with Theorem \ref{theorem:depth}, Lemma \ref{lemma: DK} and \eqref{single:learning:rate:union}.

\bibliographystyle{alpha}
\bibliography{refs}
	
\end{document}